\documentclass[a4paper]{amsart}
\usepackage{graphicx, url, color, xcolor, amssymb, amsmath,enumitem,multirow,appendix, bbm, soul} 
\usepackage{hyperref}
\usepackage[new]{old-arrows}

\theoremstyle{plain}
\newtheorem{thm}{Theorem}[section]
\newtheorem*{thm*}{Theorem}
\newtheorem{lem}[thm]{Lemma}

\newtheorem{prop}[thm]{Proposition}
\newtheorem{conj}[thm]{Conjecture}

\theoremstyle{definition}

\newtheorem{example}[thm]{Example}
\newtheorem{rem}[thm]{Remark}

\title{Hecke modules associated to certain ternary lattices of half-discriminant $p^2$ and $4p^2$.}
\author{Giovanna De Lauri, Dan Fretwell}
\date{}

\begin{document}

\begin{abstract}
We consider Hecke modules of orthogonal modular forms attached to definite ternary lattices of half-discriminant $p^2$ with $p$ prime and $4p^2$ with $p$ an odd prime. Where possible, we prove explicit Jacquet-Langlands style correspondences with spaces of classical modular forms (in the spirit of Birch \cite{birch1991hecke} and Hein-Tornaria-Voight \cite{heintornariavoight}), and give conjectures for the remaining genera. This is done by applying sequences of Watson-style transformations that reduce the discriminant while preserving the adelic orthogonal group.
\end{abstract}

\maketitle

\section{Introduction}

The classification of lattices up to isometry is a long-standing open problem with a rich history. While the rational quadratic spaces in which they are embedded can be completely classified via their localisations by the Hasse-Minkowski Theorem, lattices are known to resist complete classification by purely local methods. The failure of the local-global principle for a particular lattice $\Lambda$ in a rational quadratic space $(V,Q)$ is typically measured by considering its class set $\text{Cl}(\Lambda) = O(V)\backslash \text{gen}(\Lambda)$, the finite set of isometry classes of lattices that are locally isometric to $\Lambda$ everywhere. When $\text{Cl}(\Lambda) = \{[\Lambda]\}$ then the local-global principle holds for $\Lambda$, otherwise it doesn't.

Seminal work of Kneser \cite{Kneser1957} gives a very explicit algorithm for computing $\text{Cl}(\Lambda)$, via the notion of $p$-neighbour lattice (also discussed in \cite{gerstein2008basic},\cite{voight2023kneser} and \cite{schulze2020lecture}). The $p$-neighbours of $\Lambda$ are lattices $\Pi$ in the same rational quadratic space that have the same scale as $\Lambda$ and satisfy $\Lambda/(\Lambda\cap \Pi) \cong \Pi/(\Lambda\cap \Pi) \cong \mathbb{Z}/p\mathbb{Z}$. If $\text{rank}(\Lambda)\geq 3$ and $p\nmid 2\,\text{disc}(\Lambda)$ then the $p$-neighbours of $\Lambda$ are explicitly computable and can be shown to lie in $\text{gen}(\Lambda)$ (see \cite{greenbergvoight} for a general statement). This provides a simple way of navigating the genus, by tweaking the local structure of $\Lambda$ at one prime $p$ in a way that can potentially lead to new global isometry classes. A highly satisfying, non-trivial result of Kneser \cite{Kneser1957} is that every isometry class in $\text{Cl}(\Lambda)$ can be reached by a sequence of $p_i$-neighbours for some finite sequence of primes $p_i$. Quite often it suffices to use only a single prime in Kneser's result (e.g.\ if there is only one spinor genus in $\text{gen}(\Lambda)$).

The method of $p$-neighbours has been used extensively to classify isometry classes of lattices in certain genera. For example, in $24$ dimensions it was used by Niemeier in \cite{niemeier1973definite} to classify the $24$ isometry classes of even unimodular lattices and it was used by Borcherds in \cite{borcherds1999leech} to classify the $273$ isometry classes of odd unimodular lattices.

While the $p$-neighbour method is well understood and is highly successful for classifying isometry classes of lattices of low rank and discriminant, other natural questions remain. For example, how do the $p$-neighbours of a lattice $\Pi\in\text{gen}(\Lambda)$ distribute among the isometry classes in $\text{Cl}(\Lambda)$? Alternatively, how are the entries of the adjacency matrices $T_p$ of the directed regular $p$-neighbour graphs $G_p$ on $\text{Cl}(\Lambda)$ distributed? (The directed graphs with vertices corresponding to $\text{Cl}(\Lambda) = \{[\Lambda_1],[\Lambda_2],...,[\Lambda_h]\}$ and directed edges from $[\Lambda_i]$ to $[\Lambda_j]$ for each $p$-neighbour of $\Lambda_i$ that is isometric to $\Lambda_j$). 

Quite often $G_p$ is connected (by Kneser's result, when $p$-neighbours with respect to the prime $p$ suffice to exhaust the class set) and recent work of Chenevier \cite{chenevier2021statistics} proves precise distribution results on the entries of $T_p$ as $p$ gets large. Naturally, one wonders whether more can be said about the structure of the graphs $G_p$, for example via the spectral properties of $T_p$. Birch was the first to explore such questions for $\Lambda$ a rank $3$ lattice in a positive definite rational quadratic space (the most interesting case since the case of indefinite lattices is well understood). In particular, he was able to make a number of startling observations about the relationship between the non-trivial eigenvalues of the matrices $T_p$ and Hecke eigenvalues of classical modular forms of weight $2$ (see \cite{birch1991hecke}, also discussed in \cite{hein2016orthogonal} and \cite{heintornariavoight}). It is worth noting that this connection with modular forms immediately proves that the graphs $G_p$ provide families of Ramanujan graphs, i.e.\ optimal expander graphs. Such graphs have become highly useful across Computer Science, and recently have become useful in Cryptography (e.g. isogeny graphs coming from elliptic curves).

An alternative way to see $p$-neighbour data is by instead dualising and looking at the space of functions: \[\mathcal{M} = \{f:\text{Cl}(\Lambda) \rightarrow\mathbb{C}\}\cong\mathbb{C}^h.\] The space $M$ comes equipped with an action of $p$-neighbour operators $T_p$, which average the values of $f$ over all $p$-neighbours of the input lattice. These operators are then the adjacency matrices considered above and so have the same eigenvalues. Such a space of functions is a special example of a space of (definite) orthogonal modular forms, which in turn is a special example of a space of algebraic modular forms (due to Gross \cite{gross1999algebraic}). Under this guise, the $p$-neighbour operators then become Hecke operators acting on these spaces of automorphic forms. 

Once the above translation is made, one can immediately generalise to spaces of functions on lattices with other symmetries: \[\mathcal{M}(O(V),O(\widehat{\Lambda}), W) = \{f: \text{gen}(\Lambda) \rightarrow W\,|\, f(\gamma \Pi) = \gamma\cdot f(\Pi) \text{ for all } \gamma\in O(V)\}\] where $W$ is a finite dimensional representation of $O(V)$, called the weight. The space $\mathcal{M}$ is then the space $\mathcal{M}(O(V),O(\widehat{\Lambda}),\mathbb{C})$, i.e.\ with the trivial representation $W = \mathbb{C}$ as the weight (hence giving $O(V)$-invariant functions). As before, there are Hecke operators $T_p$ acting on these spaces, via averaging over $p$-neighbours. When viewed in this way, the insights of Birch are immediately interpreted as providing explicit transfer results between Hecke-modules of orthogonal modular forms and classical modular forms, in the spirit of the Jacquet-Langlands correspondence (a special case of Langlands Functoriality).

More recently, there have been a number of papers both extending the work of Birch and proving his claims. Some of these have already been referenced above. The recent paper of Hein-Tornaria-Voight \cite{heintornariavoight} proves a Hecke-equivariant transfer result for a wide range of lattices with non-square half-discriminant. Roughly speaking, if $\Lambda\subset V$ is an even lattice that corresponds via the even Clifford functor to a residually unramified order $\mathcal{O}\subset B$ in a definite quaternion algebra $B$, then they are able to prove a Hecke-equivariant isomorphism between spaces of orthogonal modular forms attached to $\Lambda$ and spaces of classical modular forms. The strength of these results is twofold: 

\begin{itemize}
\item{they treat spaces of orthogonal modular forms of varying weight, showing that other classical modular forms can be computed this way (e.g.\ not just weight $2$),}
\item{the results are entirely explicit, giving a precise and complete description of the image of the transfer in terms of weight, level, newness and Atkin-Lehner eigenvalues (at the bad primes).}
\end{itemize}

It is worth mentioning that other papers have recently been published, in an attempt to further generalise these transfer results to lattices of other ranks. For example: 

\begin{itemize}
\item{the recent thesis of Wu \cite{wu} studies the binary case and relates to Hecke characters,}
\item{an upcoming paper of Assaf-Fretwell-Ingalls-Logan-Secord-Voight studies the quaternary case and relates to Hilbert modular forms,}
\item{the paper of Dummigan-Pacetti-Rama-Tornaria \cite{dummiganpacettiramatornaria} studies the quinary case and relates to Siegel modular forms of paramodular type.}
\end{itemize}  

It is an interesting question to ask how much further these techniques can be pushed for higher ranks. To our knowledge, not much has been written down concretely. For example, there are partial results in \cite{assaffretwellingallslogansecordvoight}.

In this paper we continue to explore Hecke-modules of orthogonal modular forms attached to integral ternary lattices in positive definite rational quadratic spaces, but focus our efforts on the square half-discriminant case. In particular, we consider the simplest cases of half-discriminant $p^2$ for $p$ prime (of which there are five genera, one even one and four odd ones) and half-discriminant $4p^2$ for odd prime $p$ (of which there are twenty-two genera, two even ones and twenty odd ones). In each case  we are either able to prove or conjecture an explicit Hecke-equivariant transfer with a Hecke-module of classical modular forms (in the style of the Jacquet-Langlands correspondence). In each case we try to describe the image of the transfer as precisely as possible. The following theorem gives an example of such a result for one of the genera under consideration (see the Appendix for a full summary table of our results and Section $2$ for notation).

\begin{thm}
Let $\Lambda$ be an even lattice of discriminant $2p^2$ in a positive definite, ternary quadratic space $(V,Q)$ over $\mathbb{Q}$. Then for each $k\geq 0$ and $d\in\{1,p\}$ there exists a Hecke-equivariant isomorphism (for the algebra generated by the Hecke operators $T_q$ with $q\neq p$ prime): \[\mathcal{S}(O(V), O(\widehat{\Lambda}),W_{2k}\otimes \theta_d) \cong S_{2k+2}^{\textup{new}}\left(\Gamma_0(p),\varepsilon^{(p)}_p\varepsilon^{(p)}_d\right).\]
\end{thm}

Such theorems cannot immediately be proved using the results of Hein-Tornaria-Voight since the associated quaternion orders are residually ramified. Our general strategy in proving such a result (when we are able to) is to instead perform certain Watson-style transformations to $\Lambda$, producing a new even lattice $\Pi$ of discriminant $2p$ while preserving the adelic orthogonal group (see \cite{watson1953some}, \cite{watson1963class}). This is enough to prove a bijection between $\text{Cl}(\Lambda)\rightarrow\text{Cl}(\Pi)$ that respects the $p$-neighbour relation, hence providing a Hecke-equivariant isomorphism between the associated Hecke-modules of orthogonal modular forms. One can then appeal to the above-mentioned result of Hein-Tornaria-Voight to get the required Hecke-equivariant transfer to classical spaces of modular forms. 

In terms of spaces of orthogonal modular forms, one can view the application of such transformations as lowering the level of the forms, hence also shedding light on the newform theory for spaces of orthogonal modular forms. In particular, the theorem above can be interpreted as saying that various spaces of orthogonal modular forms attached to $\Lambda$ consist entirely of oldforms (since the RHS also corresponds to a certain space of orthogonal modular forms associated to the lattice $\Pi$ having lower discriminant $2p$).

\section*{Acknowledgements}
This work forms part of the PhD thesis of the first named author. We thank Jessica Jay for helpful suggestions and comments regarding the general exposition of the paper and John Voight for useful advice concerning residually ramified orders in definite quaternion algebras. This paper and the work contained within was written completely without AI assistance.

\section{Setup and Notation}

\subsection{Quadratic Spaces and Lattices}

Let $(V,Q)$ be a quadratic space over a field $F$ with $\text{char}(F)\neq 2$, i.e.\ an $F$-vector space $V$ equipped with a map $Q: V \rightarrow F$ such that: \begin{itemize}
\item{$Q(\alpha v) = \alpha^2 Q(v)$ for all $\alpha\in F$ and $v\in V$,}
\item{the map $T(v,w) := Q(v+w) - Q(v) - Q(w)$ is bilinear.}
\end{itemize} Note that as a consequence of our definition we relate the quadratic map and the symmetric bilinear form via $T(v,v) = 2Q(v)$. We will also assume from now on that $n = \text{dim}_{\mathbb{Q}}(V) < \infty$.

Given any basis $\{e_1,e_2,...,e_n\}$ for $V$ we define the corresponding Gram matrix $T = (T(e_i,e_j))_{1\leq i, j\leq n}$ and the discrimimant $\mathrm{disc}(V) = [\text{det}(T)]\in F/F^{\times 2}$. We say that $(V,Q)$ is non-degenerate if $\mathrm{disc}(V) \neq 0$.

An isometry between quadratic spaces $(V_1,Q_1)$ and $(V_2,Q_2)$ over $F$ is an isomorphism $\phi: V_1\rightarrow V_2$ such that $Q_2(\phi(v)) = Q_1(v)$ for all $v\in V_1$. If an isometry exists, we will write $(V_1,Q_1)\sim (V_2,Q_2)$. We will also write $(V,Q)\sim \langle A_1,A_2,...,A_m\rangle$ to signify that $(V,Q)$ is isometric to a quadratic space whose Gram matrix under some basis is block diagonal with blocks $A_i$ (with lower case $a_i$ being reserved for the special case of a $1\times 1$ block). The dimension $n$ and the discriminant $\mathrm{disc}(V)$ are isometry invariants of the quadratic space.

It is very well known that every quadratic space $(V,Q)$ over $F$ can be diagonalised, so that $(V,Q)\sim \langle a_1,a_2,...,a_n\rangle$ for some $a_i\in F$. The discriminant is then given by the product $\mathrm{disc}(V) = [a_1a_2...a_n]$.

The set of self-isometries of a quadratic space $(V,Q)$ over $F$ gives the orthogonal group of $V$: \[O(V) = \{\phi\in\text{GL}(V)\,|\, Q(\phi(v)) = Q(v)\}.\] When $(V,Q)$ is non-degenerate every element of $O(V)$ has determinant $\pm 1$ and the kernel of the determinant map is the special orthogonal group $SO(V)$. This gives an isomorphism $O(V)/SO(V) \cong \{\pm 1\}$ and when $n$ is odd the non-trivial coset is represented by $-I$.

If $(V,Q)$ is a quadratic space over $F$ and $\lambda\in F$ then we may scale the quadratic map by $\lambda$ to get the quadratic space $(V, Q^{(\lambda)})$ with quadratic map $Q^{(\lambda)}:  V \to F$ given by $Q^{(\lambda)}(v) := \lambda Q(v)$ for all $v\in V$. The associated symmetric bilinear form is then $T^{(\lambda)}(v,w) = \lambda T(v,w)$. If $T$ is a Gram matrix for $(V,Q)$ under some basis then the Gram matrix of $(V,Q^{(\lambda)})$ under the same basis is $\lambda T$. Note that this notion of scaling is not the same as scaling a basis of $V$ by $\lambda$, which would produce a quadratic space with Gram matrix $\lambda^2 T$.

From now on let $(V,Q)$ be a non-degenerate quadratic space over $\mathbb{Q}$ of dimension $n$. Attached to $V$ are various completed quadratic spaces $(V_{\infty},Q)$ and $(V_p,Q)$ for $p$ prime (these being quadratic spaces over $\mathbb{R}$ and $\mathbb{Q}_p$ respectively). The underlying vector spaces are: \begin{align*}V_{\infty} &= V\otimes_{\mathbb{Q}}\mathbb{R}\\ V_p &= V\otimes_{\mathbb{Q}}\mathbb{Q}_p \end{align*} and the quadratic map $Q$ naturally extends to these spaces via $Q(v\otimes \alpha) = \alpha^2Q(v)$. 

The notion of isometry extends naturally to local quadratic spaces, and the celebrated Hasse-Minkowski theorem (Chapter $6$ of \cite{cassels2008rational}) tells us that the global isometry class of $(V,Q)$ is determined by the local isometry classes of $(V_{\infty},Q)$ and $(V_p,Q)$. The classification of local quadratic spaces is known, and consequently the global isometry class of $(V,Q)$ is fully determined by the following list of invariants:

\begin{itemize}
\item{The discriminant $\mathrm{disc}(V)\in\mathbb{Q}^{\times}/\mathbb{Q}^{\times 2}$,}
\item{The signature $(r,s)$ of $(V_{\infty},Q)$, the unique integers $r,s\geq 0$ such that: \[(V_{\infty},Q) \sim \langle 1,1,...,1,-1,-1,...,-1\rangle,\] with $r$ occurrences of $1$ and $s$ occurrences of $-1$ (noting that $n = r+s$).}
\item{The Hasse invariant $S(V_p)\in\{\pm 1\}$ of $(V_p,Q)$, given by: \[S(V_p) = \prod_{1\leq i<j\leq n}(a_i,a_j)_p,\] where \[(V_p,Q) \sim \langle a_1,a_2,...,a_n\rangle\] is any diagonalisation and $(a,b)_p$ is the Hilbert symbol (when $n = 1$ we declare $S(V_p) = 1$).}
\end{itemize} The local Hasse invariants satisfy certain compatibility conditions (the second of which follows from Hilbert Reciprocity): 

\begin{itemize}
\item{$S(V_p) = 1$ for all but finitely many $p$}
\item{$\prod_{p\leq \infty}S(V_p) = 1$ (where $S(V_{\infty}) = (-1)^{\frac{s(s-1)}{2}}$).}
\end{itemize}

Conversely, suppose these invariants are chosen freely to satisfy the compatibility conditions. Then there exists a rational quadratic space giving rise to these invariants, with only two exceptions: 

\begin{itemize}
\item{$n=1$ and $S(V_p) = -1$ for some finite prime $p$,}
\item{$n=2$, $\mathrm{disc}(V_p) = [-1]$ and $S(V_p) = -1$ for some finite prime $p$.}
\end{itemize}

The completed spaces come equipped with local orthogonal groups: \begin{align*}O(V_{\infty}) &= \{\phi\in\text{GL}(V_{\infty})\,|\, Q(\phi(v)) = Q(v)\}\\ O(V_p) &= \{\phi\in\text{GL}(V_p)\,|\, Q(\phi(v)) = Q(v)\}\end{align*}

If $V$ is positive definite, i.e.\ it has signature $(r,s) = (n,0)$, then $O(V_{\infty}) \cong O_n(\mathbb{R})$ is the standard compact orthogonal group.

A lattice $\Lambda\subset V$ is a free $\mathbb{Z}$-module of rank $n$. Every basis of $\Lambda$ gives rise to a basis for $V$ and a Gram matrix $T$. The discriminant of $\Lambda$ is then $\mathrm{disc}(\Lambda) = [\text{det}(T)]\in\mathbb{Q}^{\times}/(\mathbb{Z}^{\times})^2 = \mathbb{Q}^{\times}$.

If $\Lambda_1\subset V_1$ and $\Lambda_2\subset V_2$ are lattices then an isometry is an isometry $\phi: V_1 \rightarrow V_2$ such that $\phi(\Lambda_1) = \Lambda_2$. The rank $n$ and discriminant $\mathrm{disc}(\Lambda)$ are both isometry invariants of a lattice. The self-isometries of $\Lambda$ give the orthogonal group of $\Lambda$: \[O(\Lambda) = \{\phi\in O(V)\,|\, \phi(\Lambda) = \Lambda\}.\]

If $\Lambda$ satisfies $T(x, y) \in \mathbb{Z}$ for all $x, y \in \Lambda$ then we call $\Lambda$ integral, and if further we have that the scale ideal is equal to $\mathbb{Z}$, i.e.\ $\mathfrak{s}(\Lambda) = \langle T(x,y)\,|\, x,y\in\Lambda \rangle = \mathbb{Z}$ then it is primitive. A primitive lattice is even if $Q(x) \in  \mathbb{Z}$ for all $x \in \Lambda$, otherwise it is odd. We define the norm ideal of $\Lambda$ to be the ideal $\mathfrak{n}(\Lambda) = \langle Q(x)\,|\, x\in\Lambda\rangle \subset\mathbb{Q}$.

One defines the discriminant group of $\Lambda$ to be $\Delta(\Lambda) = \Lambda^{\ast}/\Lambda$, where the dual lattice of $\Lambda$ with respect to $T$ is $\Lambda^{\ast}:= \{v \in V \,:\, T(v, \Lambda) \subseteq \mathbb{Z}\}$. Whenever $\Delta(\Lambda)$ is a direct sum of cyclic groups of prime order for an integral lattice, then the lattice is called elementary.

As with quadratic spaces, lattices can be completed to give local lattices for $p$ prime: \begin{align*}\Lambda_{\infty} &= \Lambda\otimes_{\mathbb{Z}}\mathbb{R} = V_{\infty}\\ \Lambda_p &= \Lambda\otimes_{\mathbb{Z}}\mathbb{Z}_p\subset V_p\end{align*} One might hope that this collection of local lattices completely determines $\Lambda$ up to isometry, but the analogue of the Hasse-Minkowski theorem fails. To quantify the failure, one considers the genus and class set: \begin{align*}\text{gen}(\Lambda) &= \{\Pi\subset V\,|\, \Pi_p \sim\Lambda_p \text{ for all } p\}\\ \text{Cl}(\Lambda) &= O(V)\backslash\text{gen}(\Lambda)\end{align*} The class set is known to be a finite set by a geometry of numbers argument, and its size $h = |\text{Cl}(\Lambda)|$ is the class number of $\Lambda$.

We can still define the notions of dual lattice, discriminant group, integral lattice, primitive lattice, scale ideal and norm ideal in the local lattice case by replacing $\mathbb{Z}$ with $\mathbb{Z}_p$. The notion of even lattice can also be defined over $\mathbb{Z}_2$ in the obvious way.

Since one direction of Hasse-Minkowski for lattices holds (i.e.\ global to local), it's still useful to have a classification of local lattices up to isometry (treated for example in Gerstein \cite{gerstein2008basic}). Indeed, for every prime $p$ it is known that $\Lambda_p$ has a Jordan decomposition, i.e.\ $\Lambda_p = \Lambda_{p,1} \perp \dots \perp  \Lambda_{p,k}$ where each $\Lambda_{p,i}$ has rank $n_i$ and is $p^{a_i}$-modular for some $a_i\in\mathbb{Z}$, i.e.\ has Gram matrix $p^{a_i} U_i$ for some $U_i\in\text{GL}_{n_i}(\mathbb{Z}_p)$. We usually rearrange the $\Lambda_{p,i}$ so that $a_1 < a_2 < ... < a_k$ and refer to the case $a_i = 0$ as being a unimodular lattice.

When $p$ is odd it is known that $\Lambda_{p,i}\sim p^{a_i}\langle 1,1,...,1,\delta_i\rangle$ for a unique $\delta_i\in\{1,r\}$, where $r$ is a quadratic non-residue mod $p$ (we will use this notation later). The collection of invariants $(n_i, a_i, \delta_i)_i$ completely determines $\Lambda_p$ up to isometry and all possibilities can occur (completing the classification of local lattices in this case). We will often use Conway-Sloane $p$-adic symbols to denote isometry classes of local lattices (see \cite[Chapter 15]{conway2013sphere} or \cite{allcock2020conway} for details). 

When $p=2$ it is not necessarily true that $\Lambda_{2,i}$ can be diagonalised. However it is known that: \[\Lambda_{2,i} \sim\begin{cases} 2^{a_i}\langle \delta_1,\delta_2,...,\delta_n\rangle & \text{ if } \mathfrak{n}(\Lambda_{2,i}) = 2^{a_i-1}\mathbb{Z}_2\\ 2^{a_i}\left\langle\begin{pmatrix}0 & 1\\ 1 & 0\end{pmatrix},\begin{pmatrix}0 & 1\\ 1 & 0\end{pmatrix},...,\begin{pmatrix}0 & 1\\ 1 & 0\end{pmatrix},A\right\rangle & \text{ if } \mathfrak{n}(\Lambda_{2,i}) = 2^{a_i}\mathbb{Z}_2 \end{cases},\] where each $\delta_i \in\{\pm 1,\pm 5\}$ and $A\in\left\{\begin{pmatrix}0 & 1\\ 1 & 0\end{pmatrix},\begin{pmatrix}2 & 1\\ 1 & 2\end{pmatrix}\right\}$. The lattices of the first kind are often referred to as Type I lattices and the second as Type II lattices. 

Unfortunately, the Jordan decomposition is not enough to classify $2$-adic lattices up to isometry. However, it is still possible to complete the classification by giving further equivalences, for example by the result of O'Meara (e.g.\ Theorem $8.11$ of \cite{gerstein2008basic}). We will often use Conway-Sloane $2$-adic symbols to denote isometry classes of local lattices up to isometry, and assume the result that two $2$-adic lattices are isometric if and only if their symbols are equal or related by a sequence of Oddity Fusion and Sign Walking rules (see \cite[Chapter 15]{conway2013sphere} or \cite{allcock2020conway} for details).

As a consequence of the classification of local lattices, we find that when $\Lambda, \Pi \subset V$ are integral lattices with $ \mathrm{disc}(\Lambda) = \mathrm{disc}(\Pi)$, then $\Lambda_p \simeq \Pi_p$ for all $p \nmid 2\,\mathrm{disc}(\Lambda)$ (since the local lattice must be diagonalisable and unimodular, hence completely determined by its discriminant).

As with quadratic spaces, we often wish to know when a collection of local lattices $\{\Lambda_p\}_{p\leq \infty}$ arises from a global lattice $\Lambda$. By Lemma $5.1$ of Chapter $9$ of Cassels \cite{cassels2008rational}, $\Lambda$ will exist if and only if the corresponding quadratic space $(V,Q)$ in which it embeds exists (relative to the collection of local quadratic spaces $(V_p,Q)$ for each $p\leq \infty$). This is simply a case of checking that the global compatibility conditions mentioned earlier are satisfied.

Kneser showed in \cite{Kneser1957} that it is often possible to compute the class set $\text{Cl}(\Lambda)$ via his method of iterated $p$-neighbours. Let $p$ be prime and $\Lambda$ be an integral lattice in some quadratic space $(V,Q)$ over $\mathbb{Q}$. We say that an integral lattice $\Pi \subset V$ is a $p$-neighbour of $\Lambda$ if $\Lambda/(\Lambda \cap \Pi) \cong \Pi/(\Lambda\cap\Pi) \cong \mathbb{Z}/p\mathbb{Z}$, i.e.\ the intersection has index $p$ in both $\Lambda$ and $\Pi$. We write $\Lambda\sim_p \Pi$ to denote that $\Pi$ is a $p$-neighbour of $\Lambda$. 

It is known that if $\Pi$ is a $p$-neighbour of $\Lambda$ then $\mathrm{disc}(\Lambda) = \mathrm{disc}(\Pi)$ and $\Pi_q = \Lambda_q$ for all primes $q \neq p$. Further, when $p \nmid 2\,\mathrm{disc}(\Lambda)$ then $\Lambda_p \sim \Pi_p$, since both are local lattices that are unimodular with the same discriminant. In this case we then have that $\Pi \in \text{gen}(\Lambda)$, so that $p$-neighbour moves allow us to navigate the genus of $\Lambda$ in a way that can generate new isometry classes. In particular, Kneser showed that if $(V,Q)$ is a non-degenerate quadratic space over $\mathbb{Q}$ of dimension $n\geq 3$, then for any two classes $[\Lambda_1], [\Lambda_2] \in \text{Cl}(\Lambda)$ there exists a finite set of primes $p_1, p_2,\dots, p_s$ and integral lattices $\Pi_i\subset V$ such that:
$$\Lambda \sim_{p_1} \Pi_1 \sim_{p_2} \dots \sim_{p_s} \Pi_s = \Lambda_i.$$ 
The lattices considered in this paper have only one spinor genus, and in this case it is known that we may take $p_i = p$ for all $i$, for any fixed prime $p\nmid 2\,\mathrm{disc}(\Lambda)$. 

As explained in \cite{voight2023kneser}, when $p\nmid 2\,\mathrm{disc}(\Lambda)$ and $n\geq 3$ the $p$-neighbours of $\Lambda$ are in bijection with isotropic lines in $\Lambda/p\Lambda$, and so are easy to compute. In particular, if $[v]\in\Lambda/p\Lambda$ generates an isotropic line then the corresponding $p$-neighbour is: \[\Pi = \mathbb{Z}\left(\frac{v}{p}\right) + \{w\in \Lambda\,|\,T(w,v)\in p\mathbb{Z}\}.\] In particular, the number of $p$-neighbours $N(p)$ is finite. When $n = 3$ we have $N(p) = p+1$ (the number of solutions to the quadric $Q(x) \equiv 0 \bmod p$ in $\mathbb{P}^2(\mathbb{F}_p)$).

Both the genus and the class set of $\Lambda$ have adelic interpretations. Consider the group of adelic points of the global orthogonal group $O(V)$ with respect to the local orthogonal groups $O(\Lambda_p)$, i.e.\: 
$$O(\widehat{V}) = \{(g_p)_p \in \prod_pO(V_p):g_p \in O(\Lambda_p) \text{ for all but finitely many primes } p\}.$$ Then $O(\widehat{V})$ acts transitively on $\text{gen}(\Lambda)$ and the stabiliser of $\Lambda$ is $O(\widehat{\Lambda}) = \prod_p O(\Lambda_p)$. Hence there are bijections: \begin{align*}\text{gen}(\Lambda) &\longleftrightarrow O(\widehat{V})/O(\widehat{\Lambda})\\ \text{Cl}(\Lambda) &\longleftrightarrow O(V)\backslash O(\widehat{V})/O(\widehat{\Lambda}) \end{align*}

As discussed in Hein \cite[Proposition 3.1.12]{hein2016orthogonal}, there is an adelic description of $p$-neighbours. If $p \nmid 2\, \mathrm{disc}(\Lambda)$ and $n\geq 3$ then $V_p$ is isotropic and we can fix a basis for $V_p$ so that $V_p = \mathbb{H} \perp \mathbb{H}^{\perp}$, with $\mathbb{H}\sim\left\langle\begin{pmatrix}0 & 1\\ 1 & 0\end{pmatrix}\right\rangle$ a hyperbolic plane. We can then construct the element $\widehat{p} \in O(\widehat{V})$ such that \[\widehat{p}_{(q)} = \begin{cases}
    \left(\begin{smallmatrix}
    p&0&0&\dots&0\\
    0& \frac{1}{p}&0&\dots&0\\
    0&0&1&\dots&0\\
    \vdots &&\ddots&\ddots&\\
    0&\dots&\dots&0&1
\end{smallmatrix}\right),&\text{if } q = p\\
I_n,&\text{if } q \neq p \end{cases}\] The $p$-neighbours of $\Lambda$ are then shown to be in bijection with coset representatives in the decomposition:
\[
O(\widehat{\Lambda})\, \widehat{p}\, O(\widehat{\Lambda}) = \bigsqcup_{m=1}^{N(p)} \widehat{p}_m O(\widehat{\Lambda}),
\]

The theory of $2$-neighbours is a little trickier to handle due to the fact that the local classification is more involved. However, when $\Lambda$ is even and $v_2(\mathrm{disc}(\Lambda)) \leq 1$ (which will often be the case of interest in this paper), the theory as described above will still work for $p=2$. We refer the reader to \cite{schulze2020lecture} for more information.

\subsection{Ternary Lattices and Quaternion orders}\label{Ternary Lattices and Quaternion orders}

In this subsection we recall the Clifford correspondence between ternary lattices and quaternion orders. We assume that the reader is familiar with the theory of quaternion algebras and their orders. A thorough discussion of these topics is given in \cite{voight2021quaternion}.

Let $\Lambda = \mathbb{Z}e_1\oplus \mathbb{Z}e_2\oplus \mathbb{Z}e_3$ be an even ternary lattice in a quadratic space $(V,Q)$ over $\mathbb{Q}$. Suppose that the Gram matrix of $\Lambda$ in this basis is \[T = \begin{pmatrix} 2a & u & v\\ u & 2b & w\\ v & w & 2c \end{pmatrix}.\] The even Clifford algebra of $V$ is the associative $\mathbb{Q}$-algebra given by: \[B = \mathrm{Cl}^0(V) = \mathbb{Q}\oplus \mathbb{Q}e_1e_2 \oplus \mathbb{Q}e_2e_3\oplus \mathbb{Q} e_1e_3,\] with relations $e_i^2 = Q(e_i)$ and $e_ie_j + e_je_i = T(e_i,e_j)$. It is known that $B$ is a quaternion algebra and is definite if and only if $(V,Q)$ is positive definite. 

The lattice $\Lambda\subset V$ gives rise to an order of discriminant $\mathrm{discrd}(\mathcal{O}) = \frac{1}{2}\mathrm{disc}(\Lambda)$: \[\mathcal{O} = \mathrm{Cl}^0(\Lambda) = \mathbb{Z} \oplus \mathbb{Z}e_1e_2 \oplus \mathbb{Z}e_2e_3 \oplus \mathbb{Z}e_1e_3\subset B.\] Note that it is important that $\Lambda$ is even, since otherwise $\mathrm{Cl}^0(\Lambda)$ may not be closed under multiplication. If $\Lambda$ is not even then we may instead work with a suitable scaling of the quadratic map (although altering the discriminant). 

We wish to give a standard basis for the quaternion algebra $B$. In order to do this we work out the Gram matrix of the trace form $\text{trd}(\alpha\bar{\beta})$ on the basis given above to get the matrix: \[A = \begin{pmatrix}2 & u & w & v\\ u & 2ab & uw-bv & aw\\ w & uw-bv & 2bc & cu\\ v & aw & cu & 2ac\end{pmatrix}.\]

The matrix: \[M = \begin{pmatrix}1 & 0 & 0 & 0\\ -u & 2 & 0 & 0\\ bv-uw & w & u & -2b\\ u^2w-buv-2abw & 2bv-uw & 4ab-u^2 & 0\end{pmatrix}\] is then such that $D = MAM^T$ is the diagonal matrix: \[D = 
\begin{pmatrix}2 & 0 & 0 & 0\\ 0 & 2(4ab-u^2) & 0 & 0\\ 0 & 0 & 2bN & 0\\ 0 & 0 & 0 & 2bN(4ab-u^2)
\end{pmatrix}.\] It is now clear that $B = \left(\frac{u^2-4ab\,,\,-bN}{\mathbb{Q}}\right)$ under the basis $\{1,i,j,k\}$ with: \begin{align*}
    i & = 2e_1e_2-u,\\ j &= we_1e_2 + ue_2e_3 - 2be_1e_3 + (bv-uw),\\ k &= (2bv-uw)e_1e_2 + (4ab-u^2)e_2e_3 + (u^2w-buv-2abw).
\end{align*} (It is easily checked that the standard relations $i^2 = u^2-4ab, j^2 = -bN$ and $ij=k=-ji$ are satisfied).\\
Inverting the change of basis matrix $M$ gives: \[M^{-1} = \begin{pmatrix}1 & 0 & 0 & 0\\ \frac{u}{2} & \frac{1}{2} & 0 & 0\\ \frac{w}{2} & \frac{uw-2bv}{8ab-2u^2} & 0 & \frac{1}{4ab-u^2}\\ \frac{v}{2} & \frac{2aw-uv}{8ab-2u^2} & -\frac{1}{2b} & \frac{u}{8ab^2-2bu^2}
\end{pmatrix}\] so that \begin{align*}e_1e_2 &= \frac{u+i}{2}\\ e_2e_3 &= \frac{w(4ab-u^2) + (uw-2bv)i + 2k}{8ab-2u^2}\\ e_1e_3 &= \frac{bv(4ab-u^2) + b(2aw-uv)i - (4ab-u^2)j + uk}{8ab^2-2bu^2}\end{align*} so that in terms of the standard basis $\{1,i,j,k\}$ we have: \small \[\mathcal{O} = \mathbb{Z}\oplus\mathbb{Z}\frac{u+i}{2}\oplus\mathbb{Z}\frac{w(4ab-u^2) + (uw-2bv)i + 2k}{8ab-2u^2}\oplus \mathbb{Z}\frac{bv(4ab-u^2) + b(2aw-uv)i - (4ab-u^2)j + uk}{8ab^2-2bu^2}.\] 
\normalsize 
It is easily checked that the Gram matrix of this order is equivalent to the one above and that $\mathrm{discrd}(\mathcal{O}) = \frac{1}{2}\mathrm{disc}(\Lambda)$ under this basis, as expected.

The association $\Lambda \rightarrow\mathrm{Cl}^0(\Lambda)$ has been studied many times historically and can be viewed as a generalisation of Gauss' correspondence between integral binary quadratic forms and ideals in quadratic rings. It has many interesting functorial properties, which are explored in Chapter $22$ of \cite{voight2021quaternion}. In particular, the correspondence is functorial in the underlying rings, so that we may compatibly extend the recipe to connect local lattices and local orders. Functoriality implies the existence of a bijection $\text{Cl}(\Lambda) \longleftrightarrow\text{Typ}(\mathcal{O})$ that identifies isometries on the LHS with isomorphisms on the RHS (where $\text{Typ}(\mathcal{O})$ is the set of orders in $B$ that are locally isomorphic to $\mathcal{O}$ everywhere, up to global isomorphism).

\subsection{Classical Modular Forms} \label{subsec: classical MFs}
We assume that the reader is familiar with the theory of modular forms and Hecke operators (see \cite{diamond2005first} and \cite{cohen2017modular} for a detailed introduction). We denote by $S_k(\Gamma_0(N))$ the space of cusp forms of weight $k$ and level $N$, and $T_p$ the usual Hecke operator at $p$ acting on this space. Recall that $S_k(\Gamma_0(N))$ has a basis of eigenforms for the algebra generated by the Hecke operators $T_p$ for $p\nmid N$ (we will not need to consider the other Hecke operators in this paper).

The space of oldforms is given by:
\begin{align*}
S_{k}^{\text{old}}\left(\Gamma_0(N)\right) = \sum_{\substack{M \mid N \\M\neq N}} \sum_{d \mid \tfrac{N}{M}} \alpha_d(S_{k}(\Gamma_0(M))).
\end{align*} where $\alpha_d(f)(\tau) = f(d\tau)$ are the degeneracy maps. The new subspace is then defined via the orthogonal direct sum decomposition (with respect to the Petersson inner product): \[S_k(\Gamma_0(N)) = S_k^{\text{old}}\left(\Gamma_0(N)\right) \oplus S_k^{\textup{new}}\left(\Gamma_0(N)\right)\] More generally, if $M|N$ then the $M$-new subspace is given by: \[S_k^{M\mathrm{-new}}\left(\Gamma_0(N)\right) = \displaystyle\sum_{M\mid M^{\prime}\mid N} \sum_{d \mid \frac{N}{M^{\prime}}} \alpha_d(S_k^{\textup{new}}(\Gamma_0(M^{\prime})).\]

We also recall that there are additional involutions acting on the spaces $S_k(\Gamma_0(N))$. For each $d$ where $\gcd(d,N/d) = 1$, the Atkin-Lehner involution $W_d$ acts on $S_{k}\left(\Gamma_0(N)\right)$ via $f \mapsto f|_{k}[W_d]$ where:    	
\begin{align*}
    W_d = \left(\begin{smallmatrix}
        dx & y \\
        Nz & d
    \end{smallmatrix}\right) \, \text{ for integers $x,y,z$ such that } dx - \tfrac{Nyz}{d} = 1.
\end{align*}
Here the weight $k$ slash operator is given by $(f|_{k} [\gamma])(\tau) = \text{det}(\gamma)^{k-1}(c\tau + d)^{-k} f(\gamma \cdot \tau)$ with $\gamma\cdot \tau$ being the Mobius action of $\gamma = \begin{pmatrix}a & b\\ c & d\end{pmatrix}\in\text{GL}_2^+(\mathbb{Q})$ on $\tau\in\mathcal{H}$ (complex upper half plane). 

In order to keep track of Atkin-Lehner eigenvalues, the following notation from \cite{heintornariavoight} will be useful. Let $N$ and $d$ be positive integers. Then the corresponding sign vector is $\varepsilon^{(N)}_d= (\eta_p)_{p^e\mid \mid N}$ where \[\eta_p = \begin{cases}+1 & \text{ if } p\nmid d\\ -1 & \text{ if } p\mid d\end{cases}.\]
Sign vectors will later be multiplied component-wise.

For each sign vector $\varepsilon^{(N)}_d$, we have eigenspaces of the form
\begin{align*}\label{eq:neweps}
  S_{k}\!\left(\Gamma_0(N), \varepsilon^{(N)}_d\right)
  \;:=\;
  \Bigl\{
    f \in S_{k}\!\left(\Gamma_0(N)\right)
    :
    W_q f = \eta_p f
    \text{ for all } q = p^e \| N
  \Bigr\}.
\end{align*} Since the Atkin-Lehner involutions $W_d$ commute with degeneracy operators and the Hecke operators $T_p$ when $p \nmid N$ and $\gcd(p,d) = 1$, this induces decompositions of Hecke-modules for each $M|N$:  \[
  S^{M\mathrm{-new}}_{k}\!\left(\Gamma_0(N)\right)
  \;=\;
  \bigoplus_{\varepsilon^{(N)}_d}
  S^{M\mathrm{-new}}_{k}\!\left(\Gamma_0(N), \varepsilon^{(N)}_d\right).
\] where $\varepsilon_d^{(N)}$ runs over all possible sign vectors corresponding to divisors $d \,|\, N$.

\subsection{Orthogonal Modular Forms}\label{subsec: OMFs}

We now introduce the theory of (definite) orthogonal modular forms. This is a special case of the wider theory of algebraic modular forms introduced by Gross in \cite{gross1999algebraic} and discussed in detail for our particular case in \cite{greenbergvoight}. Let $(V,Q)$ be a positive definite quadratic space over $\mathbb{Q}$ and $\Lambda\subset V$ be a lattice. Also let $W$ be a finite-dimensional representation of $O(V)$. The corresponding space of orthogonal modular forms of weight $W$ and level $O(\widehat{\Lambda})$ is then defined to be:

\[\mathcal{M}(O(V),O(\widehat{\Lambda}),W) = \left\{\varphi : O(\widehat{V}) \to W \,\Bigg|\, \begin{array}{c}\varphi(\gamma gk) = \gamma \varphi(g) \text{ for all }\\ (\gamma, g, k) \in O(V)\times O(\widehat{V})\times O(\widehat{\Lambda})\end{array}\right\}.\] These spaces are finite dimensional, since any such function is completely determined by its values on $O(V)\backslash O(\widehat{V})/O(\widehat{\Lambda}) \leftrightarrow \text{Cl}(\Lambda)$.

For each $\hat{g}\in O(\widehat{V})$ there exists a Hecke operator $T_{\hat{g}}$ acting on $\mathcal{M}(O(V),O(\widehat{\Lambda}),W)$. In direct analogy with classical Hecke operators, one considers the (finite) double coset decomposition: \[O(\widehat{\Lambda})\,\hat{g}\,O(\widehat{\Lambda}) = \coprod_{i=1}^{r(\hat{g})}\hat{g}_i O(\widehat{\Lambda})\] and defines: \[T_{\hat{g}}(\varphi)(g) = \sum_{i=1}^{r(\hat{g})}\varphi(gg_i).\]

Suppose from now on that $n \geq 3$. If we take $\hat{g} = \hat{p}$ for any prime $p\nmid 2\,\mathrm{disc}(\Lambda)$ (where $\hat{p}\in O(\widehat{V})$ is as defined previously) then we have a family of Hecke operators $T_{\hat{p}}$ acting on $\mathcal{M}(O(V),O(\widehat{\Lambda}),W)$. We will often simply write these operators as $T_p$ for reasons that will become clear soon. The algebra generated by these operators is commutative.

Recall that the coset representatives for $T_p$ can be identified with $p$-neighbours of $\Lambda$. Under the identification $O(\widehat{V})/O(\widehat{\Lambda}) \leftrightarrow \text{gen}(\Lambda)$ we then find that this operator simply averages the value of $\phi$ over the $p$-neighbours of the input lattice. In the special case where $W = \mathbb{C}$ is the trivial representation then \[\mathcal{M}(O(V),O(\widehat{\Lambda}),W) \cong \{\varphi:\text{Cl}(\Lambda) \rightarrow\mathbb{C}\} \cong \mathbb{C}^h\] and under the standard basis the operator $T_p$ simply counts the number of $p$-neighbours of $\Lambda_i$ that are isometric to $\Lambda_j$ (where $\text{Cl}(\Lambda) = \{[\Lambda_1],[\Lambda_2],...,[\Lambda_h]\}$). Alternatively, these operators represent the adjacency matrices of the (directed) $p$-neighbour graphs (with vertices for each of the classes $[\Lambda_i]$ and directed edges from $[\Lambda_i]$ to $[\Lambda_j]$ for each $p$-neighbour of $[\Lambda_i]$ that is isometric to $[\Lambda_j]$). 

It is known that the space $\mathcal{M}(O(V),O(\widehat{\Lambda}),W)$ has a basis of simultaneous eigenforms for the Hecke algebra generated by the operators $T_p$ for $p\nmid 2\,\mathrm{disc}(\Lambda)$ prime. Again, a lot of what is described above still works if we include the prime $p=2$ (e.g.\ if $\Lambda$ is even and $v_2(\mathrm{disc}(\Lambda)) \leq 1$).
    
Note that the constant function $\mathbbm{1}$ is an eigenvector for every Hecke operator $T_{\hat{g}}$, having eigenvalue $r(\hat{g})$ (i.e.\ when $\hat{g} = \hat{p}$ for $p\nmid 2\,\mathrm{disc}(\Lambda)$ the eigenvalue is $N(p)$, the total number of $p$-neighbours of $\Lambda$).

As with spaces of classical modular forms, it is possible to decompose into the direct product of an ``Eisenstein subspace" and a ``cuspidal subspace" (even though the locally symmetric space in this case has no cusps!). Without giving the full definition we remark that the Eisenstein subspace can only be non-trivial when the weight $W = \mathbb{C}$ is the trivial representation (see \cite[\S1.2]{rama2020computation}), and in all cases considered in this paper it is simply the one-dimensional space spanned by $\mathbbm{1}$. The cuspidal subspace $\mathcal{S}(O(V), O(\hat{\Lambda}), W)$ is then defined to be the full space $\mathcal{M}(O(V),O(\hat{\Lambda}), W)$ if $W$ is a non-trivial weight and the span of the Hecke eigenforms that are not $\mathbbm{1}$ when $W = \mathbb{C}$ is trivial (alternatively, the orthogonal complement of $\mathbbm{1}$ with respect to a suitable inner product, defined in \cite{hein2016orthogonal}). The Hecke operators are known to respect this decomposition.
    
Later we will make heavy use of theorems of Hein-Tornaria-Voight \cite{heintornariavoight}, proving the existence of Hecke-equivariant isomorphisms between certain spaces of orthogonal modular forms and spaces of classical modular forms. We will introduce these now.

From now on we specialise to the case of $n = 3$ (the case of interest in this paper). In this case it is well known that: \[O(V) \twoheadrightarrow O(V)/\{\pm I\} \cong SO(V) \cong B^{\times}/\mathbb{Q}^{\times}.\] Here $B = \mathrm{Cl}^0(V)$ is the even Clifford algebra of $V$, which in our case is a definite quaternion algebra over $\mathbb{Q}$. Any representation of $B^{\times}/\mathbb{Q}^{\times}$ can be lifted to one for $O(V)$ via the surjection above. 

One such family of representations comes from choosing a splitting: \[B^{\times}/\mathbb{Q}^{\times}\longhookrightarrow \text{PGL}_2(\mathbb{C}).\] For each integer $k\geq 0$ the irreducible representation $W_{2k} = \text{Sym}^{2k}(\mathbb{C})\otimes\text{det}^{-k}$ of $\text{GL}_2(\mathbb{C})$ descends to $\text{PGL}_2(\mathbb{C})$ and so gives rise to an irreducible representation of $O(V)$. This representation can also be constructed by considering the action of $O(V)$ on a suitable space of harmonic polynomials on $V$ with respect to an explicit Laplacian $\Delta_Q$. See Example $3.5.3$ of \cite{hein2016orthogonal} for more details.

Secondly, we can construct certain characters of $O(V)$. Applying the reduced norm map gives \[\text{nrd}:B^{\times}/\mathbb{Q}^{\times} \longrightarrow \mathbb{Q}^{\times}/(\mathbb{Q}^{\times})^2.\] For each $d\in\mathbb{N}$ we get a character: \begin{align*} \phi_d:\mathbb{Q}^{\times}/(\mathbb{Q}^{\times})^2 &\longrightarrow \{\pm 1\}\\ a &\longmapsto \prod_{p\mid d}(-1)^{v_p(a)}. \end{align*} The composite $\theta_d = \phi_d \circ \text{nrd}$ defines spinor characters of $O(V)$ for each choice of $d$.

The following results will come in handy later (with the spaces $M(SO(V),SO(\widehat{\Lambda}),W)$ being defined similarly to above).

\begin{lem}\label{OtoSO}
For each lattice $\Lambda\subset V$ there is a Hecke-equivariant isomorphism: \[\mathcal{M}(O(V), O(\widehat{\Lambda}),W_{2k}\otimes \theta_d) \cong M(SO(V),SO(\widehat{\Lambda}), W_{2k}\otimes \theta_d)\]
\end{lem}

\begin{proof}
The restriction map from $O(\widehat{V})$ to $SO(\widehat{V})$ induces a Hecke-equivariant injective homomorphism \[\mathcal{M}(O(V),O(\widehat{\Lambda}),W) \longrightarrow M(SO(V),SO(\widehat{\Lambda}),W)\] for any representation $W$ of $O(V)$. The image consists of functions \[\varphi\in M(SO(V),SO(\widehat{\Lambda}),W)\] such that $\varphi(g) = \varphi((-I)\cdot g) = (-I)\cdot \varphi(g)$ for every $g\in SO(\widehat{V})$. When $-I$ acts trivially on $W$ this condition is automatically satisfied and so we get a Hecke-equivariant isomorphism. All of the representations constructed above have this property, establishing the result. 
\end{proof}

Before moving on to the main result of Hein-Tornaria-Voight we need to recall one more definition. Let $\mathcal{O}$ be an order in a quaternion algebra $B$ over $\mathbb{Q}$ and $p \mid \mathrm{discrd}(\mathcal{O})$ be prime. Then the Eichler invariant $\varepsilon_p(\mathcal{O})$ takes values in $\{-1,0,+1\}$ and is defined by (where $\text{rad}(\mathcal{O}_p)$ is the Jacobson radical of the local order $\mathcal{O}_p = \mathcal{O}\otimes_{\mathbb{Z}}\mathbb{Z}_p$):
$$\varepsilon_p(\mathcal{O}) = \begin{cases}
 +1, \quad &\text{if }\mathcal{O}_p/\text{rad}(\mathcal{O}_p) \cong \mathbb{F}_p \times \mathbb{F}_p \text{ \quad ($\mathcal{O}$ is \textit{residually split} at } p),\\
 0, \quad &\text{if }\mathcal{O}_p/\text{rad}(\mathcal{O}_p) \cong \mathbb{F}_p \quad (\text{$\mathcal{O}$ is \textit{residually ramified} at $p$),}\\
 -1, \quad &\text{if }\mathcal{O}_p/\text{rad}(\mathcal{O}_p) \cong \mathbb{F}_{p^2} \quad (\text{ $\mathcal{O}$ is \textit{residually inert} at }p).
\end{cases}$$
 The orders featuring in the main results of \cite{heintornariavoight} are residually unramified everywhere, i.e.\ having local Eichler invariant $\varepsilon_p(\mathcal{O}) \in \{\pm 1\}$ for all primes $p \mid \mathrm{discrd}(\mathcal{O})$. In particular, Eichler orders are examples of such orders, since if $\mathrm{discrd}(\mathcal{O}) = DM$ with $D = \mathrm{disc}(B)$ and $M$ coprime to $D$ then it can be shown that \[\varepsilon_p(\mathcal{O}) = \begin{cases}
    -1, &\text{if } p \mid D, \\
     1, &\text{if } p \mid M
\end{cases}.\] A complete classification of local orders at odd primes via their Eichler invariants is possible, and is given in Section $24.6$ of \cite{voight2021quaternion}. The analogous classification for $p=2$ is given in \cite{lemurell}.

The following main result is then a consequence of the above lemma, along with Theorem 4.10, Theorem 4.12 and Corollary 6.12 in \cite{heintornariavoight}.

\begin{thm}\label{thm: 6.12 in HTV}
Suppose that $\Lambda\subset V$ is an even lattice in a positive definite ternary quadratic space $(V,Q)$ over $\mathbb{Q}$. Suppose that the associated quaternion algebra $B = \textup{Cl}^0(V)$ has discriminant $D$ and that the order $\mathcal{O} = \textup{Cl}^0(\Lambda)$ is residually unramified everywhere of discriminant $N = \textup{discrd}(\mathcal{O}) = \frac{1}{2}\textup{disc}(\Lambda)$. Let $N^+$ and $N^-$ be the products of the prime powers $p^e||N$ that satisfy $\varepsilon_p(\mathcal{O}) = 1$ and $\varepsilon_p(\mathcal{O}) = -1$ respectively. 

Then for each $k \geq 0$ and each square-free $d\mid N$ we have a Hecke-equivariant isomorphism (for all Hecke operators $T_p$ with $p \nmid N$):
$$\mathcal{S}(O(V), O(\widehat{\Lambda}), W_{2k}\otimes \theta_d) \cong 
\bigoplus_{M|N^-} S_{2k+2}^{M\textup{-new}}(\Gamma_0(MN^{+}), \varepsilon_D^{(MN^{+})}\varepsilon_d^{(MN^{+})})$$
where the sum runs over $M|N^-$ with $v_p(M) \equiv v_p(N) \bmod 2$ for all primes $p \,|\, N^-$.
\end{thm}

Note that $D = \mathrm{disc}(B)$ is known to be square-free and to divide $N^-$ and so it makes sense to construct the sign vector $\varepsilon_D^{(N)}$.

The aim of this paper is to generalise this result to the cases of $\Lambda$ of half-discriminant $p^2$ with $p$ prime and $4p^2$ with $p$ an odd prime. We will consider all genera of such discriminants, even and odd (which will prove to be a substantial task!). Such lattices cannot be handled using the above results because when the underlying orders exist (i.e.\ $\Lambda$ is even) they are residually ramified at $p$ or $2$. Our initial strategy for proving Hecke-equivariant isomorphisms between the spaces $\mathcal{S}(O(V),O(\widehat{\Lambda}),W_{2k}\otimes\theta_d)$ and spaces of classical modular forms will be to shift focus to various sublattices and superlattices $\Pi\subset V$ satisfying $O(\widehat{\Lambda}) = O(\widehat{\Pi})$. On the level of Hecke modules, doing this makes no change.

\begin{lem}\label{samespace}
Suppose $\Lambda,\Pi\subset V$ are lattices such that $O(\widehat{\Lambda}) = O(\widehat{\Pi})$. Then for each representation $W$ of $O(V)$ there is an equality of Hecke modules: \[\mathcal{M}(O(V),O(\widehat{\Lambda}),W) = \mathcal{M}(O(V),O(\widehat{\Pi}),W).\]
\end{lem}

\begin{proof}
Since $O(\widehat{\Lambda}) = O(\widehat{\Pi})$ is follows that $\mathcal{M}(O(V),O(\widehat{\Lambda}),W) = \mathcal{M}(O(V),O(\widehat{\Pi}),W)$ as spaces. We also have that: \[O(\widehat{\Lambda})\,g\, O(\widehat{\Lambda}) = O(\widehat{\Pi})\,g\, O(\widehat{\Pi})\] for any $g\in O(\widehat{V})$. This implies that we get the same Hecke representatives and hence the same Hecke operators on both sides.
\end{proof}

We will often find ourselves in a position where the related lattice $\Pi$ can be viewed as a lattice of smaller discriminant upon scaling of the underlying quadratic form. Once again, this provides an equivalent Hecke module.

\begin{lem}\label{scaling}
Let $\Lambda\subset V$ be a lattice in a quadratic space $(V,Q)$ over $\mathbb{Q}$. The Hecke module $\mathcal{M}(O(V),O(\widehat{\Lambda}),W)$ is invariant under non-zero scaling of the underlying quadratic map $Q$.
\end{lem}

\begin{proof}
This is also clear. For any $\alpha\in\mathbb{Q}^{\times}$ the scaled quadratic space $(V,Q^{(\alpha)})$ has the same orthogonal group as $(V,Q)$. The same is true for $(\widehat{V},Q^{(\alpha)})$ and $(\widehat{V},Q)$. Consequently, the subgroup $O(\widehat{\Lambda})$ is invariant under scaling of the underlying quadratic map, which by Lemma \ref{samespace} implies an equality of spaces of orthogonal modular forms and Hecke operators.
\end{proof}

\section{Local ternary lattices sharing the same orthogonal group}\label{Orthog}

As previously mentioned, we will be interested in constructing sublattices and superlattices of certain lattices $\Lambda$ that share the same adelic orthogonal group. With this in mind, it will be handy to have explicit results concerning local lattices that have equal orthogonal groups. We give the relevant results of this flavour now, in anticipation of later results.

\begin{prop}\label{p-mod}
Let $p$ be an odd prime. Suppose that $(V_p,Q)$ is a ternary quadratic space over $\mathbb{Q}_p$ containing a $p$-adic lattice $\Lambda_p\subset V_p$ with Gram matrix $\langle a,p,pb\rangle$ for some $a,b\in\{1,r\}$. Then the $p$-adic sublattice $\Pi_p\subset V_p$ of index $p$ with Gram matrix $\langle p^2a,p,pb\rangle$ satisfies $O(\Lambda_p) = O(\Pi_p)$.
\end{prop}

\begin{proof}
Note that since $\Lambda_p$ is an elementary lattice (i.e.\ $\Lambda_p^*/\Lambda_p$ is an elementary $p$-group), Theorem $1.1(1)$ of Allcock \cite{allcock2025} tells us that $O(\Lambda_p)$ is maximal in $O(V_p)$ among the orthogonal groups of all $p$-adic lattices. Hence, if we can show that $O(\Lambda_p)\subseteq O(\Pi_p)$ then we will be done.

By standard arguments, there are exactly $p^2+p+1$ sublattices of $\Lambda_p$ of index $p$ up to change of basis. If $\Lambda_p$ has basis $\{v_1,v_2,v_3\}$ with respect to the above Gram matrix then these sublattices have bases \[\{pv_1,v_2,v_3\} \qquad \{v_1+\alpha v_2, pv_2, v_3\} \qquad \{v_1+\alpha v_3, v_2 + \beta v_3, pv_3\},\] with $\alpha,\beta\in \{0,1,...,p-1\}$, and corresponding Gram matrices: \[\begin{pmatrix}
p^2a & 0 & 0\\ 0 & p & 0\\ 0 & 0 & pb
\end{pmatrix}\qquad \begin{pmatrix}
a+p\alpha^2 & p^2\alpha & 0\\ p^2\alpha & p^3 & 0\\ 0 & 0 & pb
\end{pmatrix} \qquad \begin{pmatrix}
a+pb\alpha^2 & pb\alpha\beta & p^2b
\alpha\\ pb\alpha\beta & p(1+b\beta^2) & p^2b\beta\\ p^2b\alpha & p^2b\beta & p^3b
\end{pmatrix}.\] Note that the first Gram matrix corresponds to the lattice $\Pi_p$, and is visibly the only one having norm ideal $p\mathbb{Z}_p$. Now if $\gamma\in O(\Lambda_p)$ then $\gamma\Lambda_p = \Lambda_p$ and $\gamma\Pi_p$ is a sublattice of $\Lambda_p$ of index $p$. However, isometries preserve norm ideals and so we must have that $\gamma\Pi_p = \Pi_p$, i.e.\ $\gamma\in O(\Pi_p)$. This gives $O(\Lambda_p)\subseteq O(\Pi_p)$, as required.
\end{proof}

\begin{prop}\label{2-unimod}
Let $a,b,c\in\{\pm 1, \pm 5\}$ and let $\Delta = ab+ac+bc$. Suppose that $(V_2,Q)$ is a ternary quadratic space over $\mathbb{Q}_2$ containing a $2$-adic lattice $\Lambda_2\subset V_2$ with Gram matrix $\langle a,b,c\rangle$. Then the $2$-adic sublattice $\Pi_2\subset V_2$ of index $2$ with Gram matrix: \[\begin{pmatrix}
a+c & c & 2c\\ c & b+c & 2c\\ 2c & 2c & 4c    
\end{pmatrix}\sim\begin{cases}\left\langle\begin{pmatrix}
2 & 1\\ 1 & 2
\end{pmatrix}, 12abc\right\rangle & \text{ if } \Delta \equiv 3 \bmod 8\\ \left\langle\begin{pmatrix}
0 & 1\\ 1 & 0
\end{pmatrix},-4abc\right\rangle & \text{ if } \Delta \equiv 7 \bmod 8\end{cases}\] satisfies $O(\Lambda_2) = O(\Pi_2)$.
\end{prop}

\begin{proof}
Theorem $1.1(2)$ of Allcock \cite{allcock2025} shows that $O(\Lambda_2)$ is maximal in $O(V_2)$ among the orthogonal groups of all $2$-adic lattices. Hence it again suffices to show that $O(\Lambda_2) \subseteq O(\Pi_2)$.

Fix a basis $\{v_1,v_2,v_3\}$ for $\Lambda_2$ realising the Gram matrix above. The $2^2+2+1 = 7$ sublattices of $\Lambda_2$ of index $2$ then have bases: \[\{2v_1,v_2,v_3\} \qquad \{v_1+\alpha v_2,2v_2,v_3\}\qquad \{v_1+\alpha v_3, v_2 + \beta v_3, 2v_3\}\] with $\alpha,\beta\in\{0,1\}$ and corresponding Gram matrices:

\[\begin{pmatrix}
4a & 0 & 0\\ 0 & b & 0\\ 0 & 0 & c
\end{pmatrix}\qquad \begin{pmatrix}
a+b\alpha^2 & 2b\alpha & 0\\ 2b\alpha & 4b & 0\\ 0 & 0 & c
\end{pmatrix} \qquad \begin{pmatrix}
a+c\alpha^2 & c\alpha\beta & 2c
\alpha\\ c\alpha\beta & b+c\beta^2 & 2c\beta\\ 2c\alpha & 2c\beta & 4c
\end{pmatrix}.\]

Note that $\Pi_2$ is the sublattice with Gram matrix corresponding to the choice $\alpha=\beta = 1$. It is also the only sublattice having norm ideal $\mathbb{Z}_2$ (the rest give norm ideal $\frac{1}{2}\mathbb{Z}_2$). A similar argument to before shows that $O(\Lambda_2)\subseteq O(\Pi_2)$, as required.

It remains to show the equivalence of Gram matrices.  First we notice that: \[\begin{pmatrix}1 & 0 & 0\\ 0 & 1 & 0\\ -\frac{2bc}{\Delta} & -\frac{2ac}{\Delta} & 1\end{pmatrix} \begin{pmatrix}
a+c & c & 2c
\\ c & b+c & 2c\\ 2c & 2c & 4c
\end{pmatrix} \begin{pmatrix}1 & 0 & -\frac{2bc}{\Delta}\\ 0 & 1 & -\frac{2ac}{\Delta}\\ 0 & 0 & 1\end{pmatrix} = \begin{pmatrix}a+c & c & 0\\ c & b+c & 0\\ 0 & 0 & \frac{4abc}{\Delta}\end{pmatrix}\] (noting that $\Delta\in\mathbb{Z}_2^{\times}$ since $a,b,c\in\{1,3,5,7\}$). The rest then follows by realising that the top left $2\times 2$ block is unimodular of type II with determinant $\Delta$, so must be equivalent to $\begin{pmatrix}0 & 1\\ 1 & 0\end{pmatrix}$ when $\Delta \equiv 7 \bmod 8$ or $\begin{pmatrix}2 & 1\\ 1 & 2\end{pmatrix}$ when $\Delta \equiv 3 \bmod 8$ (Theorem $8.9$ of \cite{gerstein2008basic}).
\end{proof}

\begin{prop}\label{2-mod} Let $a,b,c\in\{ \pm 1, \pm 5\}$ and $(V_2,Q)$ be a ternary quadratic space over $\mathbb{Q}_2$.
\begin{enumerate}[label = (\roman*)]
\item{Suppose that $(V_2,Q)$ contains a $2$-adic lattice $\Lambda_2\subset V_2$ with Gram matrix $\left\langle\begin{pmatrix}2 & 1 \\ 1 & 2\end{pmatrix},-2a\right\rangle$. Then the $2$-adic sublattice $\Pi_2\subset V_2$ of index $2$ with Gram matrix $\left\langle\begin{pmatrix}2 & 1 \\ 1 & 2\end{pmatrix},-8a\right\rangle$ satisfies $O(\Lambda_2) = O(\Pi_2)$.}
\item{Suppose that $(V_2,Q)$ contains a $2$-adic lattice $\Lambda_2\subset V_2$ with Gram matrix $\langle a,b,2c\rangle$. Then the $2$-adic sublattices $\Pi_2, \Omega_2\subset V_2$ of index $2$ with Gram matrices: 

\begin{align*}&\langle a,b+2c,8(c+2b)\rangle &&\langle a+2c,b,8(c+2a)\rangle &&\text{ if } a+b\equiv 0 \bmod 4\\ &\langle a,b,8c\rangle &&\langle a+2c,b+2c+4,8(c+2a+2b)\rangle &&\text{ if } a+b \equiv 2 \bmod 4\end{align*} satisfy $O(\Pi_2) = O(\Lambda_2) = O(\Omega_2)$.

Further, the $2$-adic sublattice $\Sigma_2\subset V_2$ with Gram matrix \[\begin{pmatrix}a+b & 2b & 0\\ 2b & 4b & 0\\ 0 & 0 & 2c\end{pmatrix} \sim \begin{cases}\langle-2c,-2c,-10c\rangle & \text{ if } a+b \equiv 0 \bmod 4 \text{ and } ab \equiv 3 \bmod 8\\ \langle2,-2,2c\rangle & \text{ if } a+b \equiv 0 \bmod 4 \text{ and } ab \equiv 7 \bmod 8\\ \langle 2a,2a,2c\rangle & \text{ if } a+b \equiv 2 \bmod 4 \text{ and } ab \equiv 1 \bmod 8\\ \langle -10a,-2a,2c\rangle & \text{ if } a+b \equiv 2 \bmod 4 \text{ and } ab \equiv 5 \bmod 8 \end{cases}\] satisfies $O(\Lambda_2) \subseteq O(\Sigma_2)$, with equality when any of the following conditions are satisfied: \begin{itemize}
\item{$a+b\equiv 0 \bmod 4$ and $ab\equiv 7 \bmod 8$,}
\item{$a+b\equiv 2 \bmod 4$, $ab\equiv 1 \bmod 8$ and $a\not\equiv c \bmod 4$,}
\item{$a+b\equiv 2 \bmod 4$, $ab\equiv 5 \bmod 8$ and $a\equiv c \bmod 4$.}
\end{itemize}}
\end{enumerate}
\end{prop}

\begin{proof}
\begin{enumerate}[label = (\roman*)]
\item{Theorem $1.1(2)$ of Allcock \cite{allcock2025} shows that $O(\Lambda_2)$ is maximal in $O(V_2)$ among the orthogonal groups of all $2$-adic lattices. Hence it again suffices to show that $O(\Lambda_2) \subseteq O(\Pi_2)$.

Fix a basis $\{v_1,v_2,v_3\}$ for $\Lambda_2$ realising the Gram matrix above. The sublattices of $\Lambda_2$ of index $2$ then have bases: \[\{2v_1,v_2,v_3\}\qquad \{v_1+\alpha v_2,2v_2,v_3\}\qquad\{v_1+\alpha v_3, v_2 + \beta v_3, 2v_3\}\] with $\alpha,\beta\in\{0,1\}$ and corresponding Gram matrices:

\[\begin{pmatrix}
8 & 2 & 0\\ 2 & 2 & 0\\ 0 & 0 & -2a
\end{pmatrix}\qquad \begin{pmatrix} 2(1+\alpha+\alpha^2) & 2(1+2\alpha) & 0\\ 2(1+2\alpha) & 8 & 0\\ 0 & 0 & -2a
\end{pmatrix} \qquad \begin{pmatrix}
2(1-a\alpha^2) & 1-2a\alpha\beta & -4a\alpha\\ 1-2a\alpha\beta & 2(1-a\beta^2) & -4a\beta\\ -4a\alpha & -4a\beta & -8a
\end{pmatrix}.\] A simple calculation shows that the first lattice and the two lattices in the second family are all isometric (each Gram matrix is equivalent to $\langle 6,2,-2a\rangle$ over $\mathbb{Z}_2$). Similarly, the three lattices in the third family corresponding to $(\alpha,\beta) \neq (0,0)$ are all isometric (each Gram matrix is equivalent to $\left\langle\begin{pmatrix}0 & 1\\ 1 & 0\end{pmatrix},24a\right\rangle$ over $\mathbb{Z}_2$). 

The lattice corresponding to $(\alpha,\beta) = (0,0)$ has Gram matrix $\left\langle\begin{pmatrix}2 & 1\\ 1 & 2\end{pmatrix},-8a\right\rangle$, which is not isometric to any of the other index $2$ lattices, showing that $O(\Lambda_2) \subseteq O(\Pi_2)$, as required.}
\item{Fix a basis $\{v_1,v_2,v_3\}$ for $\Lambda_2$ realising the Gram matrix above. Then the sublattices of $\Lambda_2$ of index $2$ have bases: \[\{2v_1, v_2, v_3\}\qquad \{v_1+\alpha v_2,2v_2,v_3\}\qquad \{v_1+\alpha v_3, v_2 + \beta v_3, 2v_3\}\] with $\alpha,\beta\in\{0,1\}$ and corresponding Gram matrices: \[\begin{pmatrix}
4a & 0 & 0\\ 0 & b & 0\\ 0 & 0 & 2c
\end{pmatrix}\qquad \begin{pmatrix} a+b\alpha^2 & 2b\alpha & 0\\ 2b\alpha & 4b & 0\\ 0 & 0 & 2c
\end{pmatrix} \qquad \begin{pmatrix}
a+2c\alpha^2 & 2c\alpha\beta & 4c\alpha\\ 2c\alpha\beta & b+2c\beta^2 & 4c\beta\\ 4c\alpha & 4c\beta & 8c
\end{pmatrix}.\]

The first lattice and the second with $\alpha = 0$ give isometric lattices (either $a,b \equiv \pm 1 \bmod 8$ or $a,b\equiv \pm 5 \bmod 8$, in which case the lattices have equal Conway-Sloane $2$-adic symbol, otherwise they are related by two applications of the standard Sign Walking rule, see \cite{conway2013sphere} or \cite{allcock2020conway}). These lattices visibly have one-dimensional unimodular, $2$-modular and $4$-modular components.

The lattice $\Sigma_2$ comes from the second family with $\alpha = 1$. This lattice is even and all of the other lattices are visibly odd. It follows that $O(\Lambda_2) \subseteq O(\Sigma_2)$. To show that this lattice has Gram matrix equivalent to the diagonal ones given in the proposition we argue as follows:

\vspace{0.1in}
\noindent\textbf{Case 1:} If $a+b\equiv 2 \bmod 4$ then: \begin{align*}\begin{pmatrix}a+b & 2b & 0\\ 2b & 4b & 0\\ 0 & 0 & 2c\end{pmatrix} = 2\begin{pmatrix}\frac{a+b}{2} & b & 0\\ b & 2b & 0\\ 0 & 0 & c\end{pmatrix} &\sim 2\begin{pmatrix}\frac{a+b}{2} & 0 & 0\\ 0 & 2b-\frac{2b^2}{a+b} & 0\\ 0 & 0 & c\end{pmatrix}\\ &\sim \langle a+b,\,\, \frac{4ab}{a+b},\,\, 2c\rangle\\ &\sim\langle a+b,\,\, ab(a+b),\,\, 2c\rangle\end{align*} (since $\frac{2}{a+b}\in\mathbb{Z}_2^{\times}$). If $ab \equiv 1 \bmod 8$ then $b \equiv a \bmod 8$, giving $\langle 2a, 2a, 2c\rangle$. If $ab \equiv 5 \bmod 8$ then $b \equiv 5a \bmod 8$ and so this lattice is $\langle 6a, 30a, 2c\rangle \sim \langle -10a, -2a, 2c\rangle$.

\vspace{0.1in}
\noindent\textbf{Case 2:} If $a+b\equiv 0 \bmod 4$ then the $2\times2$ submatrix: \[\begin{pmatrix}a+b & 2b\\ 2b & 4b\end{pmatrix}\] is even unimodular of discriminant $[2ab]$. It follows that: \[\begin{pmatrix}a+b & 2b & 0\\ 2b & 4b & 0\\ 0 & 0 & 2c\end{pmatrix}\sim \begin{cases}\begin{pmatrix}0 & 2 & 0\\ 2 & 0 & 0\\ 0 & 0 & 2c\end{pmatrix} & \text{ if } ab \equiv 7 \bmod 8\\ \begin{pmatrix}4 & 2 & 0\\ 2 & 4 & 0\\ 0 & 0 & 2c\end{pmatrix} & \text{ if } ab \equiv 3 \bmod 8\end{cases}.\] It is well known that these lattices are isometric to $\langle -2c,-2c,-10c\rangle$ and $\langle 2,-2,2c\rangle$ respectively (since their Conway-Sloane 2-adic symbols satisfy $2^{+2}_{\text{II}}2^\varepsilon_c\sim 2^{(\varepsilon_c)3}$ and $2^{-2}_{\text{II}}2^{\varepsilon_c}_c \sim 2^{(-\varepsilon_c)3}_c$, by Oddity Fusion).

We now show that $O(\Sigma_2)\subseteq O(\Lambda_2)$, under the special conditions given in the proposition. Consider the seven superlattices of $\Sigma_2$. These have bases: \begin{align*}&\left\{\frac{v_1+v_2}{2},2v_2,v_3\right\}\qquad \left\{v_1+v_2,v_2-\frac{\alpha (v_1+v_2)}{2},v_3\right\} \\ &\left\{v_1+v_2,2v_2,\frac{v_3}{2}-\frac{\alpha(v_1+v_2)}{2}-\beta v_2\right\}\end{align*} with $\alpha,\beta\in\{0,1\}$ and Gram matrices: \begin{align*}&\begin{pmatrix}
\frac{a+b}{4} & b & 0\\ b & 4b & 0\\ 0 & 0 & 2c
\end{pmatrix}\qquad \begin{pmatrix} a+b & -\frac{a\alpha}{2}+b\left(1-\frac{\alpha}{2}\right) & 0\\ -\frac{a\alpha}{2}+b\left(1-\frac{\alpha}{2}\right) & \frac{a\alpha^2}{4}+b\left(1-\frac{\alpha}{2}\right)^2 & 0\\ 0 & 0 & 2c
\end{pmatrix} \\ &\begin{pmatrix}
a+b & 2b & -\frac{a\alpha}{2}-b\left(\frac{\alpha}{2}+\beta\right)\\ 2b & 4b & -b(\alpha+2\beta)\\ -\frac{a\alpha}{2}-b\left(\frac{\alpha}{2}+\beta\right) & -b(\alpha+2\beta) & \frac{c}{2}+\frac{a\alpha^2}{4}+b\left(\frac{\alpha}{2}+\beta\right)^2
\end{pmatrix}.\end{align*} A very lengthy brute force calculation shows that when we have one of the following conditions: \begin{itemize}
\item{$a+b\equiv 0 \bmod 4$ and $ab\equiv 7 \bmod 8$,}
\item{$a+b\equiv 2 \bmod 4$, $ab\equiv 1 \bmod 8$ and $a\not\equiv c \bmod 4$,}
\item{$a+b\equiv 2 \bmod 4$, $ab\equiv 5 \bmod 8$ and $a\equiv c \bmod 4$}
\end{itemize} the lattice $\Lambda_2$ is unique up to isometry in this list (and in all other cases there are repetitions). This proves that $O(\Sigma_2)\subseteq O(\Lambda_2)$ in these cases, so that we have equality.

We now consider the four lattices from the third family. After another tedious calculation we find that the Gram matrices are equivalent to the following over $\mathbb{Z}_2$: \begin{align*}&\langle a,b,8c\rangle\\ &\langle a,b+2c,8(c+2b)\rangle\\ &\langle b,a+2c,8(c+2a)\rangle\\ &\langle a+2c,b+2c+4,8(c+2a+2b)\rangle\end{align*} and one sees that all have a two-dimensional unimodular component and a one-dimensional $8$-modular component (hence cannot be isometric to any of the previously considered lattices). 

When $a+b \equiv 0  \bmod 4$, we see that lattices one and four in this list are isometric (using the Oddity Fusion rule, see \cite{conway2013sphere} or \cite{allcock2020conway}), but that lattices two and three are unique up to isometry in the list. When $a+b \equiv 2 \bmod 4$ one instead finds that the lattices two and three are isometric but that lattices one and four are unique up to isometry in the list. In either case, these are the two lattices $\Pi_2,\Omega_2$ given in the statement of the proposition, and the uniqueness of Gram matrix up to isometry once again proves that $O(\Lambda_2)\subseteq O(\Pi_2)$ and $O(\Lambda_2) \subseteq O(\Omega_2)$.

In order to prove the reverse inclusion, consider the generic lattice in this family, i.e.\ the lattice with basis $\{v_1+\alpha v_3, v_2+\beta v_3, 2v_3\}$. There are seven superlattices of index $2$, having bases: \begin{align*}&\left\{\frac{v_1+\alpha v_3}{2},v_2+\beta v_3, 2v_3\right\} \qquad \left\{v_1 + \alpha v_3, \frac{v_2+\beta v_3}{2}-\frac{\gamma (v_1+\alpha v_3)}{2},2v_3\right\}\\ &\left\{v_1+\alpha v_3, v_2+\beta v_3, v_3 - \frac{\gamma(v_1+\alpha v_3)}{2}-\frac{\delta(v_2+\beta v_3)}{2}\right\}\end{align*} with $\gamma,\delta\in\{0,1\}$ and Gram matrices: \begin{align*}&\begin{pmatrix}
\frac{a+2c\alpha^2}{4} & c\alpha\beta & 2c\alpha\\ c\alpha\beta & b+2c\beta^2 & 4c\beta\\ 2c\alpha & 4c\beta & 8c
\end{pmatrix}\qquad \begin{pmatrix}
a+2c\alpha^2 & -\frac{a\gamma}{2}+c\alpha(\beta-\alpha\gamma) & 4c\\ -\frac{a\gamma}{2}+c\alpha(\beta-\alpha\gamma) & \frac{b}{4}+\frac{a\gamma^2+c(\beta-\alpha\gamma)^2}{2} & 2c(\beta-\alpha\gamma)\\ 4c & 2c(\beta-\alpha\gamma) & 8c
\end{pmatrix}\\ &\begin{pmatrix}
a+2c\alpha^2 & 2c\alpha\beta & -\frac{a\gamma}{2} - c\alpha(\alpha\gamma+\beta\delta-2)\\ 2c\alpha\beta & b+2c\beta^2 & -\frac{b\delta}{2}-c\beta(\alpha\gamma+\beta\delta-2)\\ -\frac{a\gamma}{2} - c\alpha(\alpha\gamma+\beta\delta-2) & -\frac{b\delta}{2}-c\beta(\alpha\gamma+\beta\delta-2) & \frac{a\gamma^2+b\delta^2}{4} + 2c\left(1-\frac{\alpha\gamma+\beta\delta}{2}\right)^2
\end{pmatrix}\end{align*} It is immediate that the Gram matrix corresponding to $(\gamma,\delta) = (0,0)$ is the only one that is integral (this lattice being isometric to the original lattice $\Lambda_2$). It follows that $O(\Pi_2)\subseteq O(\Lambda_2)$ and $O(\Omega_2)\subseteq O(\Lambda_2)$, so that equality follows in all of the relevant cases.}
\end{enumerate}
\end{proof}

\begin{prop}\label{2-misc}
Let $(V_2,Q)$ be a ternary quadratic space over $\mathbb{Q}_2$.

\begin{enumerate}[label = (\roman*)]
\item{Suppose that $(V_2,Q)$ contains a $2$-adic lattice $\Lambda_2\subset V_2$ with Gram matrix $\langle 1,-1,-2\rangle$. Then there exist $2$-adic sublattices $\Pi_2,\Omega_2\subset V_2$ of index $2$ with Gram matrices $\langle 1,-2,-4\rangle$ and $\langle 1,-1,-8\rangle$ and satisfying $O(\Pi_2) = O(\Omega_2)$.}
\item{Suppose that $(V_2,Q)$ contains a $2$-adic lattice $\Lambda_2\subset V_2$ with Gram matrix $\langle 1,1,2\rangle$. Then there exist $2$-adic sublattices $\Pi_2,\Omega_2\subset V_2$ of index $2$ with Gram matrices $\langle 1,2,4\rangle$ and $\langle 1,-5,-40\rangle$ satisfying $O(\Pi_2) = O(\Omega_2)$.}
\end{enumerate}
\end{prop}

\begin{proof} In each case, fix a basis $\{v_1,v_2,v_3\}$ for $\Lambda_2$ realising the corresponding Gram matrix.
\begin{enumerate}[label = (\roman*)]
\item{We may take $\Pi_2$ to be the lattice with basis $\{v_1,2v_2,v_3\}$ and $\Omega_2$ to be the lattice with basis $\{v_1,v_2,2v_3\}$. Consider the common index $2$ sublattice $\Sigma_2$ with basis $\{v_1,2v_2,2v_3\}$ and Gram matrix $\langle 1,-4,-8\rangle$. A similar calculation to those above shows that $\Sigma_2$ is the unique index sublattice of $\Pi_2$ with Gram matrix isometric to $\langle 1,-4,-8\rangle$, so that $O(\Pi_2)\subseteq O(\Sigma_2)$. Similarly we find that $O(\Omega_2)\subseteq O(\Sigma_2)$.

To show the converse, we compute the superlattices of $\Sigma_2$ of index $2$. There are once again seven of these, having bases \[\left\{\frac{1}{2}v_1,2v_2,2v_3\right\}\qquad \left\{v_1, v_2 - \frac{\alpha}{2}v_1,2v_3\right\}\qquad \left\{v_1, 2v_2, v_3-\frac{\alpha}{2}v_1-\beta v_2\right\}\] with $\alpha,\beta\in\{0,1\}$ and corresponding Gram matrices:\[\begin{pmatrix}
\frac{1}{4} & 0 & 0\\ 0 & -4 & 0\\ 0 & 0 & -8
\end{pmatrix}\qquad \begin{pmatrix} 1 & -\frac{\alpha}{2} & 0\\ -\frac{\alpha}{2} & -1+\frac{\alpha^2}{4} & 0\\ 0 & 0 & -8
\end{pmatrix} \qquad \begin{pmatrix}
1 & 0 & -\frac{\alpha}{2}\\ 0 & -4 & 2\beta\\ -\frac{\alpha}{2} & 2\beta & -2+\frac{\alpha^2}{4}-\beta^2
\end{pmatrix}.\] The lattices corresponding to $\alpha = 0$ and $(\alpha,\beta) = 0$ correspond to $\Omega_2$ and $\Pi_2$ respectively and it can be checked that these Gram matrices are unique up to isometry in this list. This proves that $O(\Sigma_2)\subseteq O(\Pi_2)$ and $O(\Sigma_2)\subseteq O(\Omega_2)$, so that the desired equality now follows.}
\item{We may take $\Pi_2$ to be the lattice with basis $\{v_1,2v_2,v_3\}$ and $\Omega_2$ to be the lattice with basis $\{v_1+v_3,v_2,2v_3\}$ (giving two Gram matrices that are equivalent to $\langle 1,2,4\rangle$ and $\langle 1,-5,-40\rangle$ over $\mathbb{Z}_2$). Consider the common index $2$ sublattice $\Sigma_2$ with basis $\{v_1+v_3,2v_2,2v_3\}$ and Gram matrix \[T = \begin{pmatrix}3 & 0 & 4\\ 0 & 4 & 0\\ 4 & 0 & 8\end{pmatrix}.\] The same argument as (i) then shows that $O(\Pi_2) = O(\Sigma_2) = O(\Omega_2)$ (i.e.\ $\Sigma_2$ is unique up to isometry in the list of index $2$ sublattices of $\Pi_2$ and $\Omega_2$, and $\Pi_2, \Omega_2$ are unique up to isometry in the list of index $2$ superlattices of $\Sigma_2$).}
\end{enumerate}
\end{proof}

\begin{rem}
It is tempting to try and prove the above result by using the fact that $\Pi_2,\Omega_2 \subseteq \Lambda_2$ are index $2$ sublattices. However, the proof of Proposition \ref{2-mod}(ii) shows that $O(\Pi_2) \neq O(\Lambda_2)$ and $O(\Omega_2)\neq O(\Lambda_2)$. We must instead descend to a common index $2$ sublattice $\Sigma_2$ to show equality of orthogonal groups.
\end{rem}

\begin{prop}\label{4-mod}
Let $a\in\{\pm 1,\pm 5\}$. Suppose that $(V_2,Q)$ is a ternary quadratic space over $\mathbb{Q}_2$ containing a $2$-adic lattice $\Lambda_2\subset V_2$ with Gram matrix $\left\langle\begin{pmatrix}
0 & 1\\ 1 & 0
\end{pmatrix},-4a\right\rangle$. Then the $2$-adic sublattice $\Pi_2\subset V_2$ with Gram matrix $\langle 2,-2,-4a\rangle$ satisfies $O(\Lambda_2) = O(\Pi_2)$.
\end{prop}

\begin{proof}
We omit the details of the proof since this is a very similar computation to the above (compute the sublattices/superlattices and compare isometry classes).
\end{proof}

\section{Genera of half-discriminant $p^2$}\label{2p2}

Let $(V,Q)$ be a positive definite ternary quadratic space over $\mathbb{Q}$, and $\Lambda \subset V$ be an integral lattice of discriminant $2p^2$ with $p$ prime.  We are interested in the spaces $\mathcal{S}(O(V),O(\widehat{\Lambda}),W_{2k}\otimes\theta_d)$ of orthogonal modular forms for various weights and the associated Hecke action coming from the $p$-neighbour operators $T_q$ (for $q\neq 2,p$). We start this task by restricting to $p$ an odd prime and classifying the possible local lattices $\Lambda_q$.

\vspace{0.1in}
\noindent\underline{\textbf{q $\neq$ 2,p}}

\vspace{0.1in}
\noindent Note that $\Lambda_q$ is unimodular if $q\neq 2,p$ and so \[\Lambda_q \sim \langle 1,1,\text{disc}(\Lambda_q)\rangle\sim \langle 1,1,2p^2\rangle\sim\langle 1,1,2\rangle\] 

\noindent\underline{\textbf{q = 2}}

\vspace{0.1in}
\noindent Here $\text{disc}(\Lambda_2) = [2p^2] = [2]\in\mathbb{Z}_2/(\mathbb{Z}_2^{\times})^2$ and a simple calculation gives the following three possibilities (with their Conway-Sloane symbols, Types and Hasse symbols):

\vspace{0.1in}
\begin{center}
\begin{tabular}{c|c|c|c}
$\Lambda_2$ & Symbol & Type & $S(V_2)$\\

\hline

$\langle 1,1,2\rangle$ & $[1^{+2}2^{+1}]_3$ & I & $1$\\

$\langle 1,-1,-2\rangle$ & $[1^{+2}2^{+1}]_7$ & I & $-1$\\

$\left\langle\begin{pmatrix}0 & 1\\ 1 & 0\end{pmatrix},-2\right\rangle$ & $1_{\text{II}}^{+2} 2^{+1}_7$ & II & $-1$
\end{tabular}
\end{center}

\vspace{0.2in}
\noindent\underline{\textbf{q = p}}

\vspace{0.1in}
\noindent Here we have: \[\text{disc}(\Lambda_p) = [2p^2] = \begin{cases} [p^2] & \text{ if } p\equiv 1,7\bmod 8\\ [rp^2] & \text{ if } p\equiv 3,5\bmod 8\end{cases}.\] A similar calculation gives the following possibilities (with $\varepsilon = (-1)^{\frac{p-1}{2}}$):

\vspace{0.1in}
\begin{center}
\begin{tabular}{c|c|c|c}
$\Lambda_p$ & Symbol & $\text{disc}(\Lambda_p)$ & $S(V_p)$\\

\hline
$\langle 1,p,p\rangle$ & $1^{+1}p^{+2}$ & $[p^2]$ & $\varepsilon$\\
$\langle r,p,pr\rangle$ & $1^{-1}p^{-2}$ & $[p^2]$ & $-\varepsilon$\\
$\langle 1,1,p^2\rangle$ & $1^{+2}(p^2)^{+1}$ & $[p^2]$ & $1$\\
$\langle 1,r,p^2r\rangle$ & $1^{-2}(p^2)^{-1}$ & $[p^2]$ & $1$\\
\hline
$\langle r,p,p\rangle$ & $1^{-1}p^{+2}$ & $[rp^2]$ & $\varepsilon$\\
$\langle 1,p,pr\rangle$ & $1^{+1}p^{-2}$ & $[rp^2]$ & $-\varepsilon$\\
$\langle 1,r,p^2\rangle$ & $1^{-2}(p^2)^{+1}$ & $[rp^2]$ & $1$\\
$\langle 1,1,p^2r\rangle$ & $1^{+2}(p^2)^{-1}$ & $[rp^2]$ & $1$
\end{tabular}
\end{center}

\vspace{0.1in}
Now we can classify the genera of such lattices. Since $(V,Q)$ is positive definite we have $S(V_{\infty}) = 1$. We also have $S(V_q) = 1$ for any $q\neq 2,p$, and so each possible genus for $\Lambda$ must give local lattices satisfying $S(V_2)S(V_p) = 1$ in order to satisfy the compatibility relations for Hasse symbols. Hence we must have $S(V_2) = S(V_p)$ and it is now clear that there are five possible genera of interest, with Gram matrices locally at $2$ and $p$ given below (depending on the class $p\bmod 8$):

\vspace{0.2in}
\begin{tabular}{c|c|c|c|c|c}
\multirow{2}{*}{Label} & \multirow{2}{*}{$\Lambda_2$} & \multicolumn{4}{c}{$\Lambda_p$}\\ \cline{3-6} & & $p\equiv 1 \bmod 8$ & $p\equiv 3 \bmod 8$ &  $p\equiv 5 \bmod 8$ & $p\equiv 7 \bmod 8$\\ \hline $1$ & $\left\langle\begin{pmatrix}0 & 1\\ 1 & 0\end{pmatrix},-2\right\rangle$ & \multirow{2}{*}{$\langle r,p,pr\rangle$} & \multirow{2}{*}{$\langle 2,p,p\rangle$} & \multirow{2}{*}{$\langle 1,p,2p\rangle$} & \multirow{2}{*}{$\langle 1,p,p\rangle$}\\ \cline{1-2} $2$ & $\langle 1,-1,-2\rangle$ & & & &\\ \hline $3$ & \multirow{3}{*}{$\langle 1,1,2\rangle$} & $\langle 1,p,p\rangle$ & $\langle 1,p,2p\rangle$ & $\langle 2,p,p\rangle$ & $\langle -1,p,-p\rangle$\\ \cline{3-6} $4$ & & $\langle 1,1,p^2\rangle$ & $\langle 1,2,p^2\rangle$ & $\langle 1,1,2p^2\rangle$ & $\langle 1,-1,-p^2\rangle$\\ \cline{3-6} $5$ & & $\langle 1,r,p^2r\rangle$ & $\langle 1,1,2p^2\rangle$ & $\langle 1,2,p^2\rangle$ & $\langle 1,1,p^2\rangle$
\end{tabular}
\vspace{0.2in}

For a representative lattice $\Lambda\subset V$ coming from each of the genera above, we will either prove or conjecture that there is a Hecke-equivariant isomorphism between $\mathcal{S}(O(V), O(\widehat{\Lambda}),W_{2k}\otimes\theta_d)$ and a space of classical modular forms. 

\subsection{Genus $1$}

This is the only even genus of discriminant $\text{disc}(\Lambda) = 2p^2$. Before proving the relevant transfer result we introduce the linear map $\hat{\psi}_{(p)}\in\text{GL}(\widehat{V})$ with components: \begin{align*}\psi_{(p),q} = \begin{cases}\text{diag}(p,1,1) & \text{ if } q = p\\ \text{id} & \text{ otherwise }\end{cases}\end{align*} where we note that at $p$ this map depends on a choice of basis for $V_p$ (which will usually be a basis for $\Lambda_p$ realising a given Gram matrix).

\begin{thm}\label{genus1} Let $\Lambda$ be a lattice representing Genus $1$. Then for each $k\geq 0$ and $d\in\{1,p\}$ there exists a Hecke-equivariant isomorphism (for the algebra generated by the Hecke operators $T_q$ with $q\neq p$ prime): \[\mathcal{S}(O(V),O(\widehat{\Lambda}),W_{2k}\otimes \theta_d) \cong S_{2k+2}^{\textup{new}}\left(\Gamma_0(p),\varepsilon^{(p)}_p\varepsilon^{(p)}_d\right).\]
\end{thm}

\begin{proof}
We first pick a basis of $\Lambda_p$ realising the diagonal Gram matrix given in the table above and then apply $\hat{\psi}_{(p)}$. This gives us an even lattice $\Pi = V\cap \hat{\psi}_{(p)}(\widehat{\Lambda})$ of discriminant $2p^4$, which we may view as an even lattice of discriminant $2p$ in $(V,Q^{\left(\frac{1}{p}\right)})$ with local lattices: \[\Pi_q \sim\begin{cases} \langle p^{-1}, p^{-1}, 2p^{-1}\rangle \sim \langle 1,1,2p\rangle & \text{ if } q \nmid 2p\\ \left\langle \begin{pmatrix}0 & p^{-1}\\ p^{-1} & 0\end{pmatrix}, -2p^{-1}\right\rangle \sim  \left\langle\begin{pmatrix}0 & 1\\ 1 & 0\end{pmatrix}, -2p\right\rangle & \text{ if } q = 2\\ \langle 1,r,pr\rangle & \text{ if } q = p\equiv 1 \bmod 8\\ \langle 1,1,pr\rangle & \text{ if } q = p\equiv 3 \bmod 8\\ \langle 1,r,p\rangle & \text{ if } q=p\equiv 5 \bmod 8\\ \langle 1,1,p\rangle & \text{ if } q=p\equiv 7 \bmod 8\end{cases}\]

We have that $O(\widehat{\Lambda}) = O(\widehat{\Pi})$ (since $O(\Lambda_p) = O(\Pi_p)$ by Proposition \ref{p-mod} and $O(\Lambda_q) = O(\Pi_q)$ when $q\neq p$), and both Lemma \ref{samespace} and Lemma \ref{scaling} then imply an equality of Hecke modules: \[\mathcal{S}(O(V),O(\widehat{\Lambda}),W) = \mathcal{S}(O(V),O(\widehat{\Pi}),W)\] for every representation $W$ of $O(V)$.

By the recipe outlined in Subsection \ref{Ternary Lattices and Quaternion orders} we find that the even Clifford algebra of $(V,Q^{(\frac{1}{p})})$ is a definite quaternion algebra $B = \text{Cl}^0(V)$ of discriminant $p$. The order $\mathcal{O} = \text{Cl}^0(\Pi)$ has discriminant $\text{discrd}(\mathcal{O}) = p$ and so is maximal, hence residually unramified everywhere with $\varepsilon_p(\mathcal{O}) = -1$. The result then follows from the results of Hein-Tornaria-Voight (see Theorem \ref{thm: 6.12 in HTV}).
\end{proof}

\begin{example}
We give an example of Theorem \ref{genus1}, performing the computations using the Algebraic Modular Forms package in Magma \cite{magma}. We refer to \cite{hein2016orthogonal} for a more in-depth discussion of the practical computational details.

Let $\Lambda = \mathbb{Z}e_1 \oplus \mathbb{Z}e_2 \oplus \mathbb{Z}e_3\subset V$ be an even lattice with Gram matrix \[T = \begin{pmatrix}
		4 & 1 & -1\\ 
        1 & 28 & 9\\ 
        -1 & 9 & 28
	\end{pmatrix}.\] Here, $\text{disc}(\Lambda) = 2738 = 2 \cdot 37^2$. Computations with $p$-neighbours show that $\text{Cl}(\Lambda) = \{[\Lambda_1], [\Lambda_2]\}$ is of size two. The Hecke operators $T_p$ acting on the $2$-dimensional space of trivial weight orthogonal modular forms $\mathcal{M}(O(V), O(\widehat{\Lambda}), \mathbb{C})$ are given by: \begin{align*}
        T_2= \begin{pmatrix}
				1 & 2\\ 
                1 & 2
			\end{pmatrix},\quad T_3= \begin{pmatrix}
				2 & 2\\ 
                1 & 3
			\end{pmatrix},\quad T_5 = \begin{pmatrix}
				2 & 4\\
                2 & 4
			\end{pmatrix},\quad T_7 = \begin{pmatrix}
				2 & 6\\
                3 & 5
			\end{pmatrix},\quad \dots
        \end{align*}
A quick computation then reveals that the eigenforms and eigenvalues are given by:$$
\begin{tabular}{l l l l l l}
			$v_1 = [1,1]$ & $\gamma_{2} = 3$& $\gamma_{3} = 4$& $\gamma_{5} = 6$& $\gamma_{7} = 8$& $\dots$ \\
			$v_2 = [2,-1]$ &  $\lambda_{2} = 0$&  $\lambda_{3} = 1$&  $\lambda_{5} = 0$&  $\lambda_{7} = -1$& $\dots$
		\end{tabular}$$

Note that the first eigenform is the Eisenstein eigenform with eigenvalues $\lambda_p = p+1$. The other eigenform forms a basis for the $1$-dimensional space $\mathcal{S}(O(V), O(\widehat{\Lambda}), \mathbb{C})$ and has eigenvalues $\lambda_p = a_p$ (for $p \neq 37$), where: 
$$f = \sum_{n\geq 1} a_n q^n = q+q^3-2q^4-q^7-2q^9+3q^{11}-2q^{12}-4q^{13}+ \dots $$
is the $q$-expansion of the unique eigenform in $S_2^{\textup{new}}(\Gamma_0(37),-1)$ (having LMFDB label \href{https://www.lmfdb.org/ModularForm/GL2/Q/holomorphic/37/2/a/b/}{37.2.a.b}).

Note that there exists a newform $f_2\in S_2^{\textup{new}}(\Gamma_0(37),+1)$ with LMFDB label \href{https://www.lmfdb.org/ModularForm/GL2/Q/holomorphic/37/2/a/a/}{37.2.a.a}. In accordance with Theorem \ref{genus1}, we do not see this at trivial weight (but it can be seen if the weight is twisted by $\theta_{37}$). The same scenario presents itself also in higher weights.
\end{example}

\begin{rem}
One can roughly interpret the above result as telling us that spaces of orthogonal modular forms coming from Genus $1$ admit no newforms. All of the eigenforms that arise are explained by a genus of smaller discriminant $2p$ and so give rise to modular forms of level $p$ (as opposed to level $p^2$). This analogy is a good one to keep in mind for later, but we note that it is not completely accurate since although the lattices $\Lambda$ and $\Pi$ are different, the level structures $O(\widehat{\Lambda})$ and $O(\widehat{\Pi
})$ are equal (in contrast to the congruence subgroups $\Gamma_0(p^2)$ and $\Gamma_0(p)$ being unequal on the modular side).
\end{rem}

\subsection{Genus $2$}

The other four genera consist of odd lattices of discriminant $2p^2$, and we will see that Genus $2$ behaves similarly to Genus $1$. We introduce the linear map $\hat{\psi}_{(2)}\in\text{GL}(\widehat{V})$ with components: \[\psi_{(2),q} = \begin{cases}\begin{pmatrix}\frac{1}{2}&\frac{1}{2}&0\\\frac{1}{2}&-\frac{1}{2}&0\\0&0&1\end{pmatrix} & \text{ if } q = 2\\ \text{ id} & \text{ otherwise }\end{cases}.\] This map depends on a choice of basis for $V_2$ (which will usually be a basis for $\Lambda_2$ realising a given Gram matrix).

\begin{thm}\label{genus2} Let $\Lambda$ be a lattice representing Genus $2$. Then for each $k\geq 0$ and $d\in\{1,2,p,2p\}$ there exists a Hecke-equivariant isomorphism (for the algebra generated by the Hecke operators $T_q$ with $q\neq 2, p$ prime): \[\mathcal{S}(O(\widehat{\Lambda}),W_{2k}\otimes \theta_d) \cong S_{2k+2}^{p\textup{-new}}\left(\Gamma_0(2p),\varepsilon^{(2p)}_p\varepsilon^{(2p)}_d\right).\]
\end{thm}

\begin{proof}
We again pick a basis of $\Lambda_p$ realising the diagonal Gram matrix given in the table above and then apply $\hat{\psi}_{(p)}$. This gives us an odd lattice $\Pi = V\cap \hat{\psi}_{(p)}(\widehat{\Lambda})$ of discriminant $2p^4$, which we may view as an odd lattice of discriminant $2p$ in $(V,Q^{\left(\frac{1}{p}\right)})$ with local lattices: \[\Pi_q \sim\begin{cases} \langle p^{-1}, p^{-1}, 2p^{-1}\rangle \sim \langle 1,1,2p\rangle & \text{ if } q \nmid 2p\\ \left\langle p^{-1},-p^{-1},-2p^{-1}\right\rangle \sim  \left\langle p,-p,-2p\right\rangle & \text{ if } q = 2\\ \langle 1,r,pr\rangle & \text{ if } q = p\equiv 1 \bmod 8\\ \langle 1,1,pr\rangle & \text{ if } q = p\equiv 3 \bmod 8\\ \langle 1,r,p\rangle & \text{ if } q=p\equiv 5 \bmod 8\\ \langle 1,1,p\rangle & \text{ if } q=p\equiv 7 \bmod 8\end{cases}\]

We have that $O(\widehat{\Lambda}) = O(\widehat{\Pi})$ (since $O(\Lambda_p) = O(\Pi_p)$ by Proposition \ref{p-mod} and $O(\Lambda_q) = O(\Pi_q)$ when $q\neq p$), and both Lemma \ref{samespace} and Lemma \ref{scaling} then imply an equality of Hecke modules: \[\mathcal{S}(O(V),O(\widehat{\Lambda}),W) = \mathcal{S}(O(V),O(\widehat{\Pi}),W)\] for every representation $W$ of $O(V)$.

Note that we cannot apply the results of Hein-Tornaria-Voight yet, since they require an even lattice (their results implicitly require an integral lattice with respect to the quadratic form $Q$, as opposed to the bilinear form $T$). 

 Scaling the quadratic form by $2$ lets us view $\Pi$ as an even lattice of discriminant $16p$ in $(V,Q^{\left(\frac{2}{p}\right)})$, with local lattice $\Pi_2\sim \langle 2p,-2p,-4p\rangle \sim \langle 2,-2,-4p\rangle$. However, further reductions of the discriminant are possible. The lattice $\Omega = V\cap \hat{\psi}_{(2)}(\widehat{\Pi})$ is an even lattice of discriminant $4p$ in $(V,Q^{\left(\frac{2}{p}\right)})$, with local lattices \begin{align*}\Omega_q\sim\begin{cases} \langle2,2,4p\rangle \sim \langle 1,1,p\rangle\qquad & \text{ if } q\nmid 2p\\ 
\left\langle \begin{pmatrix}0 & 1\\ 1 & 0\end{pmatrix},-4p\right\rangle & \text{ if } q = 2\\ \langle 2,2r,2pr\rangle\sim \langle 1,r,pr\rangle &\text{ if } p\equiv 1 \bmod 8\\ \langle 2,2,2pr\rangle
    \sim \langle 1,1,p\rangle &\text{ if } p\equiv 3 \bmod 8\\
    \langle 2,4,2p\rangle\sim\langle 1,2,2p\rangle &\text{ if } p\equiv 5 \bmod 8\\
    \langle 2,2,2p\rangle \sim \langle 1,1,p\rangle &\text{ if } p\equiv 7 \bmod 8\\
\end{cases}\end{align*}
 
We have that $O(\widehat{\Omega}) = O(\widehat{\Pi})$ (since $O(\Omega_2) = O(\Pi_2)$ by Proposition \ref{4-mod} and $\Pi_q = \Omega_q$ when $q\neq 2$). Another application of Lemma \ref{samespace} and Lemma \ref{scaling} now implies that there is an equality of Hecke modules: \[\mathcal{S}(O(V),O(\widehat{\Pi}),W) = \mathcal{S}(O(V),O(\widehat{\Omega}),W)\] for every representation $W$ of $O(V)$.

By the recipe outlined in Subsection \ref{Ternary Lattices and Quaternion orders} we find that the even Clifford algebra of $(V,Q^{(\frac{2}{p})})$ is a definite quaternion algebra $B = \text{Cl}^0(V)$ of discriminant $p$. The order $\mathcal{O} = \text{Cl}^0(\Omega)$ has discriminant $\text{discrd}(\mathcal{O}) = 2p$ and so is Eichler, hence residually unramified everywhere with $\varepsilon_2(\mathcal{O}) = 1$ and $\varepsilon_p(\mathcal{O}) = -1$. The result then follows from the results of Hein-Tornaria-Voight (see Theorem \ref{thm: 6.12 in HTV}).
\end{proof}

\begin{example}
We provide an example in non-trivial weight. Let $\Lambda\subset V$ be an odd lattice with Gram matrix \[T = \begin{pmatrix}
		2 & 0 & 0\\ 
        0 & 11 & 0\\ 
        0 & 0 & 11
	\end{pmatrix}.\]
    One can easily check that $\Lambda$ is a representative for Genus $2$ and that it has discriminant $\text{disc}(\Lambda) = 242 = 2 \cdot 11^2$.
    
    At weight $W_8$ we compute that $\text{dim}(\mathcal{S}(O(V),O(\widehat{\Lambda}),W_8)) = 7$, and the characteristic polynomials of $T_p$ (for $p\neq 11$) have irreducible factors $F_{p,1}(x)$ and $F_{p,2}(x)$ as follows: 

\begin{align*}
    F_{2,1}(x) &= x^5 - 16x^4 - 1506x^3 + 6428x^2 + 619552x + 4247232\\ F_{2,2}(x) &= (x + 16)^2\\ F_{3,1}(x) &= x^5 - 112x^4 - 80330x^3 + 6248892x^2 + 1106354961x + 6873235452\\ F_{3,2}(x) &= x^2 + 21x - 17892\\ F_{5,1}(x) &= x^5 - 1594x^4 - 6694866x^3 + 8153795144x^2 + 11912315141065x - 8650710939654150\\ F_{5,2}(x) &= x^2 + 521x - 2023290\\ &\vdots
\end{align*}

We match these with Galois orbits of newforms $f \in S_{10}^{\text{new}}(\Gamma_0(11))$ and $g \in S_{10}^{\text{new}}(\Gamma_0(22))$ respectively (with LMFDB labels \href{https://www.lmfdb.org/ModularForm/GL2/Q/holomorphic/11/10/a/b/}{11.10.a.b} and \href{https://www.lmfdb.org/ModularForm/GL2/Q/holomorphic/22/10/a/d/}{22.10.a.d}). These forms are both $11$-new and have Atkin-Lehner eigenvalues $-1$ at $11$ and in the second case $+1$ at $2$ (as predicted by Theorem \ref{genus2}). Note that there exist other newforms of levels $11$ and $22$ but they have incorrect Atkin-Lehner eigenvalues, hence we do not see them in this calculation (but we would if we twisted the weight by the characters $\theta_2,\theta_{11}$ and $\theta_{22}$).
\end{example}

\subsection{Genus $3$}

For a lattice $\Lambda$ in Genus $3$ we may once again apply the map $\hat{\psi}_{(p)}$ to give an odd lattice $\Pi = V\cap \hat{\psi}_{(p)}(\widehat{\Lambda})$ of discriminant $2p^4$, which may be viewed as an odd lattice of discriminant $2p$ in $(V,Q^{(\frac{1}{p})})$. Once again, we have $O(\widehat{\Lambda}) = O(\widehat{\Pi})$ by Proposition \ref{p-mod} and Lemma \ref{samespace} and Lemma \ref{scaling} give an equality of Hecke modules: \[\mathcal{S}(O(V),O(\widehat{\Lambda}),W) = \mathcal{S}(O(V),O(\widehat{\Pi}),W)\] for every representation $W$ of $O(V)$.

We must again address the fact that $\Pi$ is odd. There are two potential ways to deal with this. One way is to apply the map $\hat{\psi}_{(2)}^{-1}$, giving an even lattice $\Sigma = V\cap\hat{\psi}_{(2)}^{-1}(\widehat{\Pi})$ of discriminant $8p$ in $(V,Q^{\left(\frac{1}{p}\right)})$. The local $2$-adic lattice of $\Sigma$ satisfies:

\begin{align*}\Sigma_2 &\sim  \langle 2p, 2p, 2p \rangle \sim\begin{cases}\langle 2,2,2\rangle & \text{ if } p\equiv 1 \bmod 8\\ \langle -10,-10,-10\rangle & \text{ if } p\equiv 3 \bmod 8\\ \langle 10,10,10\rangle & \text{ if } p\equiv 5 \bmod 8\\ \langle -2,-2,-2\rangle & \text{ if } p\equiv 7 \bmod 8\\ \end{cases}\end{align*} However, unlike the previous genus, Proposition \ref{2-mod}(ii) can only provide us with an inclusion $O(\Pi_2)\subseteq O(\Sigma_2)$, which induces a Hecke-equivariant inclusion: \[\mathcal{S}(O(V),O(\widehat{\Sigma}),W) \longhookrightarrow \mathcal{S}(O(V),O(\widehat{\Pi}),W)\] for every representation $W$ of $O(V)$. We are unable to conclude that $O(\Pi_2) = O(\Sigma_2)$, since it simply is not true (e.g. see Table $4.1$ of Allcock \cite{allcock2025}). We are also able to observe this fact by direct computation as follows.

\begin{example}
Consider the lattices $\Lambda$ and $\Sigma$ with Gram matrices
\begin{align*}
    T_{\Lambda} = \begin{pmatrix} 1& 0 & 0 \\ 0& 11 & 0 \\ 0& 0& 22
    \end{pmatrix} \quad \text{and} \quad & T_{\Sigma} = \begin{pmatrix} 2& 0 & 0 \\ 0& 4 & 2 \\ 0& 2& 12
    \end{pmatrix}
\end{align*}
and discriminants $\text{disc}(\Lambda) = 242 = 2\cdot 11^2$ and $\text{disc}(\Sigma) = 88 = 8\cdot 11$ respectively. It is easily checked that these lattices match the description of the two lattices of interest. A quick computation reveals that $|\text{Cl}(\Lambda)| = 2$ and $|\text{Cl}(\Sigma)| = 1$, showing that $O(\widehat{\Lambda}) \neq O(\widehat{\Sigma})$ (the issue must then be that $O(\Lambda_2) \neq O(\Sigma_2)$ since the local orthogonal groups are equal at all other primes). Further computation reveals that the characteristic polynomial of $T_p$ acting on the second class set always divides the one coming from the action on the first class set, in agreement with the embedding above.
\end{example}

As a consequence of the above discussion, we conclude that for Genus $3$ we have both an oldform subspace (coming from orthogonal modular forms for $\Sigma$) and a newform subspace.

The second way to proceed is to simply scale the quadratic map by $2$ and view $\Pi$ as an even lattice in $(V,Q^{\left(\frac{2}{p}\right)})$. This space gives rise to the quaternion algebra $B = \text{Cl}^0(V)$ of discriminant $\text{disc}(B) = 2$, and the corresponding order $\mathcal{O} = \text{Cl}^0(\Pi)$ is an order of discriminant $8p$ with $\varepsilon_p(\mathcal{O}) = 1$ and $\varepsilon_2(\mathcal{O}) = 0$ (i.e. it is residually ramified at $2$). This complicates matters, since we are now unable to use the results of Hein-Tornaria-Voight. However, computational evidence suggests that the following conjecture is true.

\begin{conj}\label{genus3}
Let $\Lambda$ be a lattice representing Genus $3$. Then for each $k\geq 0$ and $d\in\{1,2,p,2p\}$ there exists a Hecke-equivariant isomorphism (for the algebra generated by the Hecke operators $T_q$ with $q\neq 2,p$ prime): \[\mathcal{S}(O(\widehat{\Lambda}),W_{2k}\otimes \theta_d) \cong S_{2k+2}^{2\textup{-new}}\left(\Gamma_0(4p),\varepsilon^{(4p)}_2\varepsilon^{(4p)}_d\right).\]
\end{conj}

\begin{example}
Let $\Lambda\subset V$ be an odd lattice representing Genus $3$ with Gram matrix \[T = \begin{pmatrix}
		2 & 0 & 0\\ 
        0 & 5 & 0\\ 
        0 & 0 & 5
	\end{pmatrix}\]
    and discriminant $\text{disc}(\Lambda) = 50 = 2 \cdot 5^2$. As in previous examples, the class set is trivial and so we must use higher weights in order to see any non-trivial modular data. At weight $W_8$ we compute that $\text{dim}(\mathcal{S}(O(V),O(\widehat{\Lambda}),W_8)) = 5$. The four characteristic polynomials of $T_p$ (for $p\neq 5$) have irreducible factors $F_{p,1}(x), F_{p,2}(x),F_{p,3}(x)$ and $F_{p,4}(x)$ as follows: 
\[
\renewcommand{\arraystretch}{1.25}
\begin{array}{c|llll}
 p & F_{p,1}(x) & F_{p,2}(x) & F_{p,3}(x) & F_{p,4}(x) \\ \hline
 2 & x-16  & x      & x-16   & x^2 \\
 3 & x+156 & x-228  & x-174  & x^2+260x-3324 \\
 7 & x+952 & x+6328 & x-4658 & x^2+380x-96250364
\end{array}
\]

We match these with newforms $f \in S_{10}^{\text{new}}(\Gamma_0(2)), g \in S_{10}^{\text{new}}(\Gamma_0(4)), h \in S_{10}^{\text{new}}(\Gamma_0(10))$ and a Galois orbit of newforms $j \in S_{10}^{\text{new}}(\Gamma_0(20))$ respectively (with LMFDB labels \href{https://www.lmfdb.org/ModularForm/GL2/Q/holomorphic/2/10/a/a/}{2.10.a.a},  \href{https://www.lmfdb.org/ModularForm/GL2/Q/holomorphic/4/10/a/a/}{4.10.a.d}, \href{https://www.lmfdb.org/ModularForm/GL2/Q/holomorphic/10/10/a/c/}{10.10.a.c} and \href{https://www.lmfdb.org/ModularForm/GL2/Q/holomorphic/20/10/a/b/}{20.10.a.b}). These are precisely the forms that are $2$-new of level $20$ and have Atkin-Lehner eigenvalues $-1$ at $2$ and $+1$ at $5$ (as predicted by Theorem \ref{genus3}). Note that there exist other newforms of levels $10$ and $20$ but they have incorrect Atkin-Lehner eigenvalues and we would reach them via twisting the weight. 

Additionally, we point out that no forms of level $40$ occur (which might be expected to occur, given that the corresponding order has $\text{discrd}(\mathcal{O}) = 8p = 40$). In particular, the form \href{https://www.lmfdb.org/ModularForm/GL2/Q/holomorphic/40/10/a/c/}{40.10.a.c} has the correct Atkin-Lehner eigenvalues but does not feature in this computation.
\end{example}

\subsection{Genus $4$ and Genus $5$}

The remaining two genera are even more problematic. Firstly, we are unable to apply the map $\hat{\psi}_{(p)}$ and reduce the power of $p$ in the discriminant. One might be tempted to try other transformations, e.g.\ go from a local Gram matrix of the form $\langle 1,a,p^2b\rangle$ to $\langle 1,a,b\rangle$ by scaling the third basis vector by $\frac{1}{p}$, but no such transformations seem to preserve the orthogonal group (e.g.\ see Table $4.1$ of Allcock \cite{allcock2025}). This issue might not be so much of a problem since the corresponding quaternion algebra $B = \text{Cl}^0(V)$ has discriminant $2$ in this case, and the two local orders at $p$ are unramified with Eichler invariants $\varepsilon_p(\mathcal{O}_p) = 1$ or $\varepsilon_p(\mathcal{O}_p) = -1$ according to whether $\Lambda$ is in Genus $4$ or Genus $5$ respectively (the first order is Eichler and the second is not). 

The real issue is instead at the prime $2$, since all attempts by the authors to turn $\Lambda$ into an even lattice resulted in quaternion orders that are residually ramified at $2$, i.e.\ have $\varepsilon_2(\mathcal{O}) = 0$. Proving results here would require a more intimate knowledge of the Jacquet-Langlands correspondence for such orders, which the authors were unable to find in the literature. It would also require an explicit and precise understanding of the transfer from orthogonal modular forms for $\Lambda$ to quaternionic modular forms for $\mathcal{O}$ via the even Clifford functor. For example, in \cite{heintornariavoight} it is proved that $\mathcal{S}(O(V),O(\widehat{\Lambda}),W)$ can be identified with a space of quaternionic modular forms for the order $\mathcal{O}$, i.e.\ the fixed subspace under the action of the Atkin-Lehner group $\text{AL}(\mathcal{O}) = \prod_{p|\text{discrd}(\mathcal{O})}(N_{B_p^{\times}}(\mathcal{O}_p)/\mathbb{Q}_p^{\times}\mathcal{O}_p^{\times})$. The fact that this group is not generally isomorphic to $\prod_{p|\text{discrd}(\mathcal{O})}(\mathbb{Z}/2\mathbb{Z})$ makes it trickier to generalise the results of Hein-Tornaria-Voight.

The best that the authors are able to offer at present are conjectures based on computational evidence.

\begin{conj}\label{genus4} Let $\Lambda$ be a lattice representing Genus $4$. Then for each $k\geq 0$ and $d\in\{1,2,p,2p\}$ there exists a Hecke-equivariant isomorphism (for the algebra generated by the Hecke operators $T_q$ with $q\neq 2,p$ prime) so that:
$$ \mathcal{S}(O(\widehat{\Lambda}),W_{2k}\otimes \theta_d) \cong S^{2\textup{-new}}_{2k+2}\left(\Gamma_0(4p^2),\varepsilon_2^{(4p^2)}\varepsilon_d^{(4p^2)}\right).$$
\end{conj}

\begin{example}
Let $\Lambda$ be a lattice representing Genus $4$ with Gram matrix
\[T = \begin{pmatrix}
        1 & 0 & 0\\
        0 & 2 & 0\\
        0 & 0 & 121
    \end{pmatrix},\]
so that $\text{disc}(\Lambda) = 242 = 2\cdot 11^2$. The class set is trivial, and at weight $W_4$ we compute
$\dim(\mathcal{S}(O(V),O(\widehat{\Lambda}),W_4)\big) = 45$. The characteristic polynomials of $T_3$, $T_5$ and $T_7$ split in the following way and match Galois orbits of newforms (for forms of dimension at least $3$ we only give the degree of the factor):
\[
\renewcommand{\arraystretch}{1.15}
\begin{array}{l|l|l|l|c}
F_{3,i}(x) & F_{5,i}(x) & F_{7,i}(x) & \text{newform orbit} & (w_2,w_{11})\\ \hline
(x+12)^2 & (x-54)^2 & (x+88)^2 & \href{https://www.lmfdb.org/ModularForm/GL2/Q/holomorphic/4/6/a/a/}{4.6.a.a} & (-,*)\\ \hline
x+29 & x+31 & x+230 & \href{https://www.lmfdb.org/ModularForm/GL2/Q/holomorphic/22/6/a/c/}{22.6.a.c} & (-,-)\\
x^2-29x+12 & x^2+13x-4914 & x^2+14x-7088 & \href{https://www.lmfdb.org/ModularForm/GL2/Q/holomorphic/22/6/a/d/}{22.6.a.d} & (-,+)\\
x-7 & x+79 & x+50 & \href{https://www.lmfdb.org/ModularForm/GL2/Q/holomorphic/44/6/a/a/}{44.6.a.a} & (-,-)\\
x^2+6x-487 & x^2+22x-1863 & x^2-268x+13492 & \href{https://www.lmfdb.org/ModularForm/GL2/Q/holomorphic/44/6/a/b/}{44.6.a.b} & (-,+)\\ \hline
\deg 3 & \deg 3 & \deg 3 & \href{https://www.lmfdb.org/ModularForm/GL2/Q/holomorphic/242/6/a/j/}{242.6.a.j} & (-,+)\\
\deg 4 & \deg 4 & \deg 4 & \href{https://www.lmfdb.org/ModularForm/GL2/Q/holomorphic/242/6/a/n/}{242.6.a.n} & (-,+)\\
\deg 6 & \deg 6 & \deg 6 & \href{https://www.lmfdb.org/ModularForm/GL2/Q/holomorphic/242/6/a/p/}{242.6.a.p} & (-,+)\\
(x+15)^2 & (x-9)^2 & x^2-44256 & \href{https://www.lmfdb.org/ModularForm/GL2/Q/holomorphic/484/6/a/c/}{484.6.a.c} & (-,+)\\
x^2-31x+232 & x^2+57x-6126 & x^2 & \href{https://www.lmfdb.org/ModularForm/GL2/Q/holomorphic/484/6/a/e/}{484.6.a.e} & (-,+)\\
\deg 10 & \deg 10 & \deg 10 & \href{https://www.lmfdb.org/ModularForm/GL2/Q/holomorphic/484/6/a/h/}{484.6.a.h} & (-,+)\\
(\deg 5)^2 & (\deg 5)^2 & \deg 10 & \href{https://www.lmfdb.org/ModularForm/GL2/Q/holomorphic/484/6/a/j/}{484.6.a.j} & (-,+)
\end{array}
\]
These are precisely the forms that we expect to see given Conjecture \ref{genus4} (for example, we do not see newforms with Atkin-Lehner sign $+1$ at 2).
Note that the newform of level $4$ appears twice and that we see newforms of levels $22$ and $44$ with Atkin-Lehner signs $(-1,-1)$. Both of these are new phenomena, but are to be expected since oldforms of levels $N$ and $Np$ can lift in multiple ways to give eigenforms of level $Np^2$ with a variety of Atkin-Lehner signs. We did not observe this behaviour earlier since our levels were generally square-free.

Increasing the weight to $W_8$ allows us to also see the level $2$ newform of weight $10$  (\href{https://www.lmfdb.org/ModularForm/GL2/Q/holomorphic/2/10/a/a/}{2.10.a.a}) with multiplicity two. This justifies the need for the $2$-new subspace as opposed to some other proper subspace.
\end{example}

\begin{conj}\label{genus5} Let $\Lambda$ be a lattice representing Genus $5$. Then for each $k\geq 0$ and $d\in\{1,2,p,2p\}$ there exists a Hecke-equivariant isomorphism (for the algebra generated by the Hecke operators $T_q$ with $q\neq 2,p$ prime) so that:
$$\mathcal{S}(O(\widehat{\Lambda}),W_{2k}\otimes \theta_d) \cong S^{2\textup{-new}}_{2k+2}\left(\Gamma_0(4),\varepsilon^{(4)}_2\varepsilon^{(4)}_d\right)\oplus S^{2p^2\textup{-new}}_{2k+2}\left(\Gamma_0(4p^2),\varepsilon^{(4p^2)}_2\varepsilon^{(4p^2)}_d\right).$$
\end{conj}

\begin{example}
Let $\Lambda$ be a lattice representing Genus $5$ with Gram matrix
\[T = \begin{pmatrix}
        1 & 0 & 0\\
        0 & 2 & 0\\
        0 & 0 & 25
    \end{pmatrix},\]
so that $\text{disc}(\Lambda) = 50 = 2\cdot 5^2$. The class set is trivial and at weight $W_8$ we compute
$\dim\big(\mathcal{S}(O(V),O(\widehat{\Lambda}),W_8)\big) = 13$. The Hecke operators $T_3$ and $T_7$ split as follows:
\begin{align*}
    F_{3,1}(x) &= x + 156, & F_{7,1}(x) &= x + 952,\\
    F_{3,2}(x) &= x - 228, & F_{7,2}(x) &= x + 6328,\\
    F_{3,3}(x) &= x + 46, & F_{7,3}(x) &= x - 10318,\\
    F_{3,4}(x) &= x - 204, & F_{7,4}(x) &= x + 5432,\\
    F_{3,5}(x) &= x^2 - 68x - 24069, & F_{7,5}(x) &= x^2 - 56x - 44496116,\\
    F_{3,6}(x) &= x + 228, & F_{7,6}(x) &= x - 6328,\\
    F_{3,7}(x) &= x - 48, & F_{7,7}(x) &= x - 532,\\
    F_{3,8}(x) &= x^2 - 260x - 3324, & F_{7,8}(x) &= x^2 - 380x - 96250364,\\
    F_{3,9}(x) &= x^3 + 73x^2 - 25209x - 1912257, & F_{7,9}(x) &= x^3 + 1482x^2 - 107965284x + 170184217112.
\end{align*}
The factors $F_{q,i}$ can be matched with Galois orbits of newforms with respective LMFDB symbols
    \href{https://www.lmfdb.org/ModularForm/GL2/Q/holomorphic/2/10/a/a/}{2.10.a.a},
    \href{https://www.lmfdb.org/ModularForm/GL2/Q/holomorphic/4/10/a/a/}{4.10.a.a},
    \href{https://www.lmfdb.org/ModularForm/GL2/Q/holomorphic/50/10/a/d/}{50.10.a.d},
    \href{https://www.lmfdb.org/ModularForm/GL2/Q/holomorphic/50/10/a/f/}{50.10.a.f},
    \href{https://www.lmfdb.org/ModularForm/GL2/Q/holomorphic/50/10/a/j/}{50.10.a.j},
    \href{https://www.lmfdb.org/ModularForm/GL2/Q/holomorphic/100/10/a/a/}{100.10.a.a},
    \href{https://www.lmfdb.org/ModularForm/GL2/Q/holomorphic/100/10/a/b/}{100.10.a.b},
    \href{https://www.lmfdb.org/ModularForm/GL2/Q/holomorphic/100/10/a/c/}{100.10.a.c} and
    \href{https://www.lmfdb.org/ModularForm/GL2/Q/holomorphic/100/10/a/d/}{100.10.a.d}. All of these forms have Atkin-Lehner sign $w_2=-1$, and the forms of level $50$ and $100$ have $w_5=+1$.
Moreover, they are exactly the newforms that we expect to see in Conjecture \ref{genus5}. No newforms of level
$10$ or $20$ occur for this genus. For example, the newforms
\href{https://www.lmfdb.org/ModularForm/GL2/Q/holomorphic/10/10/a/c/}{10.10.a.c} and
\href{https://www.lmfdb.org/ModularForm/GL2/Q/holomorphic/20/10/a/b/}{20.10.a.b},
do not appear, even though their Atkin-Lehner signs are $(w_2,w_5)=(-,+)$. As with the previous example, it is also possible to see some forms with Atkin-Lehner signs $(w_2,w_5) = (-1,-1)$, without needing to twist the weight. This is all in agreeent with Conjecture \ref{genus5}.

 Note also that this time we only see the newforms of level $2$ and $4$ appear once. This is one of the main differences between Conjecture \ref{genus4} and Conjecture \ref{genus5} (the direct sum decomposition in this case is explicitly considering this form as new at level $4$, as opposed to considering it via its lifts to level $100$).
\end{example}

\section{Genera of half-discriminant $4$}

In the last subsection we considered the five genera of discriminant $2p^2$ for $p$ an odd prime. In this subsection we now consider the case $p=2$.  The possible local lattices are as follows.

\vspace{0.1in}
\noindent\underline{\textbf{q $\neq$ 2}}

\vspace{0.1in}
\noindent Note that $\Lambda_q$ is unimodular if $q\neq 2$ and so \[\Lambda_q \sim \langle 1,1,\text{disc}(\Lambda_q)\rangle\sim \langle 1,1,8\rangle\sim\langle 1,1,2\rangle\] 

\noindent\underline{\textbf{q = 2}}

\vspace{0.1in}
\noindent Here the discriminant is $\text{disc}(\Lambda_2) = [8] \in\mathbb{Z}_2/(\mathbb{Z}_2^{\times})^2$ and a simple calculation gives the following possibilities (with their Conway-Sloane symbols, Types and Hasse symbols):

\vspace{0.1in}
\begin{center}
\begin{tabular}{c|c|c|c}
$\Lambda_2$ & Symbol & Type & $S(V_2)$\\ \hline
$\left\langle\begin{pmatrix}0 & 1\\ 1 & 0\end{pmatrix},-8\right\rangle$ & $1^{+2}_{\text{II}}8^{+1}_{-1}$ & II & -1  \\
$\left\langle\begin{pmatrix}2 & 1\\ 1 & 2\end{pmatrix},-40\right\rangle$ & $1^{-2}_{\text{II}}8^{-1}_{-5}$ & II & -1 \\
$\langle 2,-2,-2\rangle$ & $2_{-1}^{+3}$ &II & -1  \\ 
$\langle 1,-2,-4\rangle$ & $[1^{+1}2^{+1}4^{+1}]_7$ &I&-1 \\ 
$\langle -1, -1, 8\rangle$ & $1^{+2}_{6}8^{+1}_{1} $ & I & -1  \\
$\langle 1,5,40\rangle$ & $1^{-2}_{6}8^{-1}_{5}$ & I & -1  \\
$\langle 1,-1,-8\rangle$ & $1^{+2}_{0}8^{+1}_{-1} $ & I & -1\\
$\langle 2,2,2\rangle $ & $ 2^{+3}_{3}$ & II & +1 \\
$\langle 1,2,4\rangle $ & $[1^{+1}2^{+1}4^{+1}]_3$ &I & +1 \\
$\langle -1,-5,40\rangle $ & $1^{-2}_{2}8^{-1}_5$ &I & +1 \\
$\langle 1,1,8\rangle $ & $ 1^{+2}_2 8^{+1}_1$ &I & +1\\
$\langle 1,-5,-40\rangle $ & $1^{-2}_48^{-1}_{-5}$ &I & +1\\
\end{tabular} \\
\end{center}

Now we can classify the genera of such lattices. Since $(V,Q)$ is positive definite we have $S(V_{\infty}) = 1$. We also have $S(V_q) = 1$ for any $q\neq 2$, and so each possible genus for $\Lambda$ must give local lattice satisfying $S(V_2) = 1$ in order to satisfy the compatibility relations for Hasse symbols. Hence there are five possible genera of interest, with Gram matrices locally at $2$ given below:

\begin{center}
\begin{tabular}{c|c}
Label & $\Lambda_2$ \\ \hline
A & $\langle 2,2,2\rangle$  \\
B & $\langle 1,2,4\rangle $  \\
C& $\langle -1,-5,40\rangle $  \\
D& $\langle 1,1,8\rangle $ \\
E& $\langle 1,-5,-40\rangle $ \\
\end{tabular} \\
\end{center}

 Each such genus satisfies $|\mathrm{Cl}(\Lambda)| = 1$, so that when $W = \mathbb{C}$ is the trivial representation we find that \[\mathcal{M}(O(V),O(\widehat{\Lambda}),W) = \mathbb{C}\mathbbm{1}\] is $1$-dimensional spanned by the Eisenstein form with Hecke eigenvalues $p+1$ for all primes $p\neq 2$. Hence, for trivial weight there is no modular data to be found (which will agree with the results below, since the analogous classical modular forms will not exist either!). However, for other weights there may still exist non-trivial orthogonal modular forms and so these genera are worthy of study.

\subsection{Genus A}

This is the unique even genus of discriminant $8$. Before proving the relevant transfer result we introduce the linear map $\hat{\eta}_{(2)}\in\text{GL}(\widehat{V})$ with local components: 

\[\eta_{(2),q} = \begin{cases}\begin{pmatrix}1 & 0 & 1\\ 0 & 1 & 1\\ -2 & -2 & 2\end{pmatrix} & \text{ if } q = 2\\ \text{id} & \text{ if } q\neq 2\end{cases}.\] 

Again, this map depends on a choice of basis for $V_2$ (which will usually be a basis for $\Lambda_2$ realising a given Gram matrix).

\begin{thm}\label{genusA}
    Let $\Lambda$ be a lattice representing Genus \textup{A}. Then for each $k\geq 0$ and $d\in\{1,2\}$ there exists a Hecke-equivariant isomorphism (for the algebra generated by the Hecke operators $T_q$ with $q\neq 2$ prime): \[\mathcal{S}(O(V),O(\widehat{\Lambda}),W_{2k}\otimes \theta_d) \cong S_{2k+2}^{\textup{new}}\left(\Gamma_0(2),\varepsilon^{(2)}_2\varepsilon^{(2)}_d\right).\]
\end{thm}

\begin{proof}
Pick a basis of $\Lambda_2$ realising the diagonal Gram matrix given in the table above. We view $\Lambda$ as an odd lattice of determinant $1$ in $(V,Q^{(\frac{1}{2})})$ with local lattice $\Lambda_2 \sim \langle 1,1,1\rangle$. Then the lattice $\Pi = V\cap \hat{\eta}_{(2)}(\widehat{\Lambda})$ is an even lattice of discriminant $4$ with local lattices: \[\Pi_q\sim \begin{cases}\langle 2^{-1},2^{-1},1\rangle \sim \langle 1,1,1\rangle & \text{ if } q \neq 2\\ \left\langle\begin{pmatrix}2 & 1\\ 1 & 2\end{pmatrix}, 12\right\rangle & \text{ if } q = 2\end{cases}\]

We have that $O(\widehat{\Lambda}) = O(\widehat{\Pi})$ (since $O(\Lambda_2) = O(\Pi_2)$ by Proposition \ref{2-unimod} and $\Lambda_q = \Pi_q$ when $q\neq 2$). Another application of Lemma \ref{samespace} and Lemma \ref{scaling} now implies that there is an equality of Hecke modules: \[\mathcal{S}(O(V),O(\widehat{\Lambda}),W) = \mathcal{S}(O(V),O(\widehat{\Pi}),W)\] for every representation $W$ of $O(V)$.

By the recipe outlined in Subsection \ref{Ternary Lattices and Quaternion orders} we find that the even Clifford algebra of $(V,Q^{(\frac{1}{2})})$ is a definite quaternion algebra $B = \text{Cl}^0(V)$ of discriminant $2$. The order $\mathcal{O} = \text{Cl}^0(\Pi)$ has discriminant $\text{discrd}(\mathcal{O}) = 2$ and so is maximal, hence residually unramified everywhere with $\varepsilon_2(\mathcal{O}) = -1$. The result then follows from the results of Hein-Tornaria-Voight (see Theorem \ref{thm: 6.12 in HTV}).
\end{proof}

\subsection{Genera C and D}
Genera C and D are odd genera with local lattices $\langle -1,-5,40\rangle$ and $\langle 1,1,8\rangle$ at $2$. By Proposition \ref{2-mod}(ii), we see that these genera share the same adelic orthogonal group with the global lattice $\Pi$ having Gram matrix $\langle 1,1,2\rangle$ and discriminant $2$. Hence by Lemma \ref{samespace} both genera give rise to Hecke modules of orthogonal modular forms equal to those coming from the lattice $\Pi$.

Also by Proposition \ref{2-mod}(ii) we find that the even lattice $\Sigma$ of discriminant $8$ and Gram matrix $\langle 2,2,2\rangle$ is an index $2$ sublattice of $\Pi$ satisfying $O(\widehat{\Pi}) \subseteq O(\widehat{\Sigma})$ (note that $\Sigma$ represents Genus A). This means that there is a Hecke-equivariant injection: \[\mathcal{S}(O(V),O(\widehat{\Sigma}),W) \longhookrightarrow \mathcal{S}(O(V),O(\widehat{\Lambda}),W)\] for every representation $W$ of $O(V)$. When $W = W_{2k}\otimes\theta_d$ we may use Theorem \ref{genusA} to deduce that $S_{2k}^{\text{new}}\left(\Gamma_0(2),\varepsilon_2^{(2)}\varepsilon_d^{(2)}\right)$ embeds as a Hecke module into the RHS as a space of ``oldforms".

We may view $\Pi$ as an even lattice of discriminant $16$ in the quadratic space $(V,Q^{(2)})$. The even Clifford algebra of this space is a definite quaternion algebra $\text{Cl}^0(V)$ of discriminant $2$. The order $\mathcal{O} = \text{Cl}^0(\Pi)$ has discriminant $\text{discrd}(\mathcal{O}) = 8$ and is residually ramified (i.e.\ $\varepsilon_2(\mathcal{O}) = 0$). Hence the results of Hein-Tornaria-Voight cannot be applied here.

At present we are only able to make the following conjecture, based on computational evidence and supported by the deductions above.

\begin{conj}\label{genusCD}
Let $\Lambda$ be a lattice representing Genus \textup{C} or \textup{D}. Then for each $k\geq 0$ and $d\in\{1,2\}$ there exists a Hecke-equivariant isomorphism (for the algebra generated by the Hecke operators $T_q$ with $q\neq 2$ prime): \[\mathcal{S}(O(V),O(\widehat{\Lambda}),W_{2k}\otimes \theta_d) \cong S_{2k+2}^{2\textup{-new}}\left(\Gamma_0(4),\varepsilon^{(2)}_2\varepsilon^{(2)}_d\right).\]
\end{conj}

\begin{example}\label{Example genus C and D disc 8}

Consider the odd lattice $\Lambda$ with Gram matrix \[ T = \begin{pmatrix}
    1 & 0 & 0 \\
    0 & 1 & 0 \\
    0 & 0 & 8
\end{pmatrix}\] representing Genus D. We have that $\text{Cl}(\Lambda) = \{[\Lambda]\}$ and so there is no modular data to be found at trivial weight (which fits with $S_2^{2\text{-new}}(\Gamma_0(4)) = \{0\}$). However, at weight $W_8$ we find that $\text{dim}(\mathcal{S}(O(V),O(\widehat{\Lambda}),W_8)) = 2$, and the characteristic polynomials of $T_p$ (for $p \neq 2$) have linear factors $F_{p,1}(x)$ and $F_{p,2}(x)$ as follows:

\vspace{0.1in}
\begin{tabular}{l l l l l}
			 $F_{3,1}(x) = x + 156$&\quad $F_{5,1}(x) = x-870$&\quad $F_{7,1}(x) = x+952$ &\quad $\dots$ \\
		$F_{3,2}(x) = x - 228$&\quad  $F_{5,2}(x) = x + 666$&\quad $F_{7,2}(x) = x + 6328$&\quad $\dots$
		\end{tabular}
\vspace{0.1in}

\noindent These match with the newforms $f \in S_{10}^{\text{new}}(\Gamma_0(2))$ and $g \in S_{10}^{\text{new}}(\Gamma_0(4))$ having LMFDB labels \href{https://www.lmfdb.org/ModularForm/GL2/Q/holomorphic/2/10/a/a/}{2.10.a.a} and \href{https://www.lmfdb.org/ModularForm/GL2/Q/holomorphic/4/10/a/a/}{4.10.a.a} and Fourier expansions: 

\begin{align*}
    f(q) &=q + 16q^2 - 156q^3 + 256q^4 + 870q^5 - 2496q^6 - 952q^7 + \dots  \\
    g(q) &= q + 228 q^3 - 666 q^5 - 6328 q^7 + 32301 q^9 - 30420 q^{11}  \dots
\end{align*} The newforms above are precisely those of level $4$ that are $2$-new and have Atkin-Lehner sign $-1$, as expected from Conjecture \ref{genusCD}.
\end{example}

\subsection{Genus B and E}
Genera B and E are odd genera with local lattices $\langle 1,2,4\rangle$ and $\langle 1,-5,-40\rangle$ at $2$. By Proposition \ref{2-misc}(ii) we see that these genera share the same adelic orthogonal group, hence by Lemma \ref{samespace} the two genera give rise to the same Hecke modules of orthogonal modular forms. The proof of this lemma also reveals that these lattices share the same adelic orthogonal group as the common sublattice $\Pi$ of index $2$ with Gram matrix \[T = \begin{pmatrix}3 & 0 & 4\\ 0 & 4 & 0\\ 4 & 0 & 8\end{pmatrix}.\] It can be checked that $\Pi$ has a superlattice $\Sigma$ of index $2$ with Gram matrix $\langle 2,2,2\rangle$, and that $O(\widehat{\Pi})\subseteq O(\widehat{\Sigma})$ (note that $\Sigma$ represents Genus A). This means that there is a Hecke-equivariant injection: \[\mathcal{S}(O(V),O(\widehat{\Sigma}),W) \longhookrightarrow \mathcal{S}(O(V),O(\widehat{\Lambda}),W)\] for every representation $W$ of $O(V)$. When $W = W_{2k}\otimes\theta_d$ we may use Theorem \ref{genusA} to deduce that $S_{2k+2}^{\text{new}}\left(\Gamma_0(2),\varepsilon_2^{(2)}\varepsilon_d^{(2)}\right)$ embeds as a Hecke module into the RHS as a space of ``oldforms".

We can view either Genus B or Genus E as being represented by an even lattice of discriminant $64$ in $(V,Q^{(2)})$. The even Clifford algebra of this space is a definite quaternion algebra $\text{Cl}^0(V)$ of discriminant $2$. The order $\mathcal{O} = \text{Cl}^0(\Lambda)$ has discriminant $\text{discrd}(\mathcal{O})=32$ and is residually ramified (i.e.\ $\varepsilon_2(\mathcal{O}) = 0$). Hence the results of Hein-Tornaria-Voight cannot be applied here.

At present we are only able to make the following conjecture, based on computational evidence and supported by the deductions above.

\begin{conj}\label{conjBE} 
    Let $\Lambda$ be a lattice representing Genus \textup{B} or \textup{E}. Then for each $k\geq 0$ and $d\in\{1,2\}$ there exists a Hecke-equivariant isomorphism (for the algebra generated by the Hecke operators $T_q$ with $q\neq 2$ prime): \[\mathcal{S}(O(V),O(\widehat{\Lambda}),W_{2k}\otimes \theta_d) \cong S_{2k+2}^{\textup{new}}\left(\Gamma_0(2),\varepsilon_2^{(2)}\varepsilon_d^{(2)}\right) \oplus S_{2k+2}^{4\textup{-new}}\left(\Gamma_0(8),\varepsilon^{(2)}_2\varepsilon^{(2)}_d\right).\]
\end{conj}

\begin{example}
Consider the odd lattice with Gram matrix \[ T = \begin{pmatrix}
    1 & 0 & 0 \\
    0 & 2 & 0 \\
    0 & 0 & 4
\end{pmatrix}\] representing Genus B. Again,  $\text{Cl}(\Lambda) = \{[\Lambda]\}$ and so no modular data is found at trivial weight (which agrees with the fact that $S_2^{2\text{-new}}(\Gamma_0(8)) = \{0\}$). At weight $W_8$ we find that $\text{dim}(\mathcal{S}(O(V),O(\widehat{\Lambda}),W_8)) = 3$. The characteristic polynomials of the Hecke operators $T_p$ (with $p\neq 2$) decompose into linear factors $F_{p,1}(x)$, $F_{p,2}(x)$ and $F_{p,3}(x)$ where the first two are the forms $f$ and $g$ with LMFDB labels \href{https://www.lmfdb.org/ModularForm/GL2/Q/holomorphic/2/10/a/a/}{2.10.a.a} and \href{https://www.lmfdb.org/ModularForm/GL2/Q/holomorphic/4/10/a/a/}{4.10.a.a} from Example \ref{Example genus C and D disc 8}. The third linear factor is given by:

\vspace{0.1in}
\begin{tabular}{l l l l l}
			$F_{3,3}(x) = x - 68$& $F_{5,3}(x) = x-1510$& $F_{7,3}(x) = x-10248$ & $\dots$ 
		\end{tabular}
\vspace{0.1in}

\noindent which is identified as the newform $h\in S_{10}^{\text{new}}(\Gamma_0(8))$ with LMFDB label
        \href{https://www.lmfdb.org/ModularForm/GL2/Q/holomorphic/8/10/a/b/}{8.10.a.b} and Fourier expansion:
$$f(q) = q + 68q^3 + 1510q^5 + 10248q^7 - 15059q^9 - 3916q^{11} + 176594q^{13} +\dots $$

All three forms observed are precisely those predicted by Conjecture \ref{conjBE} (in particular all have Atkin-Lehner eigenvalue $-1$). The other forms with Atkin-Lehner eigenvalue $+1$ are not seen in this calculation (but would be seen if we twist the weight by $\theta_2$).
\end{example}

\section{Genera of half-discriminant $4p^2$}\label{8p2}

Let $(V,Q)$ be a positive definite ternary quadratic space over $\mathbb{Q}$, and $\Lambda \subset V$ be an integral lattice of discriminant $8p^2$ with $p$ an odd prime.  We are interested in the spaces $\mathcal{S}(O(V),O(\widehat{\Lambda}),W_{2k}\otimes\theta_d)$ of orthogonal modular forms for various weights and the associated Hecke action coming from the $p$-neighbour operators $T_q$ (for $q\neq 2,p$). To get us started we once again classify the local lattices.

\vspace{0.1in}
\noindent\underline{\textbf{q $\neq$ 2,p}}

\vspace{0.1in}
\noindent Note that $\Lambda_q$ is unimodular if $q\neq 2,p$ and so \[\Lambda_q \sim \langle 1,1,\text{disc}(\Lambda_q)\rangle \sim \langle 1,1,8p^2\rangle\sim\langle 1,1,2\rangle.\] 

\noindent\underline{\textbf{q = 2}}

\vspace{0.1in}
\noindent Here the discriminant is $\text{disc}(\Lambda_2) = [8p^2] = [8]\in\mathbb{Z}_2/(\mathbb{Z}_2^{\times})^2$ and a simple calculation gives the following possibilities (with their Conway-Sloane symbols, Types and Hasse symbols):

\vspace{0.1in}
\begin{center}
\begin{tabular}{c|c|c|c}
$\Lambda_2$ & Symbol & Type & $S(V_2)$\\ \hline
$\left\langle\begin{pmatrix}0 & 1\\ 1 & 0\end{pmatrix},-8\right\rangle$ & $1^{+2}_{\text{II}}8^{+1}_{-1}$ & II & -1  \\
$\left\langle\begin{pmatrix}2 & 1\\ 1 & 2\end{pmatrix},-40\right\rangle$ & $1^{-2}_{\text{II}}8^{-1}_{-5}$ & II & -1 \\
$\langle 2,-2,-2\rangle$ & $2_{-1}^{+3}$ &II & -1  \\ 
$\langle 1,-2,-4\rangle$ & $[1^{+1}2^{+1}4^{+1}]_7$ &I&-1 \\ 
$\langle -1, -1, 8\rangle$ & $1^{+2}_{6}8^{+1}_{1} $ & I & -1  \\
$\langle 1,5,40\rangle$ & $1^{-2}_{6}8^{-1}_{5}$ & I & -1  \\
$\langle 1,-1,-8\rangle$ & $1^{+2}_{0}8^{+1}_{-1} $ & I & -1\\
$\langle 2,2,2\rangle $ & $ 2^{+3}_{3}$ & II & +1 \\
$\langle 1,2,4\rangle $ & $[1^{+1}2^{+1}4^{+1}]_3$ &I & +1 \\
$\langle -1,-5,40\rangle $ & $1^{-2}_{2}8^{-1}_5$ &I & +1 \\
$\langle 1,1,8\rangle $ & $ 1^{+2}_2 8^{+1}_1$ &I & +1\\
$\langle 1,-5,-40\rangle $ & $1^{-2}_48^{-1}_{-5}$ &I & +1\\
\end{tabular} \\
\end{center}

\newpage
\noindent\underline{\textbf{q = p}}

\vspace{0.1in}
\noindent Here the discriminant is: \[\text{disc}(\Lambda_p) = [8p^2] = [2p^2]= \begin{cases} [p^2] & \text{ if } p\equiv 1,7\bmod 8\\ [rp^2] & \text{ if } p\equiv 3,5\bmod 8\end{cases}.\] A similar calculation gives (with $\varepsilon = (-1)^{\frac{p-1}{2}}$):

\vspace{0.1in}
\begin{center}
\begin{tabular}{c|c|c|c}
$\Lambda_p$ & Symbol & $\text{disc}(\Lambda_p)$ & $S(V_p)$\\

\hline
$\langle 1,p,p\rangle$ & $1^{+1}p^{+2}$ & $[p^2]$ & $\varepsilon$\\
$\langle r,p,pr\rangle$ & $1^{-1}p^{-2}$ & $[p^2]$ & $-\varepsilon$\\
$\langle 1,1,p^2\rangle$ & $1^{+2}(p^2)^{+1}$ & $[p^2]$ & $1$\\
$\langle 1,r,p^2r\rangle$ & $1^{-2}(p^2)^{-1}$ & $[p^2]$ & $1$\\
\hline
$\langle r,p,p\rangle$ & $1^{-1}p^{+2}$ & $[rp^2]$ & $\varepsilon$\\
$\langle 1,p,pr\rangle$ & $1^{+1}p^{-2}$ & $[rp^2]$ & $-\varepsilon$\\
$\langle 1,r,p^2\rangle$ & $1^{-2}(p^2)^{+1}$ & $[rp^2]$ & $1$\\
$\langle 1,1,p^2r\rangle$ & $1^{+2}(p^2)^{-1}$ & $[rp^2]$ & $1$
\end{tabular}
\end{center}

\vspace{0.1in}
Now we can classify the genera of such lattices. Since $(V,Q)$ is positive definite we have $S(V_{\infty}) = 1$. We also have $S(V_q) = 1$ for any $q\neq 2,p$, and so each possible genus for $\Lambda$ must give local lattices satisfying $S(V_2)S(V_p) = 1$ in order to satisfy the compatibility relations for Hasse symbols. Hence we must have $S(V_2) = S(V_p)$ and it is now clear that there are twenty two possible genera of interest, with Gram matrices locally at $2$ and $p$ given below (depending on the class $p\bmod 8$):

\vspace{0.2in}
\begin{tabular}{c|c|c|c|c|c}
\multirow{2}{*}{Label} & \multirow{2}{*}{$\Lambda_2$} & \multicolumn{4}{c}{$\Lambda_p$}\\ 
\cline{3-6} & & $p\equiv 1 \bmod 8$ & $p\equiv 3 \bmod 8$ &  $p\equiv 5 \bmod 8$ & $p\equiv 7 \bmod 8$\\ \hline $1$A & $\left\langle\begin{pmatrix}0 & 1\\ 1 & 0\end{pmatrix},-8\right\rangle$ & \multirow{7}{*}{$\langle r,p,pr\rangle$} & \multirow{7}{*}{$\langle 2,p,p\rangle$} & \multirow{7}{*}{$\langle 1,p,2p\rangle$} & \multirow{7}{*}{$\langle 1,p,p\rangle$}\\ 
$1$B & $\left\langle\begin{pmatrix}2 & 1\\ 1 & 2\end{pmatrix},-40\right\rangle$ &  &  &  & \\ 
\cline{1-2} $2$A & $\langle 2,-2,-2\rangle$ &  &  &  & \\ 
$2$B & $\langle 1,-2,-4\rangle$ & & & & \\ 
$2$C & $\langle-1,-1,8\rangle$ & & & & \\ 
$2$D & $\langle 1,5,40 \rangle$ & & & & \\ 
$2$E & $\langle 1,-1,-8 \rangle$ & & & & \\ 
\hline $3$A & $\langle 2,2,2\rangle$ & \multirow{5}{*}{$\langle 1,p,p\rangle$} & \multirow{5}{*}{$\langle 1,p,2p\rangle$} & \multirow{5}{*}{$\langle 2,p,p\rangle$} & \multirow{5}{*}{$\langle -1,p,-p\rangle$}\\ 
$3$B & $\langle 1,2,4\rangle$ & & & & \\ 
$3$C & $\langle-1,-5,40\rangle$ & & & & \\ 
$3$D & $\langle 1,1,8 \rangle$ & & & & \\ 
$3$E & $\langle 1,-5,-40 \rangle$ & & & & \\ 
\hline $4$A & $\langle 2,2,2\rangle$ & \multirow{5}{*}{$\langle 1,1,p^2\rangle$} & \multirow{5}{*}{$\langle 1,2,p^2\rangle$} & \multirow{5}{*}{$\langle 1,1,2p^2\rangle$} & \multirow{5}{*}{$\langle 1,-1,-p^2\rangle$}\\ 
$4$B & $\langle 1,2,4\rangle$ & & & & \\ 
$4$C & $\langle-1,-5,40\rangle$ & & & & \\ 
$4$D & $\langle 1,1,8 \rangle$ & & & & \\ 
$4$E & $\langle 1,-5,-40 \rangle$ & & & & \\ 
\hline $5$A & $\langle 2,2,2\rangle$ & \multirow{5}{*}{$\langle 1,r,p^2r\rangle$} & \multirow{5}{*}{$\langle 1,1,2p^2\rangle$} & \multirow{5}{*}{$\langle 1,2,p^2\rangle$} & \multirow{5}{*}{$\langle 1,1,p^2\rangle$}\\ 
$5$B & $\langle 1,2,4\rangle$ & & & & \\ 
$5$C & $\langle-1,-5,40\rangle$ & & & & \\ 
$5$D & $\langle 1,1,8 \rangle$ & & & & \\ 
$5$E & $\langle 1,-5,-40 \rangle$ & & & & \\ 
\end{tabular} 

\vspace{0.2in} \noindent The labels have been chosen to most suit comparison with the genera of discriminant $2p^2$ and $8$ previously considered.

Again, for a representative lattice $\Lambda\subset V$ coming from each of the genera above, we will either prove or conjecture that there is a Hecke-equivariant isomorphism between $\mathcal{S}(O(V), O(\widehat{\Lambda}),W_{2k}\otimes\theta_d)$ and a space of classical modular forms. 

\subsection{Genera $1$B, $2$A, type C and type D}

All of these genera can be studied directly by realising representative lattices as index $2$ sublattices of genera of discriminant $2p^2$.

\begin{thm}\label{genus1b2acd}
Let $\Lambda$ be a lattice representing Genus $1$\textup{B}, Genus $2$\textup{A} or any of the genera of type \textup{C} or \textup{D}, and let $N$ denote the numeric part of the genus label. Then for each representation $W$ of $V$ there is an equality of Hecke modules: \[\mathcal{S}(O(V),O(\widehat{\Lambda}),W) = \mathcal{S}(O(V),O(\widehat{\Pi}),W),\] where $\Pi$ is a lattice of discriminant $2p^2$ representing Genus $N$.

In particular, for each $k\geq 0$ and $d\in\{1,2,p,2p\}$ we have the following Hecke-equivariant isomorphisms (for the algebra generated by the Hecke operators $T_q$ with $q\neq 2,p$ prime): \[\mathcal{S}(O(V),O(\widehat{\Lambda}),W_{2k}\otimes \theta_d)\cong \begin{cases}S^{\textup{new}}_{2k+2}\left(\Gamma_0(p),\varepsilon^{(p)}_p\varepsilon^{(p)}_d\right) & \text{ if } \Lambda\sim 1\textup{B}\\ S^{p\textup{-new}}_{2k+2}\left(\Gamma_0(2p),\varepsilon^{(2p)}_p\varepsilon^{(2p)}_d\right) & \text{ if } \Lambda\sim 2\textup{A},2\textup{C},2\textup{D}\\ S_{2k+2}^{2\textup{-new}}\left(\Gamma_0(4p),\varepsilon^{(4p)}_2\varepsilon^{(4p)}_d\right) & \text{ if } \Lambda\sim 3\textup{C},3\textup{D}\\ S^{2\textup{-new}}_{2k+2}\left(\Gamma_0(4p^2),\varepsilon_2^{(4p^2)}\varepsilon_d^{(4p^2)}\right) & \text{ if } \Lambda\sim 4\textup{C},4\textup{D}\\ S^{2\textup{-new}}_{2k+2}\left(\Gamma_0(4),\varepsilon^{(4)}_2\varepsilon^{(4)}_d\right)\oplus S^{2p^2\textup{-new}}_{2k+2}\left(\Gamma_0(4p^2),\varepsilon^{(4p^2)}_2\varepsilon^{(4p^2)}_d\right) & \text{ if } \Lambda\sim 5\textup{C},5\textup{D}\end{cases}\] where we have assumed Conjectures \ref{genus3}, \ref{genus4} and \ref{genus5} for the last three isomorphisms.\end{thm}

\begin{proof}
Proposition \ref{2-mod}(i) shows that if $\Lambda$ represents Genus $1$B and $\Pi$ represents Genus $1$ then $O(\Lambda_2) = O(\Pi_2)$. Since $O(\Lambda_p) = O(\Pi_p)$ for all $p\neq 2$ we then have $O(\widehat{\Lambda}) = O(\widehat{\Pi})$.

Proposition \ref{2-mod}(ii) shows that if $\Lambda$ represents any of the other genera in consideration and $\Pi$ represents Genus $N$ then $O(\Lambda_2) = O(\Pi_2)$. Since $O(\Lambda_p) = O(\Pi_p)$ for all $p\neq 2$ we then have $O(\widehat{\Lambda}) = O(\widehat{\Pi})$.

Having established equality of adelic orthogonal groups in all cases, Lemma \ref{samespace} gives an equality of Hecke modules: \[\mathcal{S}(O(V),O(\widehat{\Lambda}),W) = \mathcal{S}(O(V),O(\widehat{\Pi}),W)\] for any representation $W$ of $O(V)$. The second claim now follows from the explicit results proved and conjectured in Section $4$.
\end{proof}

\subsection{Remaining genera of type A}

The remaining genera cannot be reduced to genera of discriminant $2p^2$. While these lattices can also be constructed as index $2$ sublattices of lattices $\Pi$ of discriminant $2p^2$, we simply do not get an equality of adelic orthogonal groups. Indeed, extensive computation reveals that the corresponding spaces of orthogonal modular forms are not isomorphic in general.

For genera $3$A, $4$A and $5$A, Proposition \ref{2-mod}(ii) only provides us with an inclusion $O(\widehat{\Pi})\subseteq O(\widehat{\Lambda})$ (as the local $2$-adic lattice $\langle 1,1,2\rangle$ fails the final condition). This provides us with a Hecke-equivariant injection: \[\mathcal{S}(O(V),O(\widehat{\Lambda}),W)\longhookrightarrow \mathcal{S}(O(V),O(\widehat{\Pi}),W),\] which establishes the fact that all such eigenforms are ``old" (and can be translated to classical modular forms by the theorems in Section \ref{2p2}). However, it does not tell us precisely which subspace of the RHS we should expect to see.

Something similar happens for Genus $1$A. The local $2$-adic lattice for this genus is $\left\langle\begin{pmatrix}0 & 1\\ 1 & 0\end{pmatrix},-8\right\rangle$. Applying the map $\hat{\mu}_{(2)}\in\text{GL}(\widehat{V})$ with local components: \[\mu_{(2),q} = \begin{cases} \text{diag}\left(1,1,\frac{1}{2}\right) & \text{ if } q = 2\\ \text{ id } & \text{ if } q\neq 2\end{cases},\] with respect to any basis realising this Gram matrix, gives the superlattice $\Pi = V\cap\hat{\mu}_{(2)}(\widehat{\Lambda})$ of index $2$ has discriminant $2p^2$ (which represents Genus $1$). An argument similar to those in Section \ref{Orthog} shows that $O(\widehat{\Lambda})\subseteq O(\widehat{\Pi})$, giving rise to a Hecke-equivariant inclusion: \begin{align*}\mathcal{S}(O(V),O(\widehat{\Pi}),W) &\longhookrightarrow \mathcal{S}(O(V),O(\widehat{\Lambda}),W)\end{align*} for any representation $W$ of $O(V)$. In this case we also find that the RHS has ``oldforms" (which again can be translated to classical modular forms by the theorems in Section \ref{2p2}). However, in general the space on the RHS is larger and we cannot learn what the full space of orthogonal modular forms should be from this calculation alone.

By instead shifting focus to lattices of discriminant $4p^2$ we can actually settle three of the four genera completely.

\begin{thm}\label{3a4a5a}
Let $\Lambda$ be a lattice representing Genus $3$\textup{A}, Genus $4$\textup{A} or Genus $5$\textup{A}. Then for each representation $W$ of $V$ there is an equality of Hecke modules: \[\mathcal{S}(O(V),O(\widehat{\Lambda}),W) = \mathcal{S}(O(V),O(\widehat{\Pi}),W),\] where $\Pi$ is a suitable even lattice of discriminant $4p^2$ in $(V,Q^{\left(\frac{1}{2}\right)})$.

In particular, for each $k\geq 0$ and $d\in\{1,2,p,2p\}$ we have the following Hecke-equivariant isomorphisms (for the algebra generated by the Hecke operators $T_q$ with $q\neq 2,p$ prime): \[\mathcal{S}(O(V),O(\widehat{\Lambda}),W_{2k}\otimes \theta_d)\cong \begin{cases}S^{2\textup{-new}}_{2k+2}\left(\Gamma_0(2p),\varepsilon^{(2p)}_2\varepsilon^{(2p)}_d\right) & \text{ if } \Lambda\sim 3\textup{A}\\ S^{2\textup{-new}}_{2k+2}\left(\Gamma_0(2p^2),\varepsilon^{(2p^2)}_2\varepsilon^{(2p^2)}_d\right) & \text{ if } \Lambda\sim 4\textup{A}\\ S^{\textup{new}}_{2k+2}\left(\Gamma_0(2),\varepsilon^{(2)}_2\varepsilon^{(2)}_d\right)\oplus S^{\textup{new}}_{2k+2}\left(\Gamma_0(2p^2),\varepsilon^{(2p^2)}_2\varepsilon^{(2p^2)}_d\right) & \text{ if } \Lambda\sim 5\textup{A}\end{cases}\]
\end{thm}

\begin{proof}
Lattices representing genera $3$A, $4$A and $5$A have local $2$-adic lattice isometric to $\langle 2,2,2\rangle$. Following suit from the proof of Theorem \ref{genusA}, we first view $\Lambda$ as an odd lattice of discriminant $p^2$ in $(V,Q^{\left(\frac{1}{2}\right)})$ with local Gram matrix $\Lambda_2\sim\langle 1,1,1\rangle$. We then apply the map $\hat{\eta}_{(2)}\in\text{GL}(\hat{V})$ to get an even lattice $\Pi = V\cap \hat{\eta}_{(2)}(\hat{\Lambda})$ of discriminant $4p^2$ with local $2$-adic lattice isometric to $\Pi_2\sim\left\langle\begin{pmatrix}2 & 1\\ 1 & 2\end{pmatrix},12\right\rangle$. Proposition \ref{2-unimod} tells us that $O(\Lambda_2) = O(\Pi_2)$, and since $O(\Lambda_p) = O(\Pi_p)$ for all primes $p\neq 2$ we must have that $O(\widehat{\Lambda}) = O(\widehat{\Pi})$. As usual, Lemma \ref{samespace} and Lemma \ref{scaling} give us an equality of Hecke modules: \[\mathcal{S}(O(V),O(\widehat{\Lambda}),W) = \mathcal{S}(O(V),O(\widehat{\Pi}),W)\] for every representation $W$ of $O(V)$.

In the case of Genus $3$A we have local lattice $\Pi_p\sim \langle a,p,pb\rangle$ for some $a,b\in\{1,r\}$ and so may also apply the map $\hat{\psi}_{(p)}$ and scale the quadratic map by $\frac{1}{p}$ to receive an even lattice $\Omega = V\cap\hat{\psi}_{(p)}(\widehat{\Pi})$ of discriminant $4p$ in $(V,Q^{\left(\frac{1}{2p}\right)})$. This lattice satisfies $O(\widehat{\Pi}) = O(\widehat{\Omega})$ (by Proposition \ref{p-mod}). Lemma \ref{samespace} and Lemma \ref{scaling} once again give an equality of Hecke modules (for each representation $W$ of $O(V)$): \[\mathcal{S}(O(V),O(\widehat{\Pi}),W) = \mathcal{S}(O(V),O(\widehat{\Omega}),W).\] By the recipe outlined in Subsection \ref{Ternary Lattices and Quaternion orders} we find that the even Clifford algebra of $(V,Q^{(\frac{1}{2p})})$ is a definite quaternion algebra $B = \text{Cl}^0(V)$ of discriminant $2$. The order $\mathcal{O} = \text{Cl}^0(\Omega)$ has discriminant $\text{discrd}(\mathcal{O}) = 2p$ and so is Eichler, hence residually unramified everywhere with $\varepsilon_2(\mathcal{O}) = -1$ and $\varepsilon_p(\mathcal{O}) = 1$. The result then follows from the results of Hein-Tornaria-Voight (see Theorem \ref{thm: 6.12 in HTV}).

In the case of Genus $4$A and Genus $5$A we are unable to further reduce the discriminant. However, one again uses the recipe in Subsection \ref{Ternary Lattices and Quaternion orders} to see that the quaternion algebra $B = \text{Cl}^0(V)$ corresponding to $(V,Q^{\left(\frac{1}{2}\right)})$ has discriminant $\text{disc}(B) = 2$ and that the quaternion order $\mathcal{O} = \text{Cl}^0(\Pi)$ has $\text{discrd}(\mathcal{O}) = 2p^2$. This order is locally maximal at $2$ with $\varepsilon_2(\mathcal{O}) = -1$. We also find that $\mathcal{O}$ is locally Eichler at $p$ in the case of Genus $4$A (with $\varepsilon_p(\mathcal{O}) = 1$), whereas it is not locally Eichler at $p$ in the case of Genus $5$A (but has $\varepsilon_p(\mathcal{O}) = -1$). In either case, the order $\mathcal{O}$ is residually unramified everywhere and so another application of Hein-Tornaria-Voight then provides the results as claimed for these genera.
\end{proof}

We finish the subsection by dealing with Genus $1$A.

\begin{thm}\label{genus1a}
Let $\Lambda$ be a lattice representing Genus $1$\textup{A}. Then for each $k\geq 0$ and $d\in\{1,2,p,2p\}$ we have the following Hecke-equivariant isomorphisms (for the algebra generated by the Hecke operators $T_q$ with $q\neq 2,p$ prime): \[\mathcal{S}(O(\hat{\Lambda}),W_{2k}\otimes \theta_d)\cong S^{p\textup{-new}}_{2k+2}\left(\Gamma_0(4p),\varepsilon^{(4p)}_p\varepsilon^{(4p)}_d\right).\]
\end{thm}

\begin{proof}
First we apply the map $\hat{\psi}_{(p)}$ and scale the quadratic map by $\frac{1}{p}$. As always, this produces a lattice $\Pi = V\cap\hat{\psi}_{(p)}(\widehat{\Lambda})$ of discriminant $8p$ in $(V,Q^{\left(\frac{1}{p}\right)})$ that satisfies $O(\widehat{\Lambda}) = O(\widehat{\Pi})$ by Proposition \ref{p-mod}. Lemma \ref{samespace} and Lemma \ref{scaling} then give an equality of Hecke modules: \[\mathcal{S}(O(V),O(\widehat{\Lambda}),W) = \mathcal{S}(O(V),O(\widehat{\Pi}),W),\] for every representation $W$ of $O(V)$. 

By the recipe outlined in Subsection \ref{Ternary Lattices and Quaternion orders} we find that the even Clifford algebra of $(V,Q^{(\frac{1}{p})})$ is a definite quaternion algebra $B = \text{Cl}^0(V)$ of discriminant $p$. The order $\mathcal{O} = \text{Cl}^0(\Pi)$ has discriminant $\text{discrd}(\mathcal{O}) = 4p$ and is Eichler with $\varepsilon_2(\mathcal{O}) = 1$ and $\varepsilon_p(\mathcal{O}) = -1$. Hence $\mathcal{O}$ is residually unramified everywhere and the result then follows from the results of Hein-Tornaria-Voight (see Theorem \ref{thm: 6.12 in HTV}).
\end{proof}

\subsection{Remaining genera of types B and E}
We are now left with the genera of types B and E (except $1$B, which we have already handled). The first step is to prove that these give rise to the same spaces of orthogonal modular forms.

\begin{thm}
Let $\Pi$ and $\Omega$ be lattices representing genera of types \textup{B} and \textup{E} with the same numerical part. Then for each representation $W$ of $V$ there is an equality of Hecke modules: \[\mathcal{S}(O(V),O(\widehat{\Pi}),W) = \mathcal{S}(O(V),O(\widehat{\Omega}),W).\]
\end{thm}

\begin{proof}
The genera $2$B and $2$E have local $2$-adic lattices isometric to $\langle 1,-2,-4\rangle$ and $\langle 1,-1,-8\rangle$. Proposition \ref{2-misc}(i) tells us that $O(\Pi_2) = O(\Omega_2)$. Since $O(\Pi_q) = O(\Omega_q)$ for all $q\neq 2$ we have that $O(\widehat{\Pi}) = O(\widehat{\Omega})$.

The other pairs of genera have local $2$-adic lattices isometric to $\langle 1,2,4\rangle$ and $\langle 1,-5,-40\rangle$. Proposition \ref{2-misc}(ii) tells us that $O(\Pi_2) = O(\Omega_2)$. Since $O(\Pi_q) = O(\Omega_q)$ for all $q\neq 2$ we have that $O(\widehat{\Pi}) = O(\widehat{\Omega})$.

Once again we have equality of adelic orthogonal groups in both cases, so Lemma \ref{samespace} provides us with the required equality of Hecke modules in each case: \[\mathcal{S}(O(V),O(\widehat{\Pi}),W) = \mathcal{S}(O(V),O(\widehat{\Omega}),W).\]
\end{proof}

It now suffices to work with lattices of genera $2$B, $3$B, $4$B and $5$B. These are all odd genera and so we run into the same issues as explained previously. In all cases we cannot apply the results of Hein-Tornaria-Voight, as we explain below, and can instead only offer conjectures based on computational evidence.

For Genus $2$B, the local Gram matrix at $p$ is of the form $\Lambda_p\sim\langle a,p,pb\rangle$ for some $a,b\in\{1,r\}$ and so we may apply the map $\hat{\psi}_{(p)}$ to get an odd lattice $\Pi = V \cap \hat{\psi}_{(p)}(\widehat{\Lambda})$ of discriminant $8p$ in $(V,Q^{(\frac{1}{p})})$. We can view $\Pi$ as an even lattice in $(V,Q^{(\frac{2}{p})})$ of discriminant $64p$, with corresponding quaternion algebra $B = \text{Cl}^0(V)$ of discriminant $p$ and order $\mathcal{O} = \text{Cl}^0(\Pi)$ of discriminant $\text{discrd}(\mathcal{O}) = 32p$. This order is maximal at $p$ with $\varepsilon_p(\mathcal{O}) = -1$ but is residually ramified at $2$ with $\varepsilon_2(\mathcal{O}) = 0$.

\begin{conj}\label{genus2be} Let $\Lambda$ be a lattice representing Genus $2$\textup{B}. Then for each $k\geq 0$ and $d\in\{1,2,p,2p\}$ there exists a Hecke-equivariant isomorphism (for the algebra generated by the Hecke operators $T_q$ with $q\neq 2,p$ prime) so that:
\[\mathcal{S}(O(V),O(\widehat{\Lambda}),W_{2k}\otimes \theta_d) \cong S^{p\textup{-new}}_{2k+2}\left(\Gamma_0(2p),\,\varepsilon_p^{(2p)}\varepsilon_d^{(2p)}\right)\oplus S_{2k+2}^{\textup{new}}\left(\Gamma_0(8p),\varepsilon_p^{(8p)}\varepsilon_d^{(8p)}\right).\]
\end{conj}

\begin{example}
Let $\Lambda$ be an odd lattice representing Genus $2$B with Gram matrix \[T = \begin{pmatrix}
        1 & 0 & 0\\
        0 & 14 & 0\\
        0 & 0 & 28
    \end{pmatrix},\] so that $\text{disc}(\Lambda) = 392 =  8\cdot 7^2$. Here, the class set has size $3$, and at weight $W_8$ one can compute that $\dim(\mathcal{S}(O(V),O(\widehat{\Lambda}),W_8)\big) = 9$. The characteristic polynomials of $T_3$, $T_5$ and $T_{11}$ factor as follows:
\[
\vspace{0.2in}
\renewcommand{\arraystretch}{1.15}
\begin{array}{l|l|l|l|c}
F_{3,i}(x) & F_{5,i}(x) & F_{11,i}(x) & \text{newform orbit} & (w_2,w_7)\\ \hline
x^2+14x-57576 & x^2+2730x+846720 & x^2-44940x-2036361600 & \href{https://www.lmfdb.org/ModularForm/GL2/Q/holomorphic/14/10/a/c/}{14.10.a.c} & (+,-)\\
\deg 3 & \deg 3 & \deg 3 & \href{https://www.lmfdb.org/ModularForm/GL2/Q/holomorphic/7/10/a/b/}{7.10.a.b} & (*, -)\\
\deg 4 & \deg 4 & \deg 4 & \href{https://www.lmfdb.org/ModularForm/GL2/Q/holomorphic/56/10/a/c/}{56.10.a.c} & (+,-)
\end{array}
\]

These are exactly the forms predicted by Conjecture \ref{genus2be}, i.e. we see exactly the newforms of level $7$ with $w_p=-1$,
and of levels $14$ and $56$ with $(w_2,w_p)=(+,-)$. The other newforms of these levels have
different Atkin-Lehner values and do not appear at trivial weight. Note that we do not see any newforms of level $28$ as these have $w_2 = +1$ (as does every newform of level $4p$ with $p$ an odd prime, explaining why no forms of level $4p$ feature in Conjecture \ref{genus2be}).
\end{example}

For Genus $3$B, the local Gram matrix at $p$ is again of the form $\Lambda_p\sim\langle a,p,pb\rangle$ for some $a,b\in\{1,r\}$ and so we may apply the map $\hat{\psi}_{(p)}$ to get an odd lattice $\Pi = V \cap \hat{\psi}_{(p)}(\widehat{\Lambda})$ of discriminant $8p$ in $(V,Q^{(\frac{1}{p})})$. We can view $\Pi$ as an even lattice in $(V,Q^{(\frac{2}{p})})$ of discriminant $64p$, with corresponding quaternion algebra $B = \text{Cl}^0(V)$ of discriminant $2$ and order $\mathcal{O} = \text{Cl}^0(\Pi)$ of discriminant $\text{discrd}(\mathcal{O}) = 32p$. This order is Eichler at $p$ with $\varepsilon_p(\mathcal{O}) = 1$ but is residually ramified at $2$ with $\varepsilon_2(\mathcal{O}) = 0$.

\newpage
\begin{conj}\label{genus3be}
Let $\Lambda$ be a lattice representing Genus $3$\textup{B}. Then for each $k\geq 0$ and $d\in\{1,2,p,2p\}$ we have the following Hecke-equivariant isomorphisms (for the algebra generated by the Hecke operators $T_q$ with $q\neq 2,p$ prime): 
\[\mathcal{S}(O(V),O(\widehat{\Lambda}),W_{2k}\otimes \theta_d)\cong S^{2\textup{-new}}_{2k+2}\left(\Gamma_0(4p),\varepsilon^{(4p)}_2\varepsilon^{(4p)}_d\right)\oplus S_{2k+2}^{8\textup{-new}}\left(\Gamma_0(8p),\varepsilon_2^{(8p)}\varepsilon_d^{(8p)}\right).\]
\end{conj}

\begin{example}
Let $\Lambda$ be an odd lattice representing Genus 3B with Gram matrix \[T = \begin{pmatrix}
        2 & 0 & -4\\
        0 & 25 & 20\\
        -4 & 20 & 28
    \end{pmatrix}.\] Here $\text{disc}(\Lambda) = 200 = 8\cdot 5^2$. At weight $W_8$ the space $\mathcal{S}(O(V),O(\widehat{\Lambda}),W_8)$ is $8$-dimensional. The characteristic polynomials of $T_3$ and $T_7$ factor as follows:
\[
\renewcommand{\arraystretch}{1.15}
\begin{array}{l|l|l|c}
F_{3,i}(x) & F_{7,i}(x) & \text{newform orbit} & \text{$(w_2,w_5)$}\\ \hline
x+156 & x+952 & \href{https://www.lmfdb.org/ModularForm/GL2/Q/holomorphic/2/10/a/a/}{2.10.a.a} & (-,*)\\
x-228 & x+6328 & \href{https://www.lmfdb.org/ModularForm/GL2/Q/holomorphic/4/10/a/a/}{4.10.a.a} & (-,*)\\
x-68 & x-10248 & \href{https://www.lmfdb.org/ModularForm/GL2/Q/holomorphic/8/10/a/b/}{8.10.a.b} & (-,*)\\ \hline
x-174 & x-4658 & \href{https://www.lmfdb.org/ModularForm/GL2/Q/holomorphic/10/10/a/c/}{10.10.a.c} & (-,+)\\
x^2+260x-3324 & x^2+380x-96250364 & \href{https://www.lmfdb.org/ModularForm/GL2/Q/holomorphic/20/10/a/b/}{20.10.a.b} & (-,+)\\
x^2-108x-23580 & x^2+908x-4271708 & \href{https://www.lmfdb.org/ModularForm/GL2/Q/holomorphic/40/10/a/c/}{40.10.a.c} & (-,+)
\end{array}
\]
The forms we see here are in accordance with Conjecture \ref{genus3be}, i.e.\ we see all newforms of level $2$, $4$, $8$ with $w_2=-1$,
together with all newforms of level $10$, $20$, $40$ with $(w_2,w_p)=(-,+)$. All other newforms of these levels have different Atkin-Lehner signs and so do not appear.
\end{example}

Genera $4$B and $5$B cannot benefit from an application of the map $\hat{\psi}_{(p)}$ due to having $p^2$-modular components in their local lattices at $p$. The best we can do at present is to view the lattice as an even one lying in $(V,Q^{(2)})$ of discriminant $64p^2$, with quaternion algebra $B = \text{Cl}^0(V)$ of discriminant $2$ and order $\mathcal{O} = \text{Cl}^0(\Lambda)$ of discriminant $\text{discrd}(\mathcal{O}) = 32p^2$. Both orders are residually ramified at $2$ (i.e.\ $\varepsilon_2(\mathcal{O}) = 0$) and are locally Eichler at $p$ with $\varepsilon_p(\mathcal{O}) = 1$ for Genus $4$B and locally non-Eichler at $p$ with $\varepsilon_p(\mathcal{O}) = -1$ for Genus $5$B. 

\begin{conj}\label{genus4be}
Let $\Lambda$ be a lattice representing Genus $4$\textup{B}. Then for each $k\geq 0$ and $d\in\{1,2,p,2p\}$ we have the following Hecke-equivariant isomorphisms (for the algebra generated by the Hecke operators $T_q$ with $q\neq 2,p$ prime): 
\[\mathcal{S}(O(V),O(\widehat{\Lambda}),W_{2k}\otimes \theta_d)\cong S^{2\textup{-new}}_{2k+2}\left(\Gamma_0(4p^2),\varepsilon^{(8)}_2\varepsilon^{(8)}_d\right)\oplus S^{8\textup{-new}}_{2k+2}\left(\Gamma_0(8p^2),\varepsilon^{(8p^2)}_2\varepsilon^{(8p^2)}_d\right).\]
\end{conj}

\begin{example}
Let $\Lambda$ be an odd lattice representing Genus $4B$ with Gram matrix \[T = \begin{pmatrix}
        1 & 0 & 0\\
        0 & 6 & 2\\
        0 & 2 & 66
    \end{pmatrix},\] so that $\text{disc}(\Lambda) = 392 = 8 \cdot 7^2$. At weight $W_8$ we compute $\dim  (\mathcal{S}(O(V),O(\widehat{\Lambda}),W_8) ) = 66$.
    The characteristic polynomials of $T_3$ and $T_5$ factor as follows: 
\[
\renewcommand{\arraystretch}{1.15}
\begin{array}{l|l|l|c}
F_{3,i}(x) & F_{5,i}(x) & \text{newform orbit} & (w_2,w_7)\\ \hline
(x+156)^2 & (x-870)^2 & \href{https://www.lmfdb.org/ModularForm/GL2/Q/holomorphic/2/10/a/a/}{2.10.a.a} & (-, *)\\
(x-228)^2 & (x+666)^2 & \href{https://www.lmfdb.org/ModularForm/GL2/Q/holomorphic/4/10/a/a/}{4.10.a.a} & (-, *)\\
(x-68)^2 & (x-1510)^2 & \href{https://www.lmfdb.org/ModularForm/GL2/Q/holomorphic/8/10/a/b/}{8.10.a.b} & (-, *)\\ \hline
x-170 & x-544 & \href{https://www.lmfdb.org/ModularForm/GL2/Q/holomorphic/14/10/a/b/}{14.10.a.b} & (-,+)\\
x^2+224x-5700 & x^2-1596x-840960 & \href{https://www.lmfdb.org/ModularForm/GL2/Q/holomorphic/28/10/a/a/}{28.10.a.a} & (-,-)\\
x^2+70x-9984 & x^2-1554x-3442720 & \href{https://www.lmfdb.org/ModularForm/GL2/Q/holomorphic/28/10/a/b/}{28.10.a.b} & (-,+)\\
\deg 3 & \deg 3 & \href{https://www.lmfdb.org/ModularForm/GL2/Q/holomorphic/56/10/a/b/}{56.10.a.b} & (-,-)\\
\deg 4 & \deg 4 & \href{https://www.lmfdb.org/ModularForm/GL2/Q/holomorphic/56/10/a/d/}{56.10.a.d} & (-,+)\\ \hline
\deg 3 & \deg 3 & \href{https://www.lmfdb.org/ModularForm/GL2/Q/holomorphic/98/10/a/j/}{98.10.a.j} & (-,+)\\
\deg 6 & \deg 6 & \href{https://www.lmfdb.org/ModularForm/GL2/Q/holomorphic/98/10/a/l/}{98.10.a.l} & (-,+)\\
\deg 6 & \deg 6 & \href{https://www.lmfdb.org/ModularForm/GL2/Q/holomorphic/196/10/a/f/}{196.10.a.f} & (-,+)\\
\deg 10 & \deg 10 & \href{https://www.lmfdb.org/ModularForm/GL2/Q/holomorphic/196/10/a/g/}{196.10.a.g} & (-,+)\\
\deg 9 & \deg 9 & \href{https://www.lmfdb.org/ModularForm/GL2/Q/holomorphic/392/10/a/l/}{392.10.a.l} & (-,+)\\
\deg 14 & \deg 14 & \href{https://www.lmfdb.org/ModularForm/GL2/Q/holomorphic/392/10/a/n/}{392.10.a.n} &(-,+)
\end{array}
\]

The forms we see here are in accordance with Conjecture \ref{genus4be}, i.e.\ we see all newforms of level $2$, $4$, $8$ with $w_2=-1$ twice, and all newforms of level $14$, $28$, $56$, $98$, $196$ and $392$ with $(w_2,w_7)=(-,*)$ once. All other newforms of these levels have different Atkin-Lehner signs and so do not appear.
\end{example}

\begin{conj}\label{genus5be}
Let $\Lambda$ be a lattice representing Genus $5$\textup{B}. Then for each $k\geq 0$ and $d\in\{1,2,p,2p\}$ there exists a Hecke-equivariant isomorphism (for the algebra generated by the Hecke operators $T_q$ with $q\neq 2,p$ prime) so that:
\begin{align*}\mathcal{S}(O(V),O(\widehat{\Lambda}),W_{2k}\otimes \theta_d) &\cong S_{2k+2}^{\textup{new}}\left(\Gamma_0(2),\varepsilon_2^{(2)}\varepsilon_d^{(2)}\right)&&\oplus S_{2k+2}^{4\textup{-new}}\left(\Gamma_0(8),\varepsilon_2^{(8)}\varepsilon_d^{(8)}\right)\\ &\oplus S_{2k+2}^{\textup{new}}\left(\Gamma_0(2p^2),\varepsilon_2^{(2p^2)}\varepsilon_d^{(2p^2)}\right)&&\oplus S_{2k+2}^{4p^2\textup{-new}}\left(\Gamma_0(8p^2),\varepsilon_2^{(8p^2)}\varepsilon_d^{(8p^2)}\right).\end{align*}
\end{conj}

\begin{example}
Let $\Lambda$ be an odd lattice representing Genus $5$B with Gram matrix \[T = \begin{pmatrix}
        1 & 0 & 0\\
        0 & 4 & 0\\
        0 & 0 & 98
    \end{pmatrix},\] where $\text{disc}(\Lambda) = 392 = 8\cdot 7^2$. For weight $W_8$ we find that $\dim(\mathcal{S}(O(V),O(\widehat{\Lambda}),W_8)) = 51$, and the characteristic polynomials of $T_3$, $T_5$ and $T_{11}$ factor as follows:
    \vspace{0.2in}
\[
\renewcommand{\arraystretch}{1.15}
\begin{array}{l|l|l|l|c}
F_{3,i}(x) & F_{5,i}(x) & F_{11,i}(x) & \text{newform orbit} & \text{$(w_2, w_7)$}\\ \hline
x+156 & x-870 & x+56148 & \href{https://www.lmfdb.org/ModularForm/GL2/Q/holomorphic/2/10/a/a/}{2.10.a.a} & (-, *)\\
x-228 & x+666 & x+30420 & \href{https://www.lmfdb.org/ModularForm/GL2/Q/holomorphic/4/10/a/a/}{4.10.a.a} & (-, *)\\
x-68 & x-1510 & x-3916 & \href{https://www.lmfdb.org/ModularForm/GL2/Q/holomorphic/8/10/a/b/}{8.10.a.b} & (-, *)\\ \hline
\deg 3 & \deg 3 & \deg 3 & \href{https://www.lmfdb.org/ModularForm/GL2/Q/holomorphic/98/10/a/j/}{98.10.a.j} & (-, +)\\
\deg 6 & \deg 6 & (\deg 3)^2 & \href{https://www.lmfdb.org/ModularForm/GL2/Q/holomorphic/98/10/a/l/}{98.10.a.l} & (-, +)\\
\deg 6 & \deg 6 & \deg 6 & \href{https://www.lmfdb.org/ModularForm/GL2/Q/holomorphic/196/10/a/f/}{196.10.a.f} & (-, +)\\
\deg 10 & \deg 10 & (\deg 5)^2 & \href{https://www.lmfdb.org/ModularForm/GL2/Q/holomorphic/196/10/a/g/}{196.10.a.g} & (-, +)\\
\deg 9 & \deg 9 & \deg 9 & \href{https://www.lmfdb.org/ModularForm/GL2/Q/holomorphic/392/10/a/l/}{392.10.a.l} & (-, +)\\
\deg 14 & \deg 14 & (\deg 7)^2 & \href{https://www.lmfdb.org/ModularForm/GL2/Q/holomorphic/392/10/a/n/}{392.10.a.n} &(-, +)
\end{array}
\]

\vspace{0.2in}
In accordance with Conjecture \ref{genus5be}, we see the newforms of level $2$, $4$ and $8$ with $w_2=-1$, each appearing once,
together with newforms of level $98$, $196$ and $392$ and Atkin-Lehner signs $(w_2,w_7) = (-1,+1)$, each also appearing once. In particular, the form \href{https://www.lmfdb.org/ModularForm/GL2/Q/holomorphic/8/10/a/a/}{8.10.a.a} (with $w_2=+1$) does not occur. Additionally, no newforms of level $7$, $14$, $28$, $56$ or $49$ occur. For example, \href{https://www.lmfdb.org/ModularForm/GL2/Q/holomorphic/14/10/a/b/}{14.10.a.b}, \href{https://www.lmfdb.org/ModularForm/GL2/Q/holomorphic/28/10/a/b/}{28.10.a.b} and \href{https://www.lmfdb.org/ModularForm/GL2/Q/holomorphic/56/10/a/d/}{56.10.a.d} (each having $(w_2,w_7)=(-,+)$) are not seen in this computation.
\end{example}

\appendix

\section{Summary table of results}\label{Summary table of results}
\noindent The following tables provide quick access to the results proved in this paper.\\

\noindent\textbf{Genera of half-discriminant $4$:}\\

\begin{center}
\begin{tabular}{c|c}
Label & $\Lambda_2$ \\ \hline
A & $\langle 2,2,2\rangle$  \\
B & $\langle 1,2,4\rangle $  \\
C& $\langle -1,-5,40\rangle $  \\
D& $\langle 1,1,8\rangle $ \\
E& $\langle 1,-5,-40\rangle $ \\
\end{tabular} \\
\end{center}
\vspace{0.1in}
\noindent\textbf{Modular forms arising from genera of half-discriminant $4$:}

\begin{center}
\vspace{0.1in}
\begin{tabular}{c|c|c}
Label & Space of modular forms & Result\\
\hline A & $S^{\textup{new}}_{2k+2}\left(\Gamma_0(2),\varepsilon^{(2)}_2\varepsilon^{(2)}_d\right)$ & Theorem \ref{genusA}\\
\hline B & $S_{2k+2}^{\text{new}}\left(\Gamma_0(2),\varepsilon_2^{(2)}\varepsilon_d^{(2)}\right) \oplus S_{2k+2}^{4\textup{-new}}\left(\Gamma_0(8),\varepsilon^{(2)}_2\varepsilon^{(2)}_d\right)$ & Conjecture \ref{conjBE}\\
\hline C & $S^{2\textup{-new}}_{2k+2}\left(\Gamma_0(4),\varepsilon^{(2)}_2\varepsilon^{(2)}_d\right)$ & Conjecture \ref{genusCD}\\
\hline D & $S^{2\textup{-new}}_{2k+2}\left(\Gamma_0(4),\varepsilon^{(2)}_2\varepsilon^{(2)}_d\right)$ & Conjecture \ref{genusCD}\\
\hline E & $S_{2k+2}^{\text{new}}\left(\Gamma_0(2),\varepsilon_2^{(2)}\varepsilon_d^{(2)}\right) \oplus S_{2k+2}^{4\textup{-new}}\left(\Gamma_0(8),\varepsilon^{(2)}_2\varepsilon^{(2)}_d\right)$ & Conjecture \ref{conjBE}\\
\end{tabular}
\end{center}

\vspace{0.2in}
\noindent\textbf{Genera of half-discriminant $p^2$:}

\begin{center}
\begin{tabular}{c|c|c|c|c|c}
\multirow{2}{*}{Label} & \multirow{2}{*}{$\Lambda_2$} & \multicolumn{4}{c}{$\Lambda_p$}\\ \cline{3-6} & & $p\equiv 1 \bmod 8$ & $p\equiv 3 \bmod 8$ &  $p\equiv 5 \bmod 8$ & $p\equiv 7 \bmod 8$\\ \hline $1$ & $\left\langle\begin{pmatrix}0 & 1\\ 1 & 0\end{pmatrix},-2\right\rangle$ & \multirow{2}{*}{$\langle r,p,pr\rangle$} & \multirow{2}{*}{$\langle 2,p,p\rangle$} & \multirow{2}{*}{$\langle 1,p,2p\rangle$} & \multirow{2}{*}{$\langle 1,p,p\rangle$}\\ \cline{1-2} $2$ & $\langle 1,-1,-2\rangle$ & & & &\\ \hline $3$ & \multirow{3}{*}{$\langle 1,1,2\rangle$} & $\langle 1,p,p\rangle$ & $\langle 1,p,2p\rangle$ & $\langle 2,p,p\rangle$ & $\langle -1,p,-p\rangle$\\ \cline{3-6} $4$ & & $\langle 1,1,p^2\rangle$ & $\langle 1,2,p^2\rangle$ & $\langle 1,1,2p^2\rangle$ & $\langle 1,-1,-p^2\rangle$\\ \cline{3-6} $5$ & & $\langle 1,r,p^2r\rangle$ & $\langle 1,1,2p^2\rangle$ & $\langle 1,2,p^2\rangle$ & $\langle 1,1,p^2\rangle$
\end{tabular}
\end{center}
\vspace{0.1in}
\noindent\textbf{Modular forms arising from half-genera of discriminant $p^2$:}

\begin{center}
\vspace{0.1in}
\begin{tabular}{c|c|c}
Label & Space of modular forms & Result\\
\hline $1$ & $S^{\textup{new}}_{2k+2}\left(\Gamma_0(p),\varepsilon^{(p)}_p\varepsilon^{(p)}_d\right)$ & Theorem \ref{genus1}\\
\hline $2$ & $S^{p\textup{-new}}_{2k+2}\left(\Gamma_0(2p),\varepsilon^{(2p)}_p\varepsilon^{(2p)}_d\right)$ & Theorem \ref{genus2}\\
\hline $3$ & $S^{2\textup{-new}}_{2k+2}\left(\Gamma_0(4p),\varepsilon^{(4p)}_2\varepsilon^{(4p)}_d\right)$ & Conjecture \ref{genus3}\\
\hline $4$ & $S^{2\textup{-new}}_{2k+2}\left(\Gamma_0(4p^2),\varepsilon_2^{(4p^2)}\varepsilon_d^{(4p^2)}\right)$ & Conjecture \ref{genus4}\\
\hline $5$ & $S^{2\textup{-new}}_{2k+2}\left(\Gamma_0(4),\varepsilon^{(4)}_2\varepsilon^{(4)}_d\right)\oplus S^{2p^2\textup{-new}}_{2k+2}\left(\Gamma_0(4p^2),\varepsilon^{(4p^2)}_2\varepsilon^{(4p^2)}_d\right)$ & Conjecture \ref{genus5}\\
\end{tabular}
\end{center}

\vspace{0.1in}
\newpage
\noindent\textbf{Genera of half-discriminant $4p^2$:}

\begin{center}
\begin{tabular}{c|c|c|c|c|c}
\multirow{2}{*}{Label} & \multirow{2}{*}{$\Lambda_2$} & \multicolumn{4}{c}{$\Lambda_p$}\\ 
\cline{3-6} & & $p\equiv 1 \bmod 8$ & $p\equiv 3 \bmod 8$ &  $p\equiv 5 \bmod 8$ & $p\equiv 7 \bmod 8$\\ \hline $1$A & $\left\langle\begin{pmatrix}0 & 1\\ 1 & 0\end{pmatrix},-8\right\rangle$ & \multirow{7}{*}{$\langle r,p,pr\rangle$} & \multirow{7}{*}{$\langle 2,p,p\rangle$} & \multirow{7}{*}{$\langle 1,p,2p\rangle$} & \multirow{7}{*}{$\langle 1,p,p\rangle$}\\ 
$1$B & $\left\langle\begin{pmatrix}2 & 1\\ 1 & 2\end{pmatrix},-40\right\rangle$ &  &  &  & \\ 
\cline{1-2} $2$A & $\langle 2,-2,-2\rangle$ &  &  &  & \\ 
$2$B & $\langle 1,-2,-4\rangle$ & & & & \\ 
$2$C & $\langle-1,-1,8\rangle$ & & & & \\ 
$2$D & $\langle 1,5,40 \rangle$ & & & & \\ 
$2$E & $\langle 1,-1,-8 \rangle$ & & & & \\ 
\hline $3$A & $\langle 2,2,2\rangle$ & \multirow{5}{*}{$\langle 1,p,p\rangle$} & \multirow{5}{*}{$\langle 1,p,2p\rangle$} & \multirow{5}{*}{$\langle 2,p,p\rangle$} & \multirow{5}{*}{$\langle -1,p,-p\rangle$}\\ 
$3$B & $\langle 1,2,4\rangle$ & & & & \\ 
$3$C & $\langle-1,-5,40\rangle$ & & & & \\ 
$3$D & $\langle 1,1,8 \rangle$ & & & & \\ 
$3$E & $\langle 1,-5,-40 \rangle$ & & & & \\ 
\hline $4$A & $\langle 2,2,2\rangle$ & \multirow{5}{*}{$\langle 1,1,p^2\rangle$} & \multirow{5}{*}{$\langle 1,2,p^2\rangle$} & \multirow{5}{*}{$\langle 1,1,2p^2\rangle$} & \multirow{5}{*}{$\langle 1,-1,-p^2\rangle$}\\ 
$4$B & $\langle 1,2,4\rangle$ & & & & \\ 
$4$C & $\langle-1,-5,40\rangle$ & & & & \\ 
$4$D & $\langle 1,1,8 \rangle$ & & & & \\ 
$4$E & $\langle 1,-5,-40 \rangle$ & & & & \\ 
\hline $5$A & $\langle 2,2,2\rangle$ & \multirow{5}{*}{$\langle 1,r,p^2r\rangle$} & \multirow{5}{*}{$\langle 1,1,2p^2\rangle$} & \multirow{5}{*}{$\langle 1,2,p^2\rangle$} & \multirow{5}{*}{$\langle 1,1,p^2\rangle$}\\ 
$5$B & $\langle 1,2,4\rangle$ & & & & \\ 
$5$C & $\langle-1,-5,40\rangle$ & & & & \\ 
$5$D & $\langle 1,1,8 \rangle$ & & & & \\ 
$5$E & $\langle 1,-5,-40 \rangle$ & & & & \\ 
\end{tabular}
\end{center}

\vspace{0.1in}
\newpage
\noindent\textbf{Modular forms arising from genera of half-discriminant $4p^2$:}

\begin{center}
\begin{tabular}{c|c|c}
Label & Space of modular forms & Result\\
\hline $1$A & $S^{p\textup{-new}}_{2k+2}\left(\Gamma_0(4p),\varepsilon^{(4p)}_p\varepsilon^{(4p)}_d\right)$ & Theorem \ref{genus1a}\\
 $1$B & $S^{\textup{new}}_{2k+2}\left(\Gamma_0(p),\varepsilon^{(p)}_p\varepsilon^{(p)}_d\right)$ & Theorem \ref{genus1b2acd}\\
\hline $2$A & $S^{p\textup{-new}}_{2k+2}\left(\Gamma_0(2p),\varepsilon^{(2p)}_p\varepsilon^{(2p)}_d\right)$ & Theorem \ref{genus1b2acd}\\
 $2$B & $S^{p\textup{-new}}_{2k+2}\left(\Gamma_0(2p),\,\varepsilon_p^{(2p)}\varepsilon_d^{(2p)}\right)\oplus S_{2k+2}^{\text{new}}\left(\Gamma_0(8p),\varepsilon_p^{(8p)}\varepsilon_d^{(8p)}\right)$ &  Conjecture \ref{genus2be}\\
 $2$C & $S^{p\textup{-new}}_{2k+2}\left(\Gamma_0(2p),\varepsilon^{(2p)}_p\varepsilon^{(2p)}_d\right)$ & Theorem \ref{genus1b2acd}\\
 $2$D & $S^{p\textup{-new}}_{2k+2}\left(\Gamma_0(2p),\varepsilon^{(2p)}_p\varepsilon^{(2p)}_d\right)$ & Theorem \ref{genus1b2acd}\\
 $2$E & $S^{p\textup{-new}}_{2k+2}\left(\Gamma_0(2p),\,\varepsilon_p^{(2p)}\varepsilon_d^{(2p)}\right)\oplus S_{2k+2}^{\text{new}}\left(\Gamma_0(8p),\varepsilon_p^{(8p)}\varepsilon_d^{(8p)}\right)$ & Conjecture \ref{genus2be}\\
\hline $3$A & $S^{2\textup{-new}}_{2k+2}\left(\Gamma_0(2p),\varepsilon^{(2p)}_2\varepsilon^{(2p)}_d\right)$ & Theorem \ref{3a4a5a}\\
 $3$B & $S^{2\textup{-new}}_{2k+2}\left(\Gamma_0(4p),\varepsilon^{(4p)}_2\varepsilon^{(4p)}_d\right)\oplus S_{2k+2}^{8\textup{-new}}\left(\Gamma_0(8p),\varepsilon_2^{(8p)}\varepsilon_d^{(8p)}\right)$ & Conjecture \ref{genus3be}\\
 $3$C & $S^{2\textup{-new}}_{2k+2}\left(\Gamma_0(4p),\varepsilon^{(4p)}_2\varepsilon^{(4p)}_d\right)$ & Theorem \ref{genus1b2acd} (Conjecture \ref{genus3})\\
 $3$D & $S^{2\textup{-new}}_{2k+2}\left(\Gamma_0(4p),\varepsilon^{(4p)}_2\varepsilon^{(4p)}_d\right)$ & Theorem \ref{genus1b2acd} (Conjecture \ref{genus3})\\
 $3$E & $S^{2\textup{-new}}_{2k+2}\left(\Gamma_0(4p),\varepsilon^{(4p)}_2\varepsilon^{(4p)}_d\right)\oplus S_{2k+2}^{8\textup{-new}}\left(\Gamma_0(8p),\varepsilon_2^{(8p)}\varepsilon_d^{(8p)}\right)$ & Conjecture \ref{genus3be}\\
\hline $4$A & $S^{2\textup{-new}}_{2k+2}\left(\Gamma_0(2p^2),\varepsilon^{(2p^2)}_2\varepsilon^{(2p^2)}_d\right)$ & Theorem \ref{3a4a5a}\\
 $4$B & $S^{2\textup{-new}}_{2k+2}\left(\Gamma_0(4p^2),\varepsilon^{(8)}_2\varepsilon^{(8)}_d\right)\oplus S^{8\textup{-new}}_{2k+2}\left(\Gamma_0(8p^2),\varepsilon^{(8p^2)}_2\varepsilon^{(8p^2)}_d\right)$ & Conjecture \ref{genus4be}\\
 $4$C & $S^{2\textup{-new}}_{2k+2}\left(\Gamma_0(4p^2),\varepsilon_2^{(4p^2)}\varepsilon_d^{(4p^2)}\right)$ & Theorem \ref{genus1b2acd} (Conjecture \ref{genus4})\\
 $4$D & $S^{2\textup{-new}}_{2k+2}\left(\Gamma_0(4p^2),\varepsilon_2^{(4p^2)}\varepsilon_d^{(4p^2)}\right)$ & Theorem \ref{genus1b2acd} (Conjecture \ref{genus4})\\
 $4$E & $S^{2\textup{-new}}_{2k+2}\left(\Gamma_0(4p^2),\varepsilon^{(8)}_2\varepsilon^{(8)}_d\right)\oplus S^{8\textup{-new}}_{2k+2}\left(\Gamma_0(8p^2),\varepsilon^{(8p^2)}_2\varepsilon^{(8p^2)}_d\right)$ &Conjecture \ref{genus4be}\\
\hline $5$A & $S^{\textup{new}}_{2k+2}\left(\Gamma_0(2),\varepsilon^{(2)}_2\varepsilon^{(2)}_d\right)\oplus S^{\textup{new}}_{2k+2}\left(\Gamma_0(2p^2),\varepsilon^{(2p^2)}_2\varepsilon^{(2p^2)}_d\right)$ & Theorem \ref{3a4a5a}\\
&&\\
 \multirow{2}{*}{$5$B} & $S_{2k+2}^{\textup{new}}\left(\Gamma_0(2),\varepsilon_2^{(2)}\varepsilon_d^{(2)}\right)\oplus S_{2k+2}^{4\textup{-new}}\left(\Gamma_0(8),\varepsilon_2^{(8)}\varepsilon_d^{(8)}\right)$ & \multirow{2}{*}{Conjecture \ref{genus5be}}\\ & $\oplus S_{2k+2}^{\textup{new}}\left(\Gamma_0(2p^2),\varepsilon_2^{(2p^2)}\varepsilon_d^{(2p^2)}\right)\oplus S_{2k+2}^{4p^2\textup{-new}}\left(\Gamma_0(8p^2),\varepsilon_2^{(8p^2)}\varepsilon_d^{(8p^2)}\right)$ & \\
 && \\
 $5$C & $S^{2\textup{-new}}_{2k+2}\left(\Gamma_0(4),\varepsilon^{(4)}_2\varepsilon^{(4)}_d\right)\oplus S^{2p^2\textup{-new}}_{2k+2}\left(\Gamma_0(4p^2),\varepsilon^{(4p^2)}_2\varepsilon^{(4p^2)}_d\right)$ & Theorem \ref{genus1b2acd} (Conjecture \ref{genus5})\\
 $5$D & $S^{2\textup{-new}}_{2k+2}\left(\Gamma_0(4),\varepsilon^{(4)}_2\varepsilon^{(4)}_d\right)\oplus S^{2p^2\textup{-new}}_{2k+2}\left(\Gamma_0(4p^2),\varepsilon^{(4p^2)}_2\varepsilon^{(4p^2)}_d\right)$ & Theorem \ref{genus1b2acd} (Conjecture \ref{genus5})\\
 &&\\
   \multirow{2}{*}{$5$E} & $S_{2k+2}^{\textup{new}}\left(\Gamma_0(2),\varepsilon_2^{(2)}\varepsilon_d^{(2)}\right)\oplus S_{2k+2}^{4\textup{-new}}\left(\Gamma_0(8),\varepsilon_2^{(8)}\varepsilon_d^{(8)}\right)$ & \multirow{2}{*}{Conjecture \ref{genus5be}}\\ & $\oplus S_{2k+2}^{\textup{new}}\left(\Gamma_0(2p^2),\varepsilon_2^{(2p^2)}\varepsilon_d^{(2p^2)}\right)\oplus S_{2k+2}^{4p^2\textup{-new}}\left(\Gamma_0(8p^2),\varepsilon_2^{(8p^2)}\varepsilon_d^{(8p^2)}\right)$ & \\
\end{tabular}
\end{center}



\end{document}